\documentclass[10pt,a4paper]{article}
\usepackage[utf8]{inputenc}
\usepackage{amsmath}
\usepackage{amsfonts}
\usepackage{amssymb}
\usepackage{subfigure}
\usepackage{graphicx}
\usepackage{amsthm}
\usepackage{bbm}
\usepackage[ruled,lined,linesnumbered]{algorithm2e}

\newtheorem{theorem}{Theorem}
\newtheorem{lemma}{Lemma}

\title{Stable and Discriminative Topological Graph Analysis}
\author{Padraig Corcoran \\ School of Computer Science \& Informatics, Cardiff University}

\begin{document}
\maketitle

\begin{abstract}
We propose a novel method for topological analysis of unweighted graphs which is based on \textit{persistent homology}. The proposed method maps the input graph to a complete weighted graph where the weighting function maps each edge to a value indicating the degree to which it belongs to a clique. The persistent homology of this weighted graph is subsequently computed to give a topological representation describing the topological features of the input graph plus their significance.

A formal and experimental analysis of the proposed and existing methods for topological graph analysis is presented. Through this analysis, we demonstrate that the proposed method possesses the properties of being stable and performing accurate discrimination such that it can make accurate inferences regarding the topological features of a given graph. On the other hand, we find that the existing methods considered do not possess these properties making it difficult for them to make such inferences. These findings are experimentally demonstrated using a number of random and real world graphs.
\end{abstract}

\section{Introduction}
A graph or network is a fundamental abstract data type consisting of a set of vertices and a set of edges where an edge corresponds to a pair of vertices. Much of the data in our world is naturally modelled as a graph. For example, a transportation network is commonly modelled as a graph where locations are modelled as vertices and transportation links are modelled as edges \cite{gagarin2018multiple}. Similarly, a social network is commonly modelled as a graph where individuals are modelled as vertices and social relationships are modelled as edges \cite{carstens2013persistent}. Finally, the brain is commonly modelled as a graph where neurons or brain region are modelled as vertices and connections are modelled as edges \cite{giusti2016two, chung2019exact}. Given data modelling as a graph, in many cases one wishes to perform an analysis of this graph to infer useful information. For example, given a transportation network modelled as a graph one may wish to perform analysis to infer the process governing its growth \cite{corcoran2013analysing}. Similarly, given a social network modelled as a graph one may wish to perform analysis to infer social roles.

There exist a number of categories of methods for performing analysis of graphs. Spectral graph analysis methods perform analysis by considering matrix representations such as the adjacency matrix or Laplacian matrix \cite{chung1997spectral}. Statistical graph analysis methods perform analysis by considering statistical representations such as the degree distribution \cite{kolaczyk2009statistical}. Machine learning graph analysis methods perform analysis by learning useful low dimensional vector space representations \cite{kipf2016variational}. Finally, topological graph analysis methods perform analysis by constructing a multi-scale representation of the graph known as a \textit{filtration} and computing the \textit{persistent homology} of this representation \cite{giusti2016two}. Persistent homology is a topological representation which encodes information concerning the existence and scale of connected components and holes of various dimensions \cite{zomorodian2005computing}. In this work we focus on topological graph analysis methods and specifically the case where the graphs in question are undirected and unweighted. 


It is important that any method for performing topological analysis of a given data type exhibit the properties of being stable and performing accurate discrimination. Being stable concerns ensuring that perturbations to the input data do not significantly change the corresponding topological representation. Performing accurate discrimination concerns ensuring that data with similar topological characteristics have similar topological representations and data with distinct topological characteristics have distinct topological representations. For the case where the data type is sets of points in Euclidean space, many methods have been proposed which exhibit both these properties \cite{bubenik2015statistical, chevyrev2018persistence}. However, the development of topological graph analysis methods which exhibit both these properties is less well studied.

The two most commonly used topological graph analysis methods involve constructing a \textit{clique complex filtration} and a \textit{power complex filtration} which are two distinct types of graph filtration. However these methods do not exhibit the properties of being stable and performing accurate discrimination. To demonstrate this consider the three graphs in Figures \ref{fig:stable_discrimination_a}, \ref{fig:stable_discrimination_b} and \ref{fig:stable_discrimination_c} which contain a single clique, two cliques and two cliques connected by a single edge respectively. Note that, the latter two graphs differ only by a single edge. It is evident that the first graph is topologically distinct from the latter two graphs. Furthermore, the latter two graphs are topologically similar. However, the two topological graph analysis methods mentioned above will determine the graphs in Figures \ref{fig:stable_discrimination_a} and \ref{fig:stable_discrimination_c} to be topologically similar and topologically distinct from the graph in Figure \ref{fig:stable_discrimination_b}. This is a consequence of the fact that these methods interpret the graphs in Figures \ref{fig:stable_discrimination_a} and \ref{fig:stable_discrimination_c} as containing one large scale connected component and interpret the graph in Figure \ref{fig:stable_discrimination_b} as containing two large scale connected components. This demonstrates that these methods do not perform accurate discrimination. Furthermore, since the graphs in Figures \ref{fig:stable_discrimination_b} and \ref{fig:stable_discrimination_c} differ only by a single edge but are determined to be topologically distinct this demonstrates that the methods do not exhibit the property of being stable. A formal proof of these facts will be presented later in this article.

\begin{figure}
\begin{center}
\subfigure[]{\includegraphics[height=1.5cm]{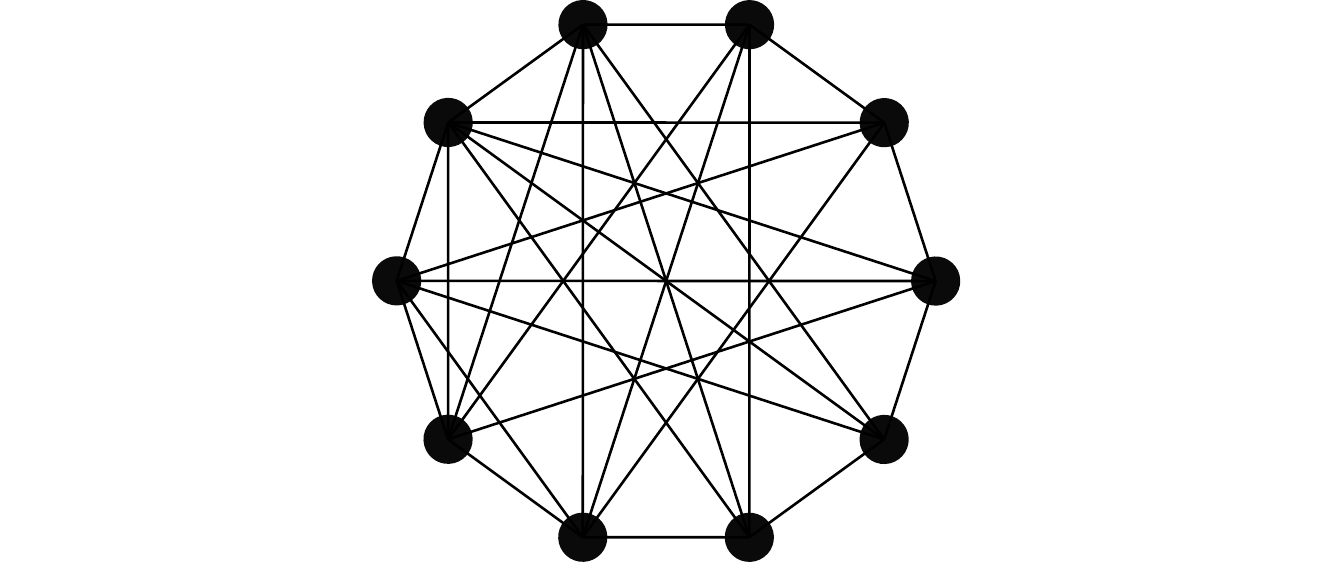}
\label{fig:stable_discrimination_a}}
\hspace{.1cm}
\subfigure[]{\includegraphics[height=1.5cm]{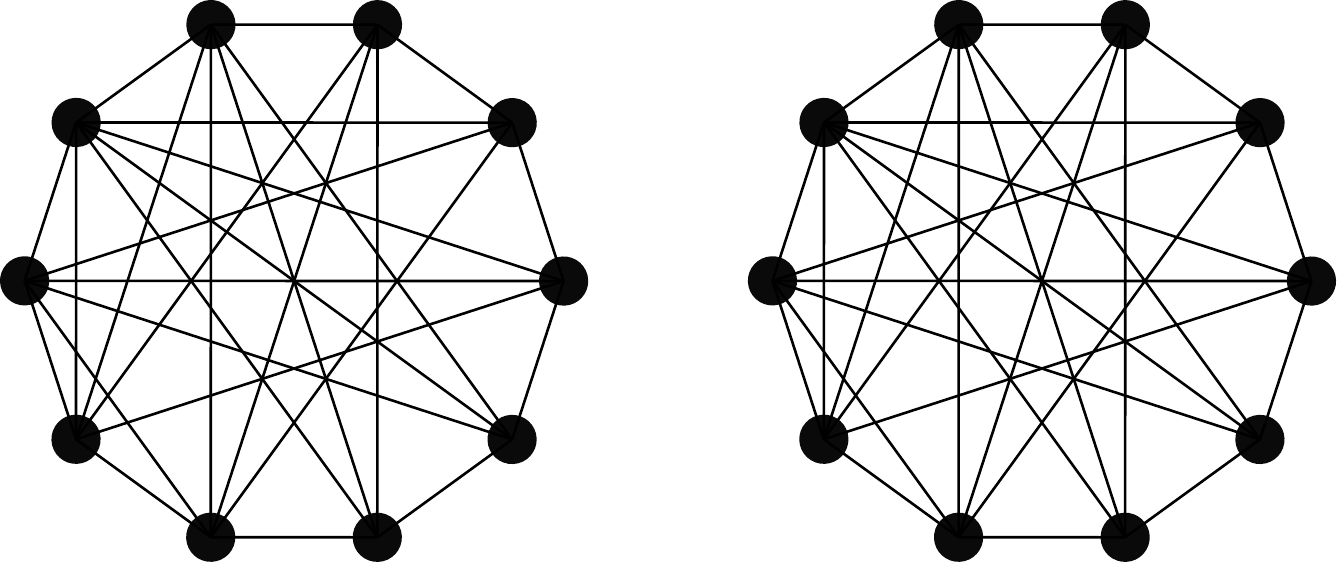}
\label{fig:stable_discrimination_b}}
\hspace{.1cm}
\subfigure[]{\includegraphics[height=1.5cm]{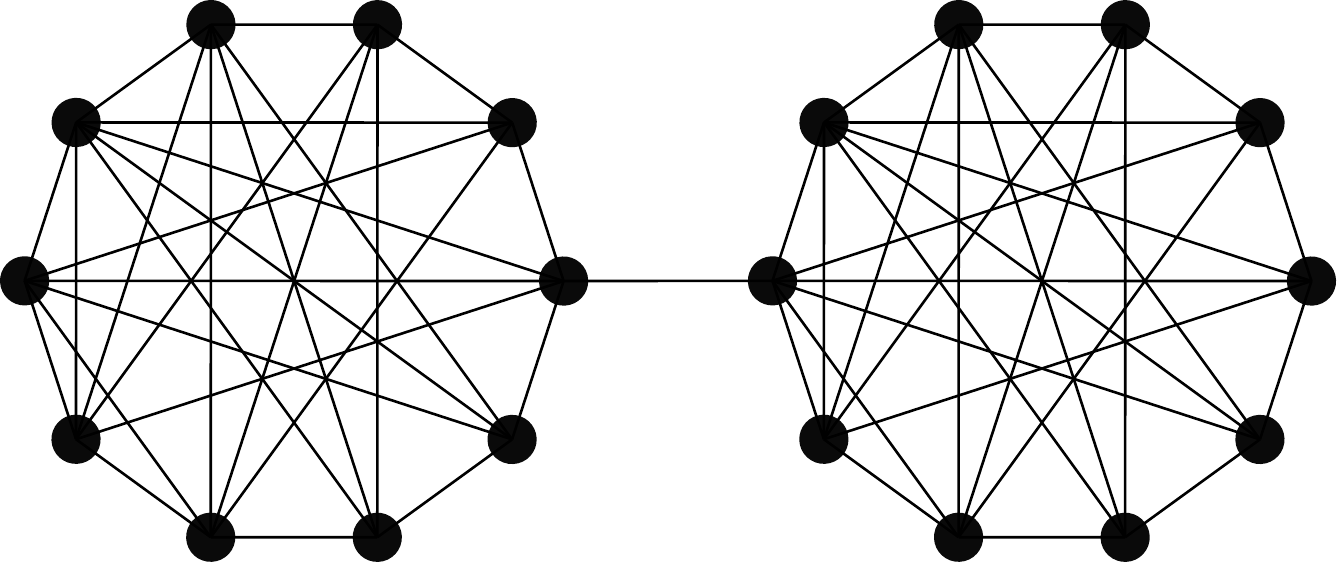}
\label{fig:stable_discrimination_c}}
\caption{The three graphs in (a), (b) and (c) contain a single clique, two cliques and two cliques connected by a single edge respectively. Note that, the latter two graphs differ only by a single edge.}
\label{fig:stable_discrimination}
\end{center}
\end{figure}

In this work we propose a novel method for performing graph topological analysis which exhibits the properties of being stable and performing accurate discrimination. This is achieved by constructing a novel filtration known as a \textit{cliqueness complex filtration} and subsequently computing the persistent homology of this filtration. This filtration considers the degree to which each current and potential edge in a given graph belongs to a clique. To demonstrate this consider again the three graphs in Figure \ref{fig:stable_discrimination}. The proposed topological graph analysis method will determine the graphs in Figures \ref{fig:stable_discrimination_b} and \ref{fig:stable_discrimination_c} to be topologically similar and topologically distinct from the graph in Figure \ref{fig:stable_discrimination_a}.  Furthermore, since the graphs in Figures \ref{fig:stable_discrimination_b} and \ref{fig:stable_discrimination_c} differ only by a single edge and are determined to be topologically similar this demonstrates that the method exhibits the property of being stable. Again, a formal proof of these facts will be presented later in this article.

The remainder of this paper is structured as follows. Section \ref{sec:ph_graphs} reviews background material on topological graph analysis. Section \ref{sec:stable_graph_analysis} presents the proposed method for topological graph analysis. Section \ref{sec:stability} presents a formal stability analysis of the proposed and existing methods for topological graph analysis. Section \ref{sec:applications} presents an experimental evaluation of these methods with respect to the task of performing topological graph analysis. Finally, section \ref{sec:conclusions} draws some conclusions from this work.

\section{Background and Related Works}
\label{sec:ph_graphs}
In this section we briefly present background material on graph theory and topology. For a reader new to these topics more detailed introductions can be found in \cite{giblin2013graphs} and \cite{edelsbrunner2010computational} respectively. We subsequently review related works on topological graph analysis.

An unweighted graph is a tuple $(V,E)$ where $V$ is a set of vertices and $E \subseteq (V \times V)$ is a set of edges. If the tuples in $E$ are unordered the graph in question is called undirected otherwise it is called directed. In this work we assume all graphs are undirected. A weighted graph is a tuple $(V,E,f)$ where $V$ is a set of vertices, $E \subseteq (V \times V)$ is a set of edges and $f:E \rightarrow \mathbb{R}$ is a weighting function. There also exists a class of weighted graphs where the domain of the weighting function is the set of graph vertices. However, in this work we assume the domain is the set of graph edges. A clique is a subset of vertices of a graph such that every two distinct vertices in the clique are adjacent. The $k$-th power of a graph $G$ is denoted $G^k$ and equals a graph with the same set of vertices as $G$ and an edge between two vertices if and only if there exists a path of length at most $k$ between them. In the case of an unweighted graph the length of a path equals the number of edges in that path. The diameter of a graph $G$ is the greatest distance between any pair of vertices in $G$.

An (abstract) simplicial complex is a higher dimensional generalization of a graph. Formally, a simplicial complex $K$ is a finite collection of sets such that for each $\sigma \in K$ all subsets of $\sigma$ are also contained in $K$. Each element $\sigma \in K$ is called a $p$-simplex where $p = \left\vert{\sigma}\right\vert - 1$ is the corresponding dimension of the simplex. The faces of a simplex $\sigma$ correspond to all simplices $\tau$ where $\tau \subset \sigma$. The dimension of a simplicial complex $K$ is the largest dimension of any simplex $\sigma \in K$. The $n$-skeleton of a simplicial complex $K$ is the simplicial complex containing the union of the simplices of $K$ of dimensions less than or equal to $n$. Given an unweighted and undirected graph $G$ there exist a number of methods for constructing a corresponding simplicial complex representation \cite{aktas2019persistence}. The \textit{clique complex} of $G$ is a simplicial complex where each clique of size $p$ in $G$ corresponds to a $p-1$-simplex. The power complex of $G$ is the clique complex of $G^p$ for a specified $p$ \cite{parks2016persistent}.

Given a simplicial complex $K$, the formal sum $c$ defined by Equation \ref{eq:p_chain} is called a $p$-chain where each $\sigma_i \in K$ is a $p$-simplex and each $\lambda_i$ is an element of a given field.

\begin{equation}
\label{eq:p_chain}
c = \sum \lambda_i \sigma_i
\end{equation}

The vector space of all $p$-chains is denoted $C_p(K)$. The boundary map $\partial_p : C_p(K) \rightarrow C_{p-1}(K)$ is the homomorphism defined by Equation \ref{eq:boundary_operator} where $\hat{v_i}$ indicates the deletion of $v_i$ from the sequence. The boundary of a $p$-simplex $\sigma=[v_0, \dots, v_{p}]$ is equal to the sum of its $(p-1)$-dimensional faces.

\begin{equation}
\label{eq:boundary_operator}
\partial_p \sigma = \sum_{i=0}^{p} (-1)^i [v_1, \dots, \hat{v_i}, \dots, v_p]
\end{equation}

The kernel of $\partial_p$ is called the vector space of $p$-cycles and is denoted $Z_p(K)$. The image of $\partial_p$ is called the vector space of $p$-boundaries and is denoted $B_p(K)$. The fact that $\partial_{p+1}\partial_{p}=0$ implies that $B_p(K) \subseteq Z_k(K)$. The quotient space $H_p(K) = Z_p(K) / B_p(K)$ is called the $p$-homology group of $K$. Intuitively an element of the $p$-homology group corresponds to a $p$-dimensional hole in $K$. That is, an element of the $0$-homology group corresponds to a connected component in $K$ while an element of the $1$-homology group corresponds to a one dimensional hole in $K$. The rank of $H_p(K)$ is called the $p$-th Betti number.

A filtration of a simplicial complex $K$ is a sequence of simplicial complexes $\left( K_0, K_1 \dots, K_{m} \right)$ which satisfy the conditions in Equation \ref{eq:filtration}. Given a filtration of a simplicial complex $K$, for every $i \leq j$ there exists an inclusion map from $K_i$ to $K_j$ and in turn an induced homomorphism from $H_p(K_i)$ to $H_p(K_j)$ for each dimension $p$. An element of the $p$-homology group is \textit{born} at $K_{i+1}$ if it exists in $H_p(K_{i+1})$ but does not exist in $H_p(K_{i})$. An element of the $p$-homology group \textit{dies} at $K_{i+1}$ if it exists in $H_p(K_i)$ but does not exist in $H_p(K_{i+1})$. If an element of a $p$-homology group never dies, its death is determined to be at a hypothetical simplicial complex $K_{\infty}$.

\begin{equation}
\label{eq:filtration}
\emptyset = K_0 \subset K_1 \subset \dots \subset K_{m} = K
\end{equation}

An element of the $p$-homology group which is born at $K_i$ and dies at $K_j$ can be represented as a point $(i,j)$ in the space $\lbrace (i,j) \in \mathbb{R}^2, i \leq j \rbrace$ with corresponding \textit{persistence} of $j-i$. For a given filtration $\left( K_0, K_1 \dots, K_{m} \right)$ of a simplicial complex $K$, the multiset of points corresponding to a $p$-homology group is called a $p$-dimensional \textit{persistence diagram} \cite{zomorodian2005computing}. Let Pers$_p$ denote the space of $p$-dimensional persistence diagrams. The bottleneck distance $B: \text{Pers$_p$} \times \text{Pers$_p$} \rightarrow \mathbb{R}$ is defined as follows. First, to each persistence diagram we add infinitely many points with equal coordinates where each has infinite multiplicity. The bottleneck distance is then defined by Equation \ref{eq:bottleneck_distance} where the map $\eta$ is a bijection and $\| . \|_{\infty}$ is the $L_{\infty}$-norm \cite{edelsbrunner2010computational}.

\begin{equation}
\label{eq:bottleneck_distance}
B(X,Y) = \inf_{\eta: X \rightarrow Y} \sup_{x \in X} \| x - \eta(x) \|_{\infty}
\end{equation}

The set of persistence diagrams corresponding to given filtration encodes topological information about that filtration. Therefore in order for persistent homology to extract useful topological information from a given graph it is necessary to first construct a useful filtration of that graph. There exists a number of methods for constructing a filtration for a weighted graph. A review of these methods can be found in the excellent review article by Aktas et al. \cite{aktas2019persistence}. On the other hand, constructing a filtration for an unweighted graph is less well studied. We now briefly review existing methods for performing this task.

The \textit{clique complex filtration} for an unweighted graph $G$ is defined in Equation \ref{eq:clique_filtration} where $K_i$ denotes the $i$-skeleton of the clique complex $K$ of $G$ and $m$ is the dimension of $K$ \cite{horak2009persistent}. The \textit{power complex filtration} for an unweighted graph $G$ is defined in Equation \ref{eq:power_filtration} where $K_i$ denotes the clique complex of the graph $G^i$ and $d$ is the diameter of $G$. Li et al. \cite{li2020hybrid} refer to this filtration as the \textit{network filtration}. Note that, the power complex filtration of a graph is equivalent to the \textit{Vietoris-Rips filtration} of that graph where the distance between two vertices equals the length of the shortest path between the vertices in question \cite{bauer2019ripser}. This is a very useful equivalence which allows the power complex filtration to be computed using the very efficient Ripser software \cite{bauer2019ripser}. Suh et al. \cite{suh2019persistent} defined a filtration for unweighted graphs which first maps the graph in question to an edge weighted graph and subsequently constructs a filtration for this graph. The edge weights in question are computed using the Jaccard index with respect to the neighbourhoods of edge vertices. Kannan et al. \cite{kannan2019persistent} defined a filtration for unweighted graphs similar to that proposed by Suh et al. \cite{suh2019persistent} but uses a mapping to a vertex weighted graph. The vertex weights in question are computed using vertex degree.

In the introduction to this paper we provided an illustrative example of the fact that the above methods for performing topological graph analysis which involve constructing a clique or power complex filtration do not exhibit the properties of being stable and performing accurate discrimination. In section \ref{sec:stability} of this paper we formally prove this fact.

\begin{equation}
\label{eq:clique_filtration}
\emptyset \subset K_0 \subset \dots \subset K_{m}
\end{equation}

\begin{equation}
\label{eq:power_filtration}
\emptyset \subset K_1 \subset \dots \subset K_{d}
\end{equation}


\section{Topological Graph Analysis}
\label{sec:stable_graph_analysis}
This section describes the proposed topological graph analysis method. The method in question is described in subsection \ref{sec:stable_graph_analysis:method}. Implementation details of the method are subsequently described in subsection \ref{sec:stable_graph_analysis:implementation}.

\subsection{Proposed Method}
\label{sec:stable_graph_analysis:method}
The proposed topological graph analysis method contains three steps. In the first step the unweighted input graph is mapped to an edge weighted graph. In the second step a filtration of this graph is constructed. In the final step the persistent homology of this filtration is computed. We now describe each of these steps in detail.

Let $\mathcal{G}^U$ and $\mathcal{G}^W$ denote the spaces of unweighted and weighted graphs respectively. Let $G = (V, E) \in \mathcal{G}^U$ denote the input graph to which we wish to apply topological graph analysis. In the first step of the proposed method we compute a map $W: \mathcal{G}^U \rightarrow \mathcal{G}^W$ of the input graph $G \in \mathcal{G}^U_n$ to a new weighted graph $G' = (V', E', C) \in \mathcal{G}^W$ as follows. We define the vertex and edge sets of $G'$ using Equation \ref{eq:weighted_complete_graph}. We denote the bijection between vertex sets in this equation as $\eta: V \rightarrow V'$. Note that, $(V', E')$ is a complete graph.

\begin{equation}
\label{eq:weighted_complete_graph}
\begin{split}
V' & = V \\
E' & = \lbrace (v_i, v_j) \mid \forall \,\, v_i, v_j \in V' \rbrace
\end{split}
\end{equation}

Let $\text{Set}$ denote the space of subsets of $V$. We define the map $N: V \rightarrow \text{Set}$ to be the map from a vertex in the graph $G$ to its closed neighbourhood. The closed neighbourhood of a vertex $v$ is the union of $v$ and the set of all vertices adjacent to $v$. Given this we define the edge weighting function $C: E' \rightarrow [0, 1]$ using Equation \ref{eq:edge_cliqueness}. The function $C$ maps an edge to the \textit{Jaccard index} of the closed neighbourhoods of its vertices. It measures the degree to which the corresponding vertices in the unweighted graph $G$ have a common neighbourhood with greater common neighbourhoods mapping to higher values. Since all vertices belonging to a given clique will have the same closed neighbourhood, it can also be interpreted as the degree to which the edge in question belongs to a clique.

\begin{equation}
\label{eq:edge_cliqueness}
C( (v_1, v_2) ) = \frac{\vert N(\eta(v_1)) \cap N(\eta(v_2)) \vert}{\vert N(\eta(v_1)) \cup N(\eta(v_2)) \vert}
\end{equation}

To illustrate the proposed mapping $W: \mathcal{G}^U \rightarrow \mathcal{G}^W$ consider the graph $G \in \mathcal{G}^U$ displayed in Figure \ref{fig:cliqueness_a}. The corresponding graph $G'=W(G) \in \mathcal{G}^W$ is displayed in Figure \ref{fig:cliqueness_b} where corresponding edge weighting function values are displayed on each edge. The graph $G'$ is a complete graph however only those edges which have a non-zero edge weighting function value are displayed in this figure. From this figure we see that inter clique edges are assigned higher values than intra clique edges. This captures the fact that these edges are more significant.

\begin{figure}
\begin{center}
\subfigure[]{\includegraphics[width=5.5cm]{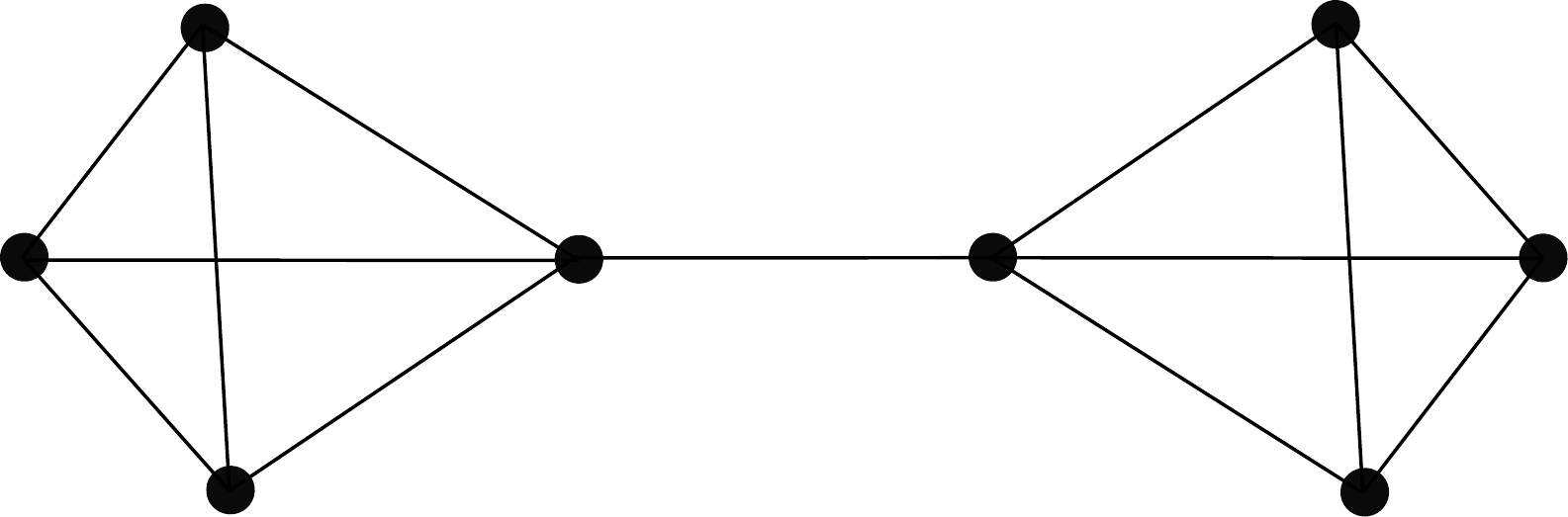}
\label{fig:cliqueness_a}}
\hspace{0.1cm}
\subfigure[]{\includegraphics[width=5.5cm]{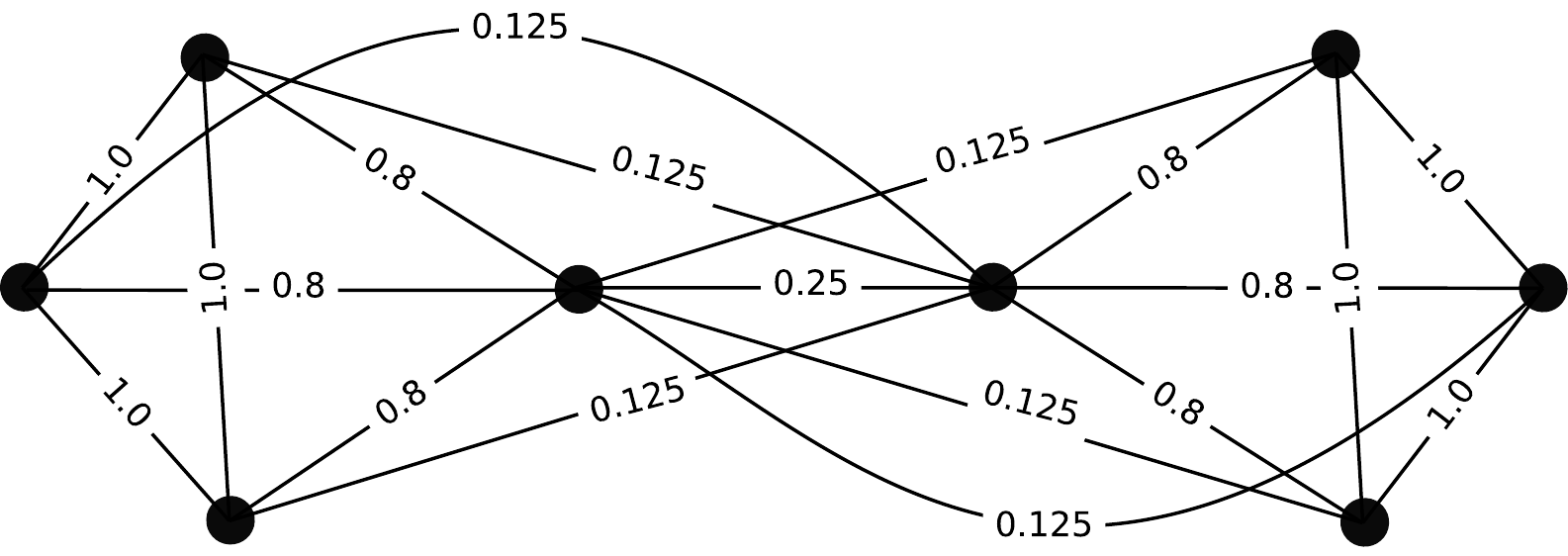}
\label{fig:cliqueness_b}}
\caption{An unweighted graph is displayed in (a) where vertices and edges are represented by dots and line segments respectively. The corresponding weighted graph computed using the proposed mapping $W$ is displayed in (b). Note that, only these edges which have a non zero edge weight value are displayed in (b).}
\label{fig:graph_cliqueness}
\end{center}
\end{figure}

In the second step of the proposed method a filtration on the weighted graph $G'=(V', E', C)$ is computed as follows. First we compute the clique complex of the graph $(V', E')$ which we denote $K$. Let $b:K \rightarrow E'$ denote the bijective map from the set of 1-simplices of $K$ to $E'$. Given this we define a function $f: K \rightarrow [0,1]$ using Equation \ref{eq:clique_filtration_f}. This function is monotonic; that is, $f(\sigma) \leq f(\tau)$ wherever $\sigma$ is a face of $\tau$. Broadly speaking, it assigns to each simplex the minimum of the $C$ function values on the set of edges in the corresponding clique.

\begin{equation}
\label{eq:clique_filtration_f}
f(\sigma) = \begin{cases}
\min(C(b(\gamma))) & \vert \sigma \vert \geq 1, \vert \gamma \vert = 1, \gamma \in \sigma \\
\max(C(b(\gamma))) & \vert \sigma \vert = 0, \vert \gamma \vert = 1, \sigma \in \gamma \\
0 & \text{Otherwise} 
\end{cases}
\end{equation}

We define a binary relation $\leq$ on the simplices of $K$ using Equation \ref{eq:total_order_relation} where the function $l: K \rightarrow \mathbb{Z}^{\geq}$ (non-negative integers) maps simplices to the random order they are stored in memory. This function is used to assign a random order to those simplices with equal $f$ function values and equal dimension.

\begin{equation}
\label{eq:total_order_relation}
\begin{split}
\leq = \lbrace (\sigma, \gamma) \mid & (f(\sigma) > f(\gamma)) \\
 & \lor (f(\sigma) = f(\gamma) \land \sigma \subset \gamma) \\
 & \lor (f(\sigma) = f(\gamma) \land \vert \sigma \vert = \vert \gamma \vert \land l(\sigma) < l(\gamma)) \rbrace
\end{split}
\end{equation}

We use this binary relation to define a total order $\sigma_1 \leq \sigma_2 \leq \dots \leq \sigma_{m}$ on the set of $m$ simplices in $K$. This total order is in turn used to define a filtration $\left( K_0, K_1 \dots, K_{m} \right)$ where $K_0 = \emptyset$ and $K_i = K_{i-1} \cup \sigma_{i}$ for $i=1 \dots m$. In the final of the three steps in the proposed method, the persistent homology of this filtration is computed using the method of \cite{zomorodian2005computing}. We represent all persistent diagrams results from this method in terms of the values of the function $f$ defined in Equation \ref{eq:clique_filtration_f}. That is, for each $(i,j)$ in a given persistence diagram we apply the map defined in Equation \ref{eq:pd_map} where the codomain of this map is the space $\lbrace a \times b \: \vert \: a \in [1, -\infty], b \in [1, -\infty], a \geq b \rbrace$.

\begin{equation}
\label{eq:pd_map}
(i,j) \mapsto (f(K_i), f(K_j)) 
\end{equation}

To illustrate the proposed method for topological graph analysis consider the graphs in Figures \ref{fig:stable_discrimination_a_PD}, \ref{fig:stable_discrimination_b_PD} and \ref{fig:stable_discrimination_c_PD} which we considered in the introduction section of this article. Recall that the latter two graphs are topologically similar to each other and dissimilar to the first graph. The $0$- and $1$-dimensional persistent diagrams corresponding to these graphs and computing using the proposed method are displayed in Figures \ref{fig:stable_discrimination_a_cliqueness_PD}, \ref{fig:stable_discrimination_b_cliqueness_PD} and \ref{fig:stable_discrimination_c_cliqueness_PD} respectively. For each graph the corresponding $0$-dimensional persistent diagram contains a point of infinite persistence plus zero, one and one points of finite persistence respectively. The additional point of finite persistence in the latter two diagrams indicates that the graphs in question contain an additional clique. In this case the $0$-dimensional persistent diagrams corresponding to the graphs in Figures \ref{fig:stable_discrimination_b_PD} and \ref{fig:stable_discrimination_c_PD} are similar to each other but dissimilar to the $0$-dimensional persistent diagram corresponding to the graph in Figure \ref{fig:stable_discrimination_a_PD}. Therefore, in the case of these three graphs, the proposed method performs accurate discrimination.

The $0$- and $1$-dimensional persistent diagrams corresponding to the graphs in Figures \ref{fig:stable_discrimination_a_PD}, \ref{fig:stable_discrimination_b_PD} and \ref{fig:stable_discrimination_c_PD} and computing using the method which computes a \textit{clique complex filtration} are displayed in Figures \ref{fig:stable_discrimination_a_clique_PD}, \ref{fig:stable_discrimination_b_clique_PD} and \ref{fig:stable_discrimination_c_clique_PD} respectively. For each graph the corresponding $0$-dimensional persistent diagram contains a point of finite persistence for each individual graph vertex plus one, two and one points of infinite persistence respectively. In this case the $0$-dimensional persistent diagrams corresponding to the graphs in Figures \ref{fig:stable_discrimination_a_PD} and \ref{fig:stable_discrimination_c_PD} are similar to each other but dissimilar to the $0$-dimensional persistent diagram corresponding to the graph in Figure \ref{fig:stable_discrimination_b_PD}. Therefore, in the case of these three graphs, the method in question does not perform accurate discrimination.

The $0$- and $1$-dimensional persistent diagrams corresponding to the graphs in Figures \ref{fig:stable_discrimination_a_PD}, \ref{fig:stable_discrimination_b_PD} and \ref{fig:stable_discrimination_c_PD} and computing using the method which computes a \textit{power complex filtration} are displayed in Figures \ref{fig:stable_discrimination_a_rips_PD}, \ref{fig:stable_discrimination_b_rips_PD} and \ref{fig:stable_discrimination_c_rips_PD} respectively. Similar to the previous discussion, for each graph the corresponding $0$-dimensional persistent diagram contains a point of finite persistence for each individual graph vertex plus one, two and one points of infinite persistence respectively. Therefore, in the case of these three graphs, the method in question does not perform accurate discrimination.

\begin{figure}
\begin{center}
\subfigure[]{\raisebox{8mm}{\includegraphics[width=2.5cm]{images/stable_discrimination_a}}
\label{fig:stable_discrimination_a_PD}}
\hspace{0.1cm}
\subfigure[]{\includegraphics[width=2.5cm]{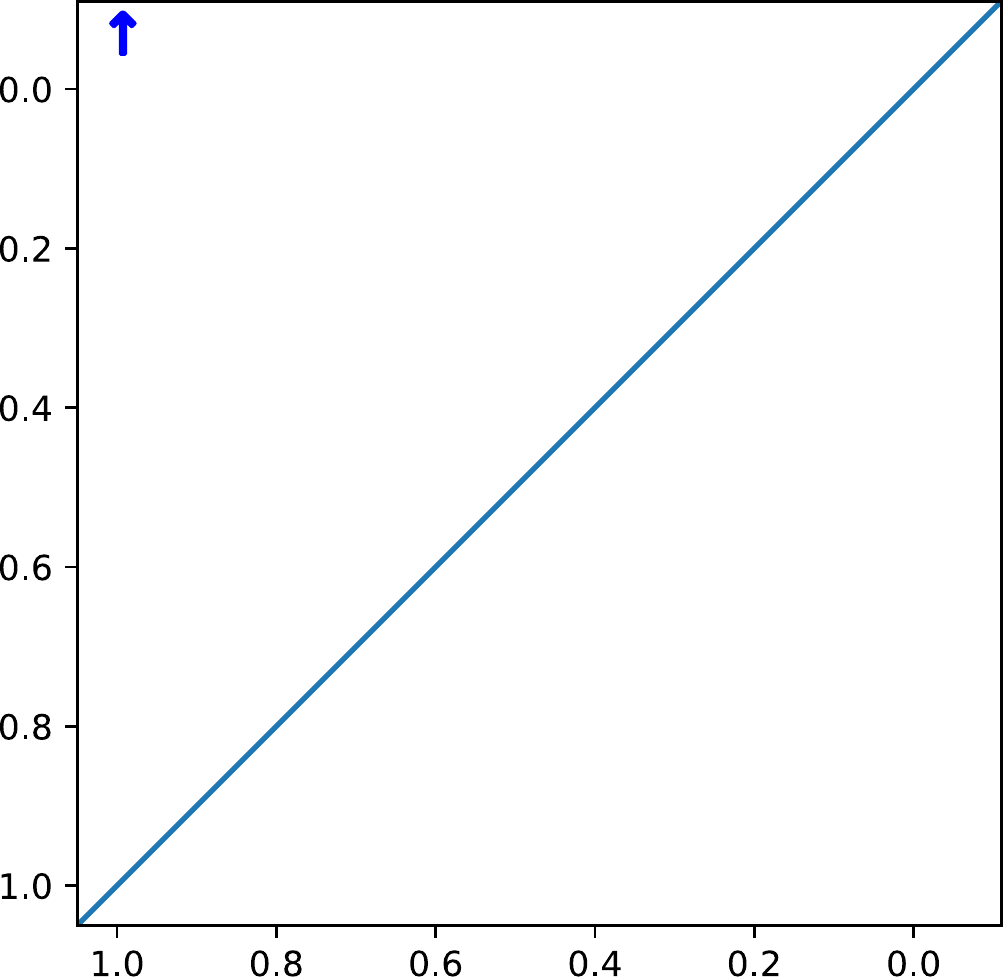}
\label{fig:stable_discrimination_a_cliqueness_PD}}
\hspace{0.1cm}
\subfigure[]{\includegraphics[width=2.5cm]{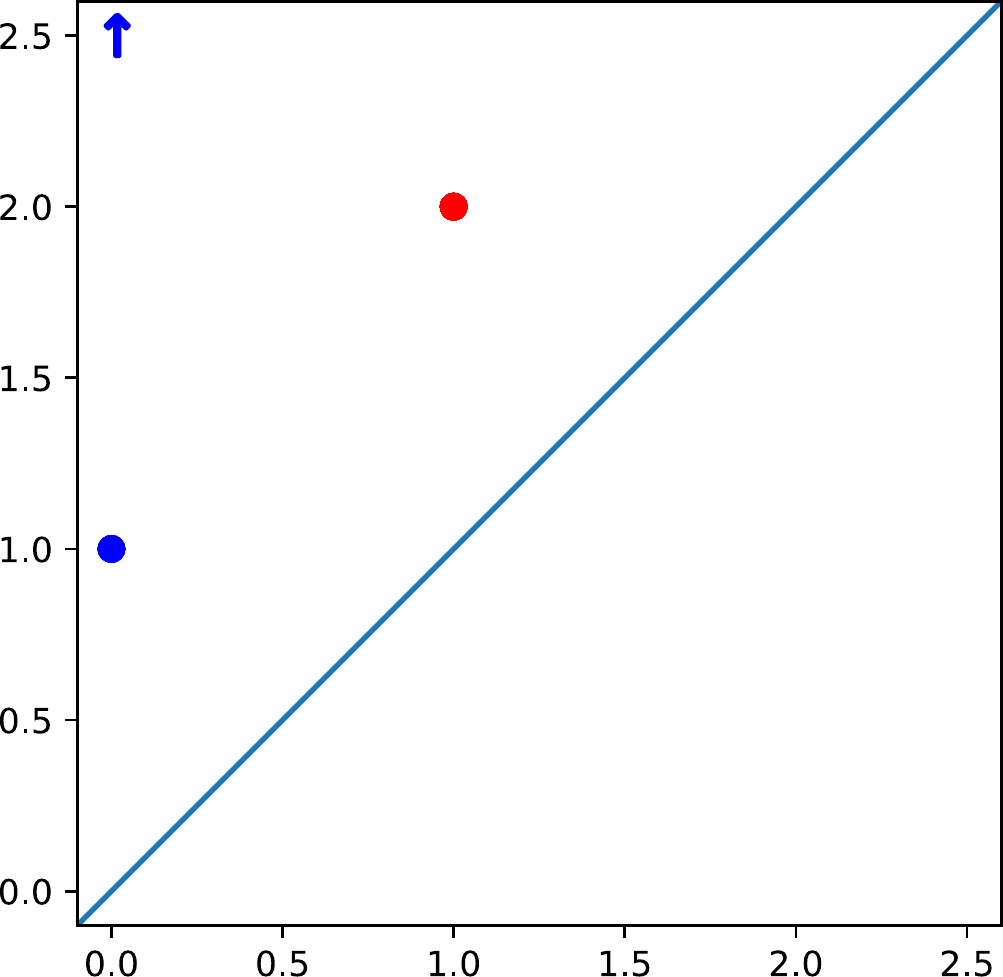}
\label{fig:stable_discrimination_a_clique_PD}}
\hspace{0.1cm}
\subfigure[]{\includegraphics[width=2.5cm]{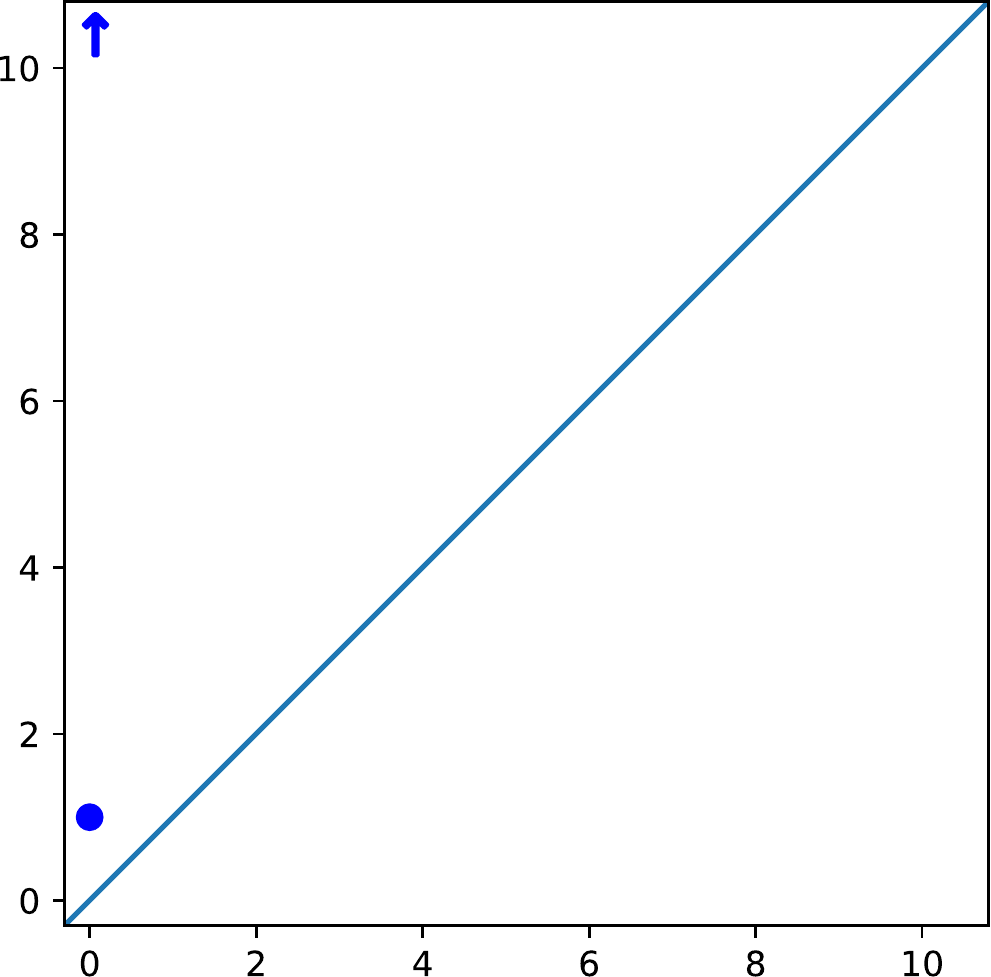}
\label{fig:stable_discrimination_a_rips_PD}}
\\
\subfigure[]{\raisebox{8mm}{\includegraphics[width=2.5cm]{images/stable_discrimination_b}}
\label{fig:stable_discrimination_b_PD}}
\hspace{0.1cm}
\subfigure[]{\includegraphics[width=2.5cm]{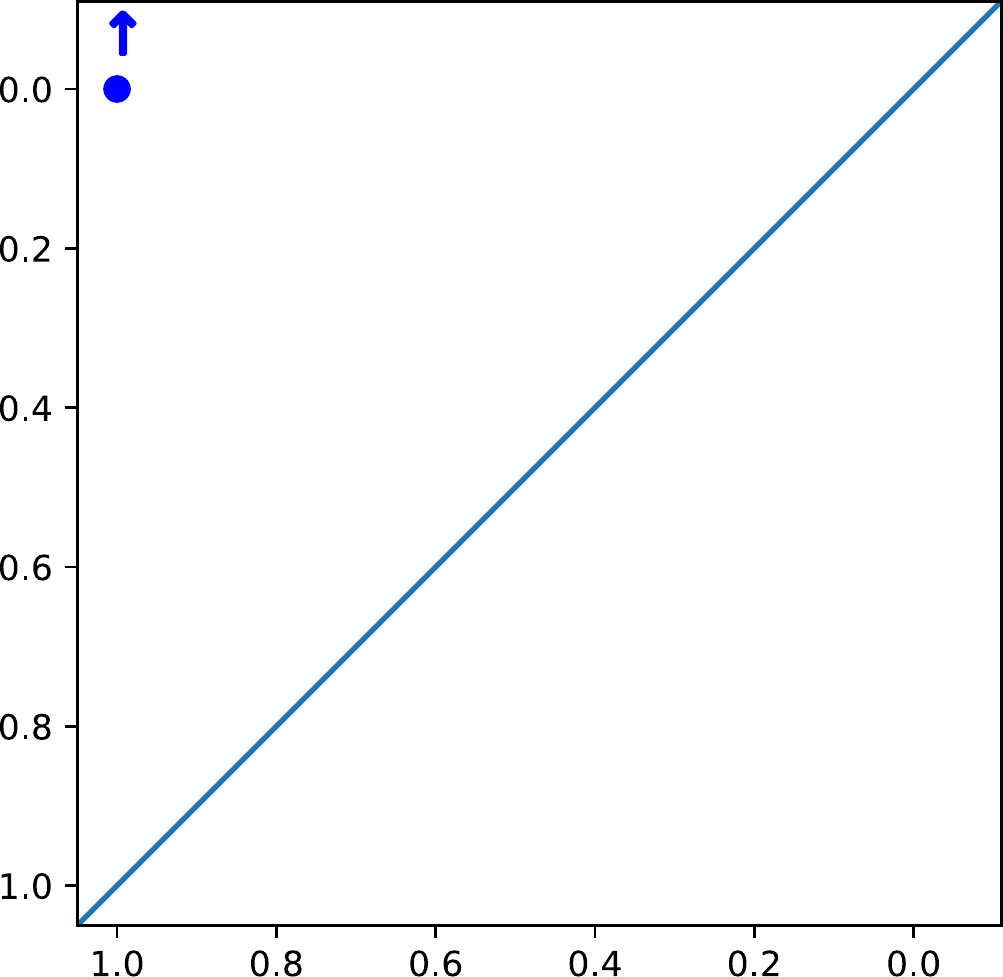}
\label{fig:stable_discrimination_b_cliqueness_PD}}
\hspace{0.1cm}
\subfigure[]{\includegraphics[width=2.5cm]{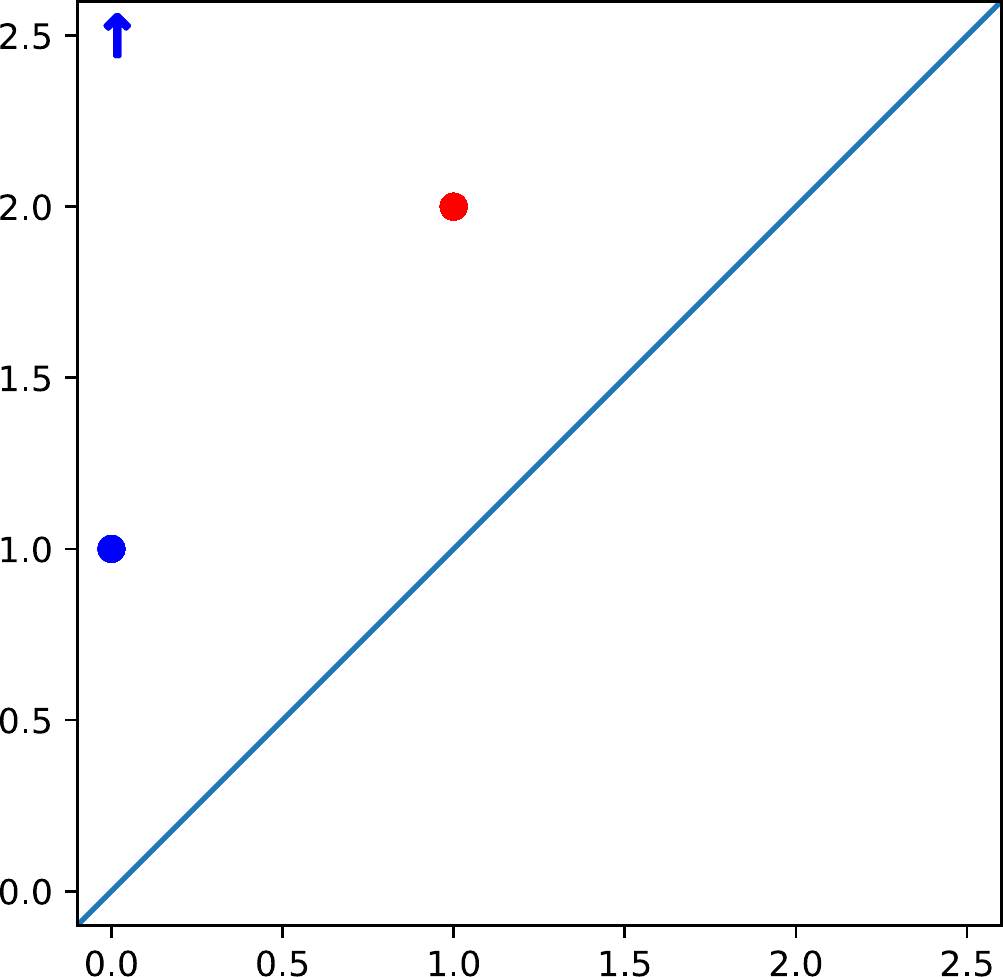}
\label{fig:stable_discrimination_b_clique_PD}}
\hspace{0.1cm}
\subfigure[]{\includegraphics[width=2.5cm]{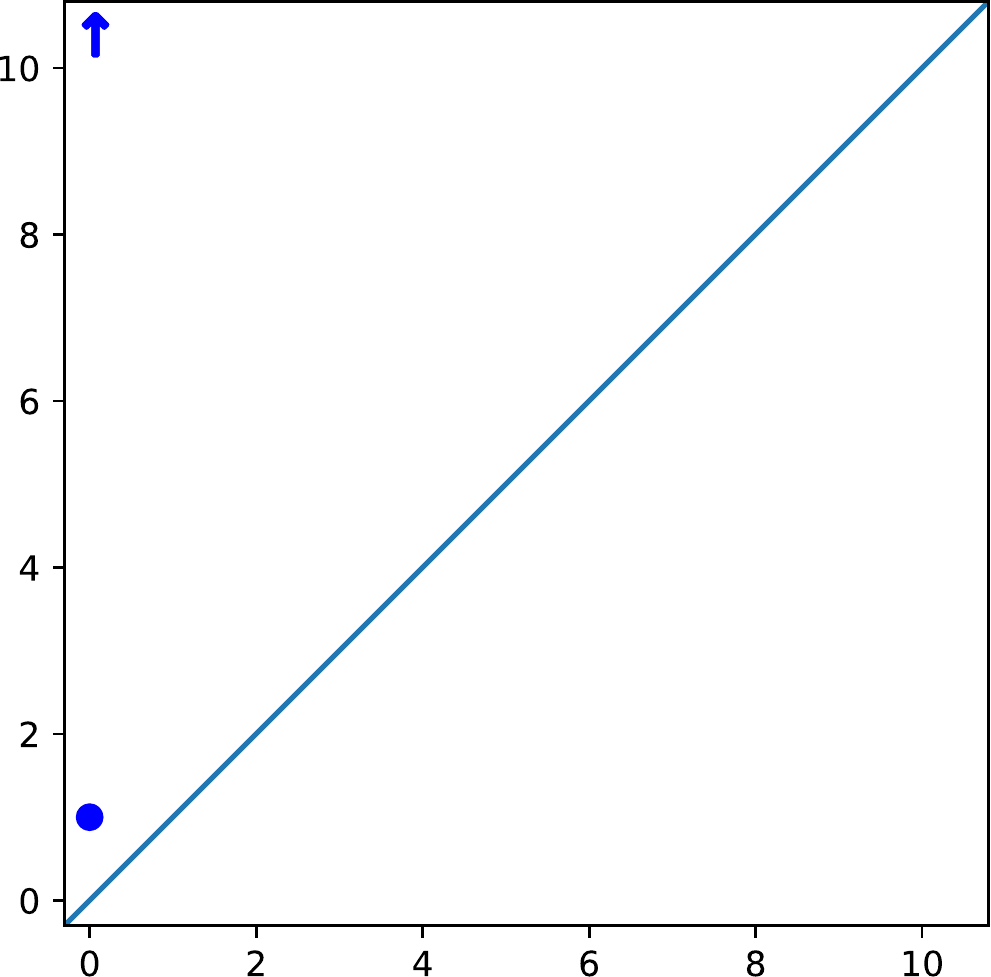}
\label{fig:stable_discrimination_b_rips_PD}}
\\
\subfigure[]{\raisebox{8mm}{\includegraphics[width=2.5cm]{images/stable_discrimination_c}}
\label{fig:stable_discrimination_c_PD}}
\hspace{0.1cm}
\subfigure[]{\includegraphics[width=2.5cm]{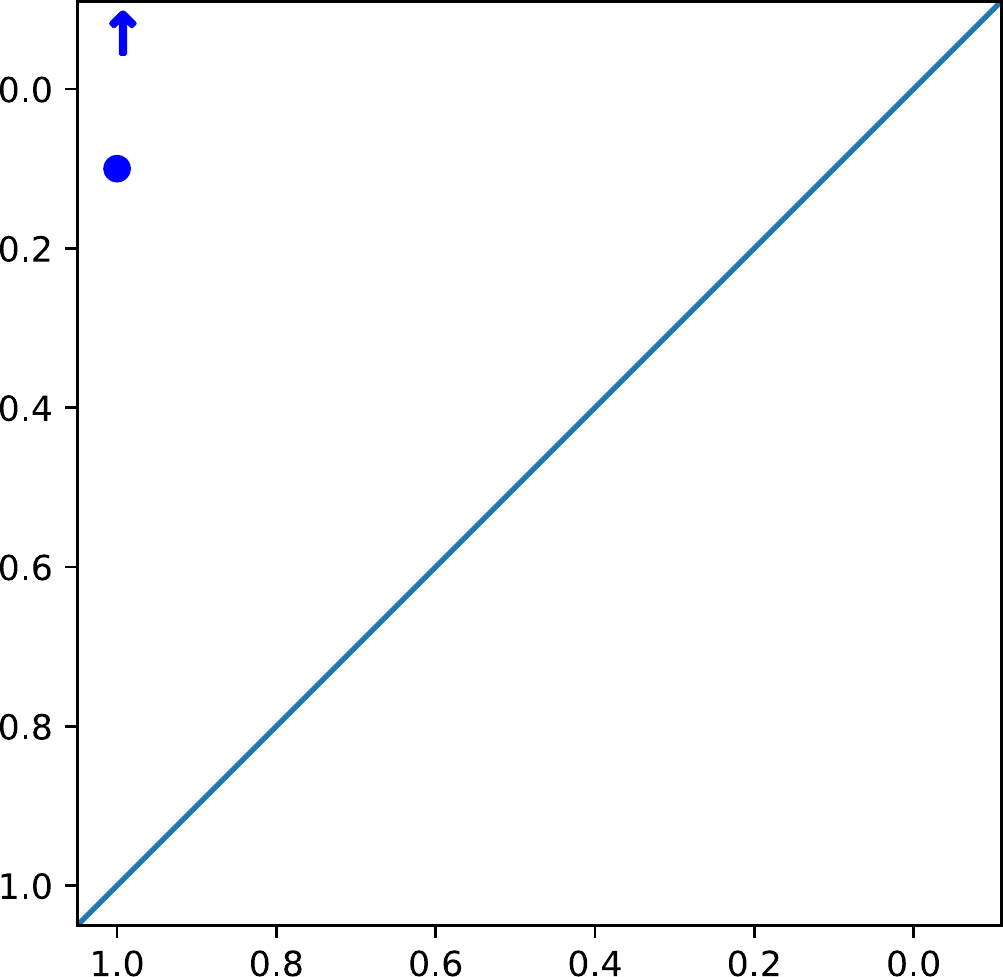}
\label{fig:stable_discrimination_c_cliqueness_PD}}
\hspace{0.1cm}
\subfigure[]{\includegraphics[width=2.5cm]{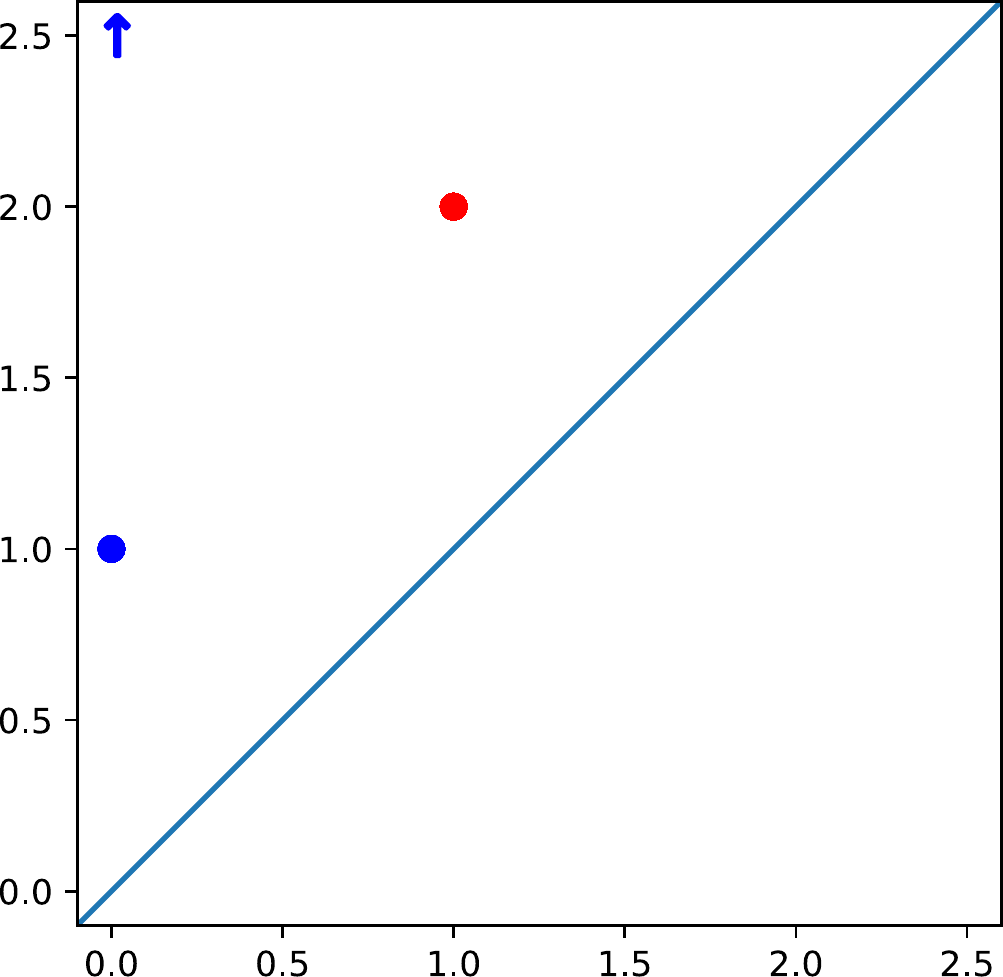}
\label{fig:stable_discrimination_c_clique_PD}}
\hspace{0.1cm}
\subfigure[]{\includegraphics[width=2.5cm]{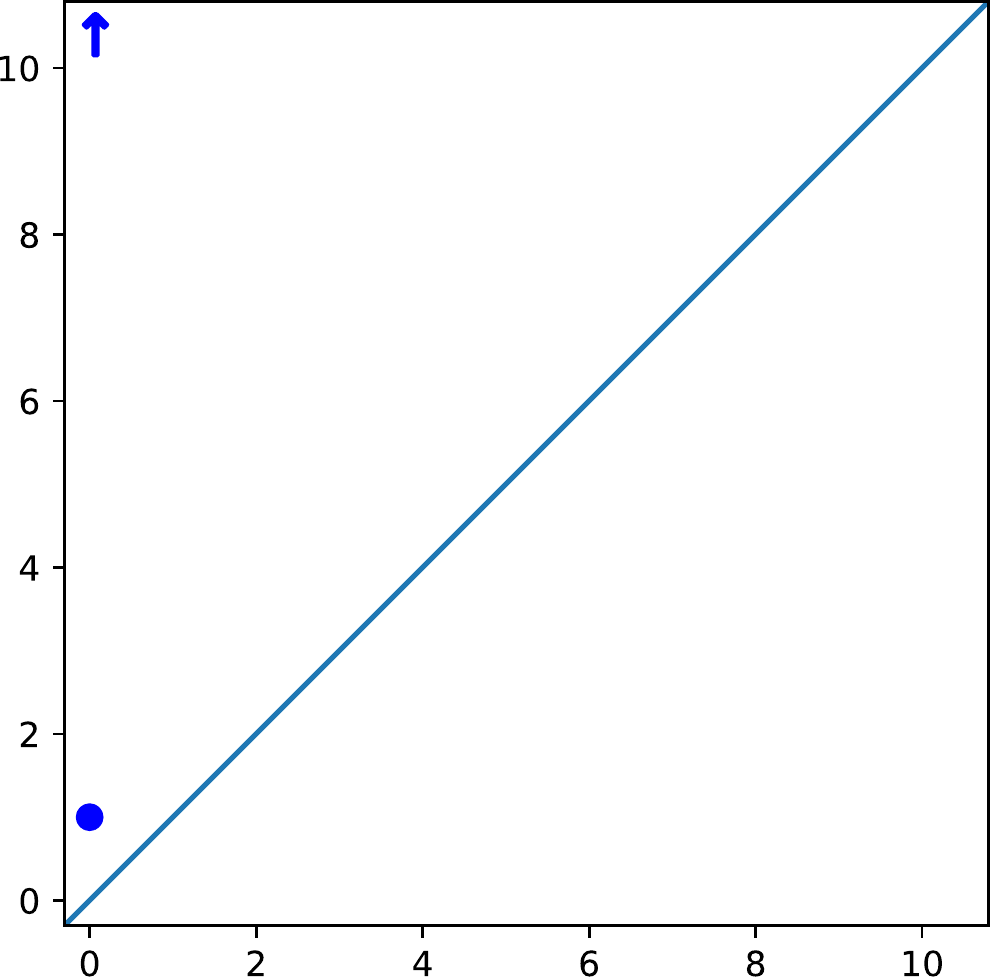}
\label{fig:stable_discrimination_c_rips_PD}}
\caption{The graphs in (a), (b) and (c) contain a single clique, two cliques and two cliques connected by a single edge respectively. The $0$- and $1$-dimensional persistent diagrams computed using the proposed method are displayed in (b), (f) and (j) respectively. The $0$- and $1$-dimensional persistent diagrams computed using a \textit{clique complex filtration} are displayed in (c), (g) and (k) respectively. The $0$- and $1$-dimensional persistent diagrams computed using a \textit{power complex filtration} are displayed in (d), (h) and (l) respectively. In all figures points belonging to $0$- and $1$-dimensional persistent diagrams are represented by blue and red dots respectively. Points which do not die and have infinite persistence are represented by arrows where the death value is replaced with an lower/upper bound.}
\label{fig:graph_loop}
\end{center}
\end{figure}

As a second example consider the graphs displayed in Figures \ref{fig:graph_loop_1} and \ref{fig:graph_loop_2}. The former graph contains a densely connected cycle while the latter contains the same cycle with an additional single edge across the cycle. Therefore the $1$-dimensional hole introduced by this additional edge is topologically less significant than that corresponding to the densely connected cycle. The $0$- and $1$-dimensional persistent diagrams corresponding to these graphs and computing using the proposed method are displayed in Figures \ref{fig:graph_loop_1_cliqueness_PD} and \ref{fig:graph_loop_2_cliqueness_PD} respectively. For each graph the corresponding $1$-dimensional persistent diagrams contain one and two points of finite persistence respectively. This accurately models the fact that the graphs in question contain one and two $1$-dimensional holes respectively. The persistence of one of the two points in Figure \ref{fig:graph_loop_2_cliqueness_PD} is significantly greater than the other. This accurately models the fact that the $1$-dimensional hole in question is more significant.

The $0$- and $1$-dimensional persistent diagrams corresponding to the graphs in Figures \ref{fig:graph_loop_1} and \ref{fig:graph_loop_2} and computing using the method which computes a \textit{clique complex filtration} are displayed in Figures \ref{fig:graph_loop_1_clique_PD} and \ref{fig:graph_loop_2_clique_PD} respectively. For each graph the corresponding $1$-dimensional persistent diagrams contain one and two points of infinite persistence respectively. This accurately models the fact that the graphs in question contain one and two $1$-dimensional holes respectively. However it does not accurately model the fact that one $1$-dimensional hole is more significant than the other.

The $0$- and $1$-dimensional persistent diagrams corresponding to the graphs in Figures \ref{fig:graph_loop_1} and \ref{fig:graph_loop_2} and computing using the method which computes a \textit{power complex filtration} are displayed in Figures \ref{fig:graph_loop_1_rips_PD} and \ref{fig:graph_loop_2_rips_PD} respectively. For each graph the corresponding $1$-dimensional persistent diagrams contain one and two points of finite persistence respectively. This accurately models the fact that the graphs in question contain one and two $1$-dimensional holes respectively. The two points in the latter persistent diagram have equal birth and death values and in turn equal persistence. This does not accurately model the fact that one $1$-dimensional hole is more significant than the other.

\begin{figure}
\begin{center}
\subfigure[]{\includegraphics[width=2.5cm]{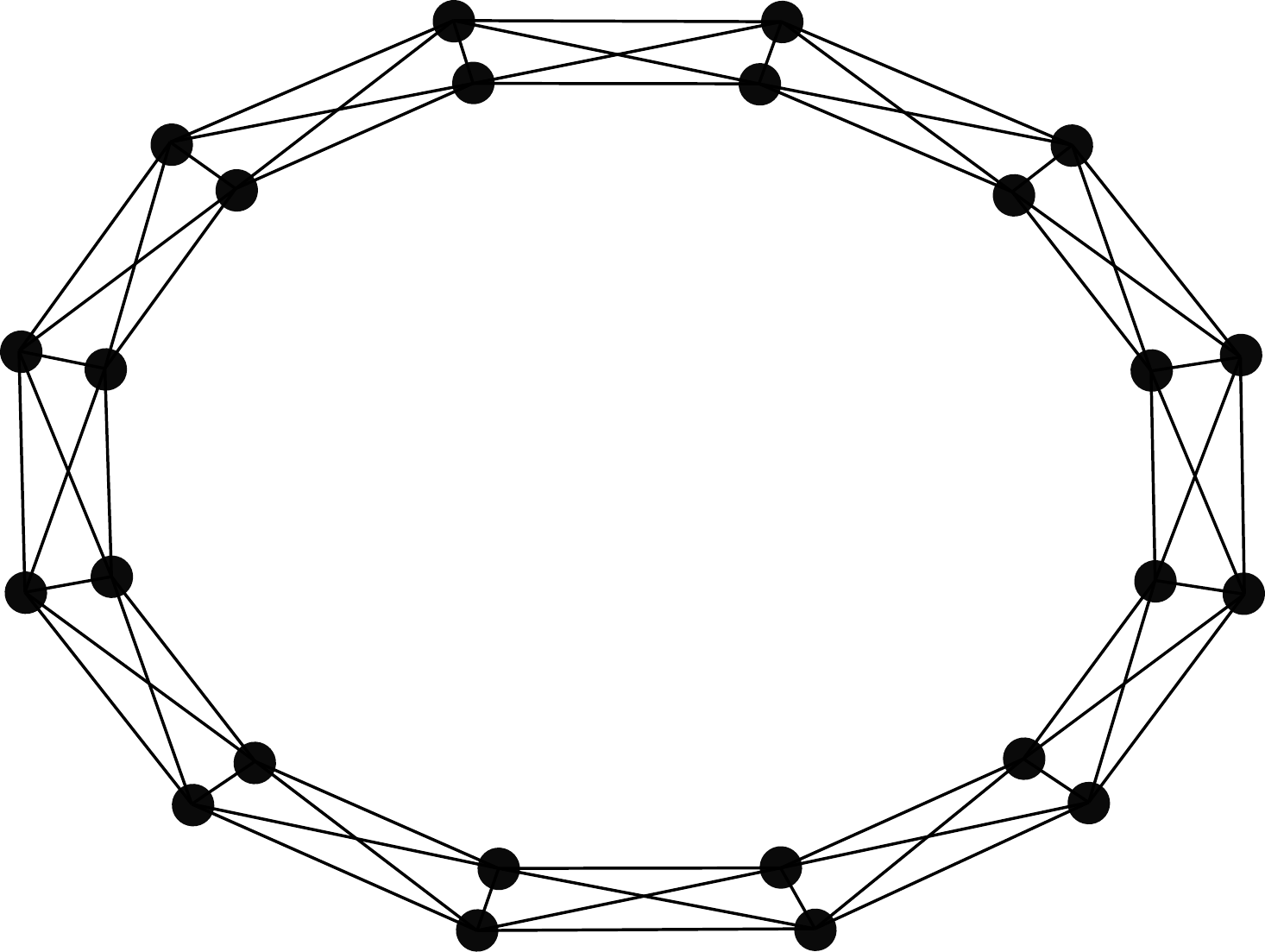}
\label{fig:graph_loop_1}}
\hspace{0.1cm}
\subfigure[]{\includegraphics[width=2.5cm]{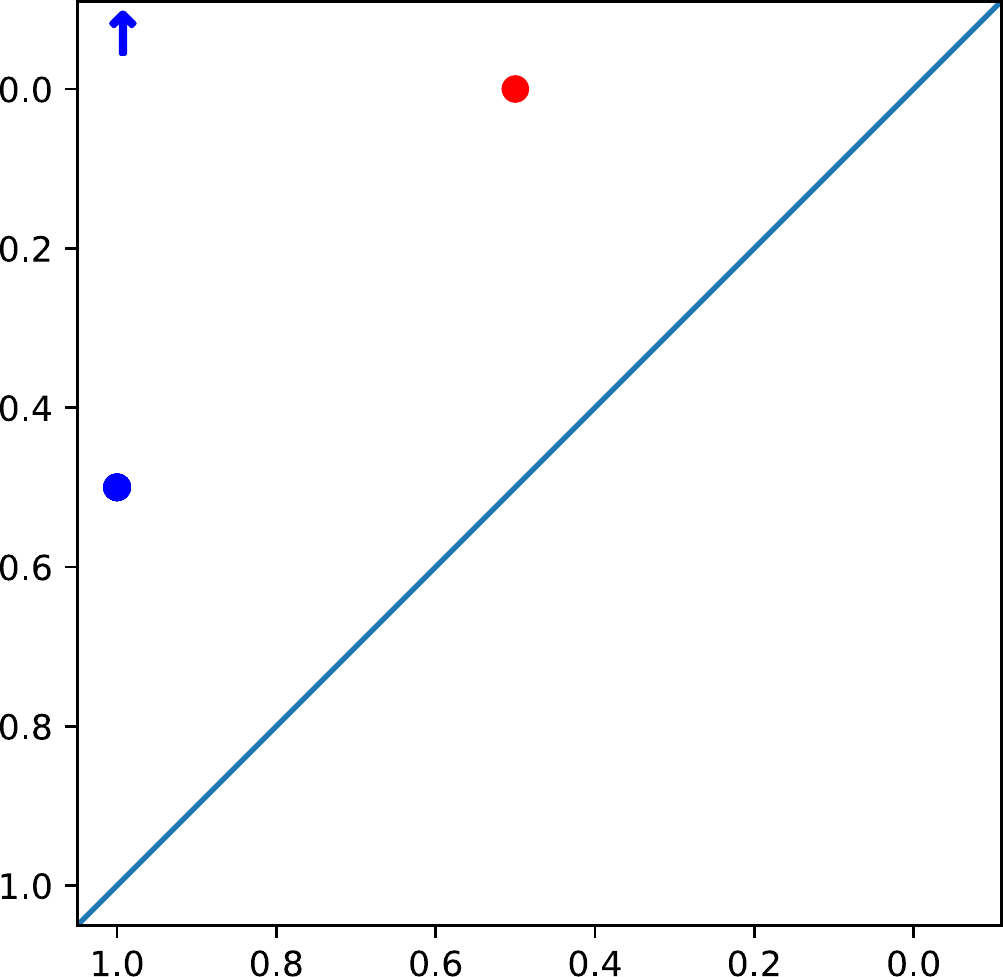}
\label{fig:graph_loop_1_cliqueness_PD}}
\hspace{0.1cm}
\subfigure[]{\includegraphics[width=2.5cm]{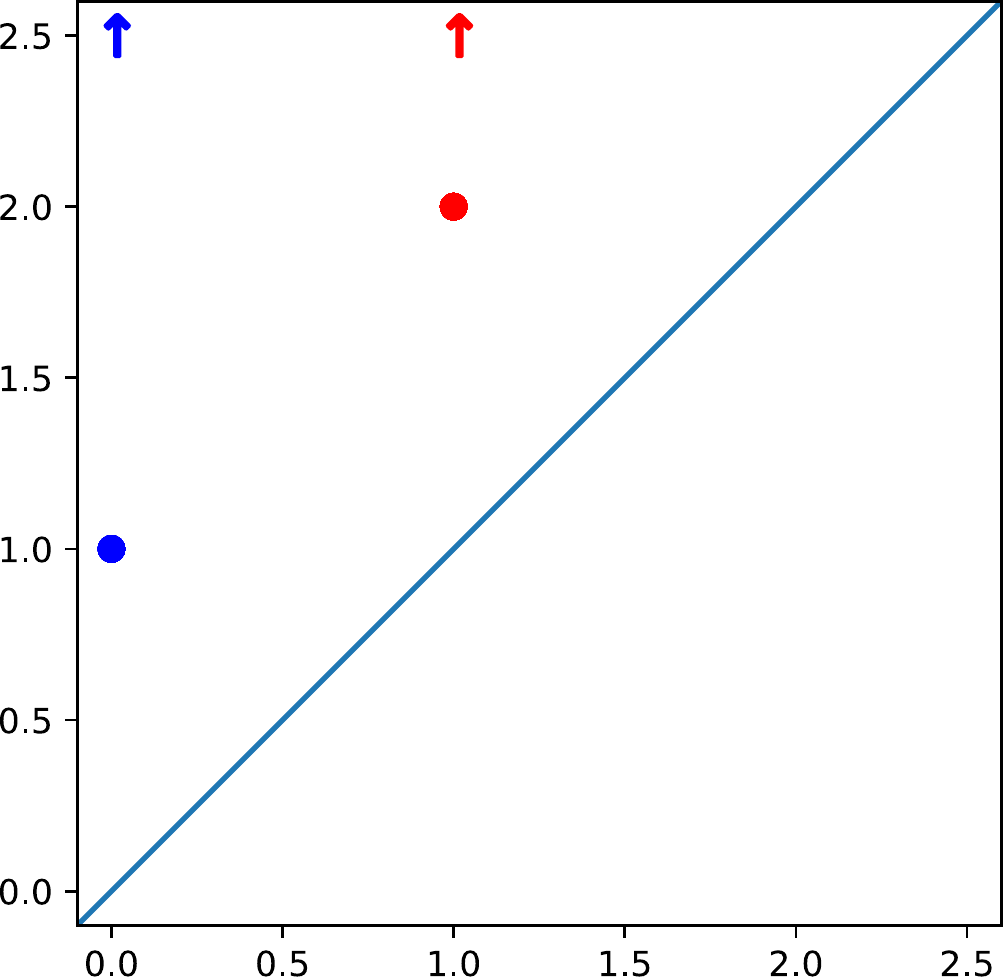}
\label{fig:graph_loop_1_clique_PD}}
\hspace{0.1cm}
\subfigure[]{\includegraphics[width=2.5cm]{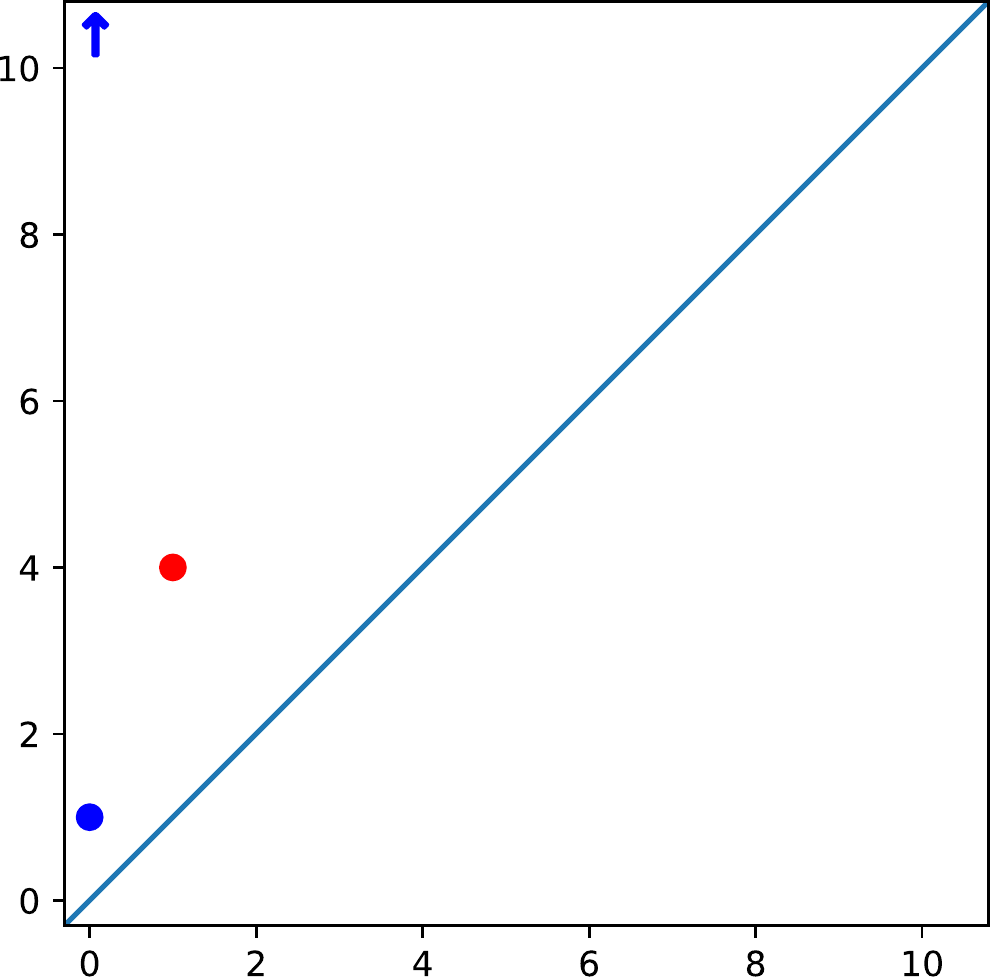}
\label{fig:graph_loop_1_rips_PD}}
\\
\subfigure[]{\includegraphics[width=2.5cm]{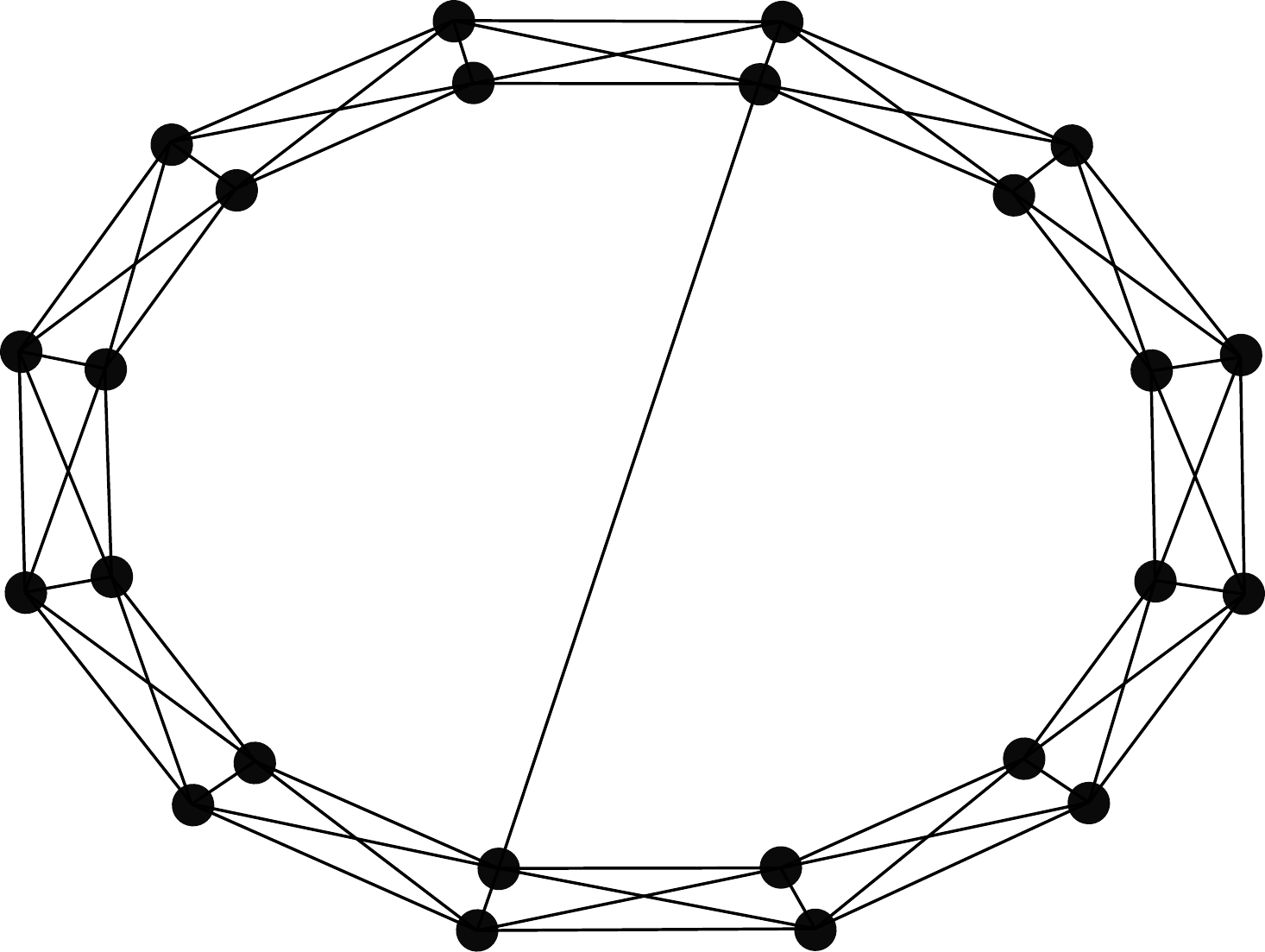}
\label{fig:graph_loop_2}}
\hspace{0.1cm}
\subfigure[]{\includegraphics[width=2.5cm]{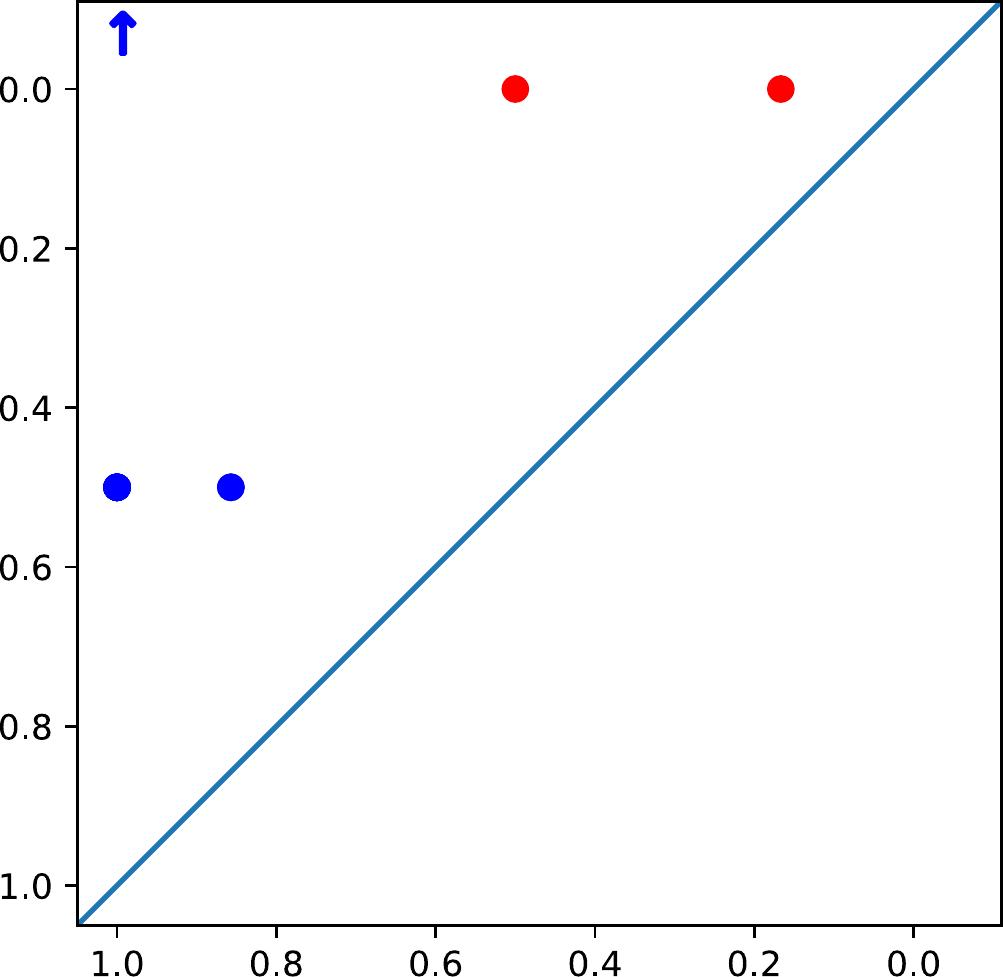}
\label{fig:graph_loop_2_cliqueness_PD}}
\hspace{0.1cm}
\subfigure[]{\includegraphics[width=2.5cm]{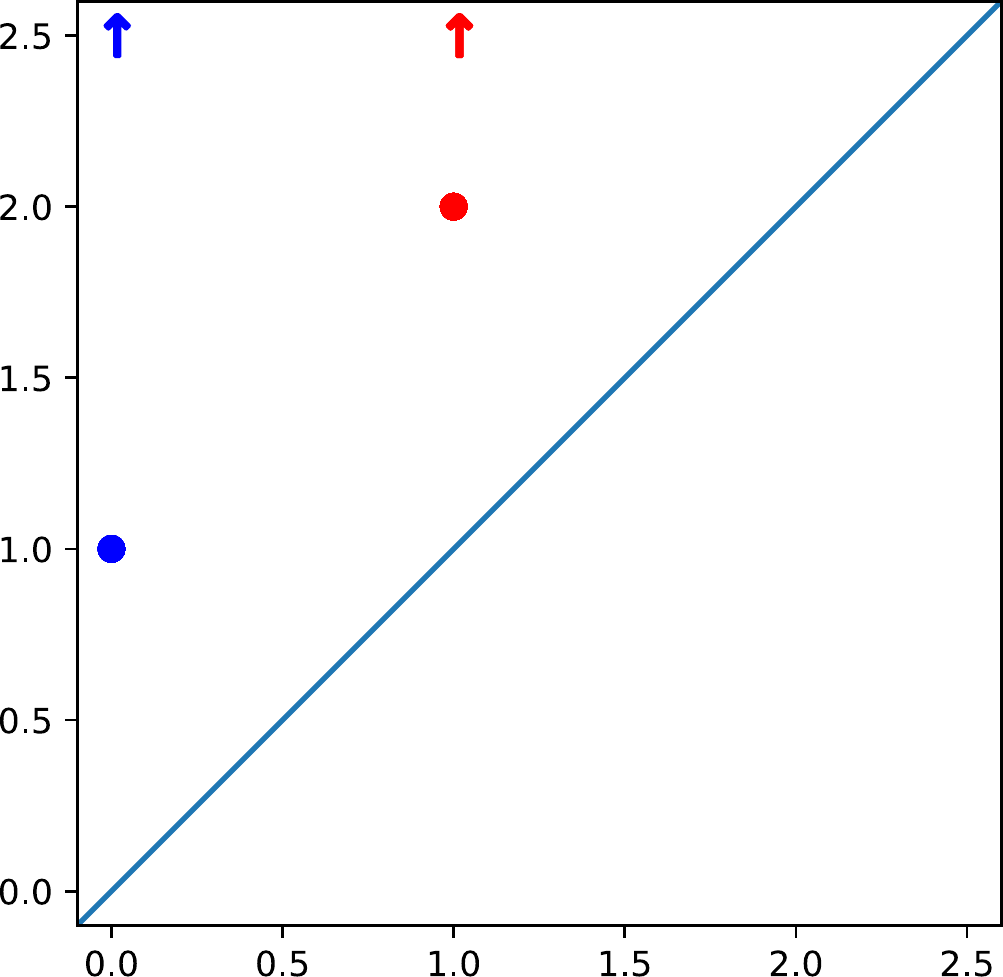}
\label{fig:graph_loop_2_clique_PD}}
\hspace{0.1cm}
\subfigure[]{\includegraphics[width=2.5cm]{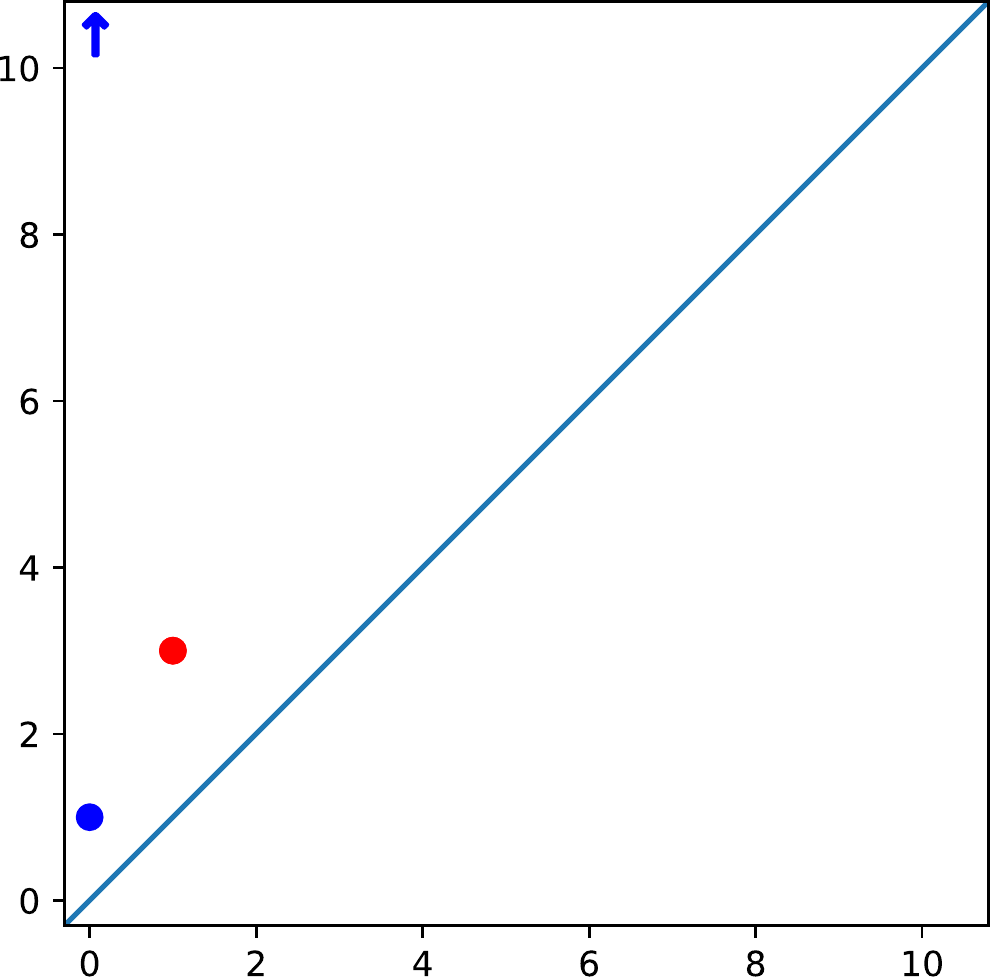}
\label{fig:graph_loop_2_rips_PD}}
\caption{The graphs in (a) and (e) contain a densely connected cycle and a densely connected cycle with an additional single edge across the cycle respectively. The $0$- and $1$-dimensional persistent diagrams computed using the proposed method are displayed in (b) and (f) respectively. The $0$- and $1$-dimensional persistent diagrams computed using a \textit{clique complex filtration} are displayed in (c) and (g) respectively. The $0$- and $1$-dimensional persistent diagrams computed using a \textit{power complex filtration} are displayed in (d) and (h) respectively.}
\label{fig:graph_loop}
\end{center}
\end{figure}

Note that, Suh et al. \cite{suh2019persistent} proposed a method for constructing a filtration for an unweighted graph which also uses the Jaccard index. This method maps the input unweighted graph to a weighted graph with the same topology. This is distinct from the method proposed in this article which maps the input unweighted graph to a complete weighted graph. The need to map to a complete graph to ensure stability and the ability to accurately discriminate between graphs is a key insight of this work.

\subsection{Implementation Details}
\label{sec:stable_graph_analysis:implementation}
In this section we present implementation details of the proposed topological graph analysis method described in the previous section.

The proposed method has structure which we exploit to achieve an efficient implementation. Graphs in the space $\mathcal{G}^W$ are complete graphs. However, for a sparse graph $G \in \mathcal{G}^U$ the corresponding graph $G' = W(G) \in \mathcal{G}^W$ will have a significant number of edges which the function $C$ maps to $0$. Recall that $K$ is the clique complex of $G'$. If a clique in $G'$ contains an edge which $C$ maps to $0$ then $f$ will map the corresponding simplex in $K$ to $0$. The function $f$ maps all simplicial in $K$ to a value greater than or equal to $0$. Therefore, with the exception of a single point in the $0$-dimensional persistence diagram corresponding to the connected component which is the complete graph, any point in a persistence diagram which is born through the additional of a simplex which $f$ maps to $0$ will immediately die and have persistence of $0$. Furthermore, with the exception of the point in the $0$-dimensional persistence diagram corresponding to the connected component which is the complete graph, all points in the persistence diagrams will die at a value greater than or equal to $0$. We exploit this structure in the following way. We only add to $E'$, the set of edges in $G'$, those edges which the function $C$ maps to a value greater than $0$. Subsequently, with the exception of a single point in the $0$-dimensional persistence diagram corresponding to the complete graph, for all points in the resulting persistence diagrams which do not die we replace their death value of $-\infty$ with a value of $0$. All isolated vertices in $G$ which are not adjacent to any edges will correspond to points in the $0$-dimensional persistence diagram with persistence $0$. We exploit this structure by removing all isolated vertices from $G$.

Pseudocode for the proposed method is presented in Algorithm \ref{alg:top_graph_analysis}. The goal of this algorithm is to compute the $p$-dimensional persistence diagram for a given value of $p$ and input graph $G \in \mathcal{G}^U$. The algorithm first removes isolated vertices from the unweighted graph $G$ (line \ref{alg:top_graph_analysis:rem_iso}). The weighted graph $G'$ is then initialized to be an edgeless graph with the same vertex set as $G$ (line \ref{alg:top_graph_analysis:init_G}). Those edges which the function $C$ maps to a non-zero value are then added to $G'$ (lines \ref{alg:top_graph_analysis:init_G_E_1} to \ref{alg:top_graph_analysis:init_G_E_2}). Next, the clique complex $K$ is constructed by iteratively adding simplices of dimensions $0$ to $p+1$ (lines \ref{alg:top_graph_analysis:init_K_1} to \ref{alg:top_graph_analysis:init_K_2}). Since the goal is to compute the $p$-dimensional persistence diagram, simplices of larger dimensions do not need to be considered. The cliques in $G'$ corresponding to these simplices are computed using the method proposed by Zhang et al. \cite{zhang2005genome}. The filtration of $K$ is then computed using Equation \ref{eq:total_order_relation} (line \ref{alg:top_graph_analysis:filtration}). Next, the $p$-dimensional persistence diagram $\mathcal{D}$ of this filtration is computed using the method by Zomorodian et al. \cite{zomorodian2005computing} (line \ref{alg:top_graph_analysis:homology}).

Let $\max$ be a map from a $0$-dimensional persistence diagram to the element in this set which has the largest birth value and has infinite persistence. If the persistence diagram $\mathcal{D}$ is of dimension $0$, for each point in $\mathcal{D} \backslash \max(D)$ if the corresponding death value is $-\infty$ we replace this with $0$ (lines \ref{alg:top_graph_analysis:replace_1} to \ref{alg:top_graph_analysis:replace_2}). For example, consider the $0$-dimensional persistence diagram $\lbrace (1, -\infty), (0.8, -\infty), (0.9, 0.1)  \rbrace$. This step will replace this persistence diagram with $\lbrace (1, -\infty), (0.8, 0.0), (0.9, 0.1)  \rbrace$. If the persistence diagram $\mathcal{D}$ is of dimension greater than $0$, for each point in $\mathcal{D}$ if the corresponding death value is $-\infty$ we replace this with $0$ (lines \ref{alg:top_graph_analysis:replace_3} to \ref{alg:top_graph_analysis:replace_4}). For example, consider the $1$-dimensional persistence diagram $\lbrace (1, -\infty), (0.8, -\infty), (0.9, 0.1)  \rbrace$. This step will replace this persistence diagram with $\lbrace (1, 0.0), (0.8, 0.0), (0.9, 0.1)  \rbrace$.

\begin{algorithm}
\label{alg:top_graph_analysis}
\caption{Proposed Topological Graph Analysis Method}
 \KwIn{An unweighted graph $G=(V,E) \in \mathcal{G}^U$ to which the method will be applied and $p \in \mathbb{Z}^{\geq}$ the dimension of the persistence diagram to be computed.}
 \KwOut{A $p$-dimensional persistence diagram describing topological features of $G$.}
 \BlankLine
\Begin{
 Remove isolated vertices from $G$ \\ \label{alg:top_graph_analysis:rem_iso}
 \vspace{.2cm}
 Initialize $G' = (V', E', C)$ with $V'=V$ and $E' = \emptyset$ \\ \label{alg:top_graph_analysis:init_G}
 \For{$v_1, v_2 \in V$} { \label{alg:top_graph_analysis:init_G_E_1}
 	\If{$C((v_1, v_2)) > 0$}{
 		$E' = E' \cup (v_1, v_2)$
 	}
 } \label{alg:top_graph_analysis:init_G_E_2}
 \vspace{.2cm}
 Initialize $K = \emptyset$ \\ \label{alg:top_graph_analysis:init_K_1}
 \For{$i \leftarrow 0$ \KwTo $p+1$} {
 	$K = K \cup \lbrace \sigma \: \vert \: \sigma \text{ is a clique in $G'$ of size $p$.} \rbrace$
 } \label{alg:top_graph_analysis:init_K_2}
 \vspace{.2cm}
 Compute filtration $\mathcal{K} = \left( K_0, K_1 \dots, K_{m} \right)$ of $K$ \\ \label{alg:top_graph_analysis:filtration}
 \vspace{.2cm}
 Compute $\mathcal{D}$ the $p$-dimensional persistence diagram of $\mathcal{K}$ \\ \label{alg:top_graph_analysis:homology}
 \uIf{$p=0$}{ \label{alg:top_graph_analysis:replace_1}
 	\For{$(b, d) \in \mathcal{D} \backslash \max(D)$} {
 		\uIf{$d = -\infty$}{
 			$d=0$
 		}
 	}
 } \label{alg:top_graph_analysis:replace_2}
 \uElse{ 
 	\For{$(b, d) \in \mathcal{D}$} { \label{alg:top_graph_analysis:replace_3}
 		\uIf{$d = -\infty$}{
 			$d=0$
 		}
 	}
 } \label{alg:top_graph_analysis:replace_4}
 return $D$
}
\end{algorithm}

All the steps in Algorithm \ref{alg:top_graph_analysis} have linear computational complexity apart from computing the cliques in the graph $G'$ and computing the $p$-dimensional persistence diagram of the filtration $\mathcal{D}$. These steps have computational complexity of $O(n^{p})$ and $O(m^3)$ respectively where $n$ is the number of vertices in $G'$ and $m$ is the number of simplices in $\mathcal{D}$. Therefore the total complexity of the algorithm is $O(n^{p}) + O(m^3)$. Empirically we find that the graph $G'$ is generally sparse and in turn the wall clock time of the algorithm to be reasonable.

\section{Stability Analysis}
\label{sec:stability}
This section presents a formal stability analysis of both the existing and the proposed methods for graph topological analysis. In subsection \ref{sec:stability_clique_power} we prove existing methods to not be stable. Subsequently, in subsection \ref{sec:stability_cliqueness} we prove the proposed method to be stable.

\subsection{Clique and Power Complex Filtration}
\label{sec:stability_clique_power}
In this subsection we prove by counter example the two existing methods for topological graph analysis which compute a \textit{clique complex filtration} and a \textit{power complex filtration} to not be stable whereby a small change in graph connectivity can result in a large change in the corresponding persistence diagrams. Although there exist other methods for constructing filtrations for unweighted graphs, these two methods are the most commonly used in practice.

Let $\mathcal{G}^U_{V}$ be the space of unweighted graphs which have an equal set of vertices $V$. We define a distance $D^U: \mathcal{G}^U_V \times \mathcal{G}^U_V \rightarrow \lbrace 0,1 \rbrace$ on this space in Equation \ref{eq:unweighted_graph_distance} where $\mathbbm{1}_{E}$ denotes an indicator function on the set $E$.

\begin{equation}
\label{eq:unweighted_graph_distance}
D^U((V, E_1), (V, E_2)) = \max_{i,j \in V} | \mathbbm{1}_{E_1}((i,j)) - \mathbbm{1}_{E_2}((i,j)) |
\end{equation}

\begin{theorem}
\label{th1}
Let $G_1 = (V, E_1)$, $G_2 = (V, E_2)$ $\in \mathcal{G}^U_V$ be two unweighted graphs. Let $P_1$ and $P_2$ be persistent diagrams of a given dimension $p$ corresponding to $G_1$ and $G_2$ respectively and computed using the clique complex filtration. The Bottleneck distance $B$ between $P_1$ and $P_2$ is not bounded above by a constant times the distance $D^U$ between the graphs $G_1$ and $G_2$; that is, $B(P_1, P_2)$ $\nleq$ $\alpha D^U(G_1, G_2)$ for a constant $\alpha$.
\end{theorem}
 
\begin{proof}
We present a proof by counter example. Let $G_1$ and $G_2$ be the two graphs in Figures \ref{fig:clique_stability_proof_a} and \ref{fig:clique_stability_proof_b} respectively. These two graphs differ by a single edge and consequently $D^U(G_1, G_2)=1$. Let us denote the $1$-dimensional persistence diagrams corresponding to these graphs as $P_1$ and $P_2$ respectively which equal the sets $\lbrace (1,\infty), (1,\infty), (1,\infty) \rbrace$ and $\lbrace (1,2), (1,2), (1,2), (1,2) \rbrace$ respectively. The bottleneck distance between these persistence diagrams is $\infty$; that is, $B(P_1, P_2) = \infty$. Hence $B(P_1, P_2)$ $\nleq$ $\alpha D^U(G_1, G_2)$.
\end{proof}

\begin{figure}
\begin{center}
\subfigure[]{\includegraphics[height=3.cm]{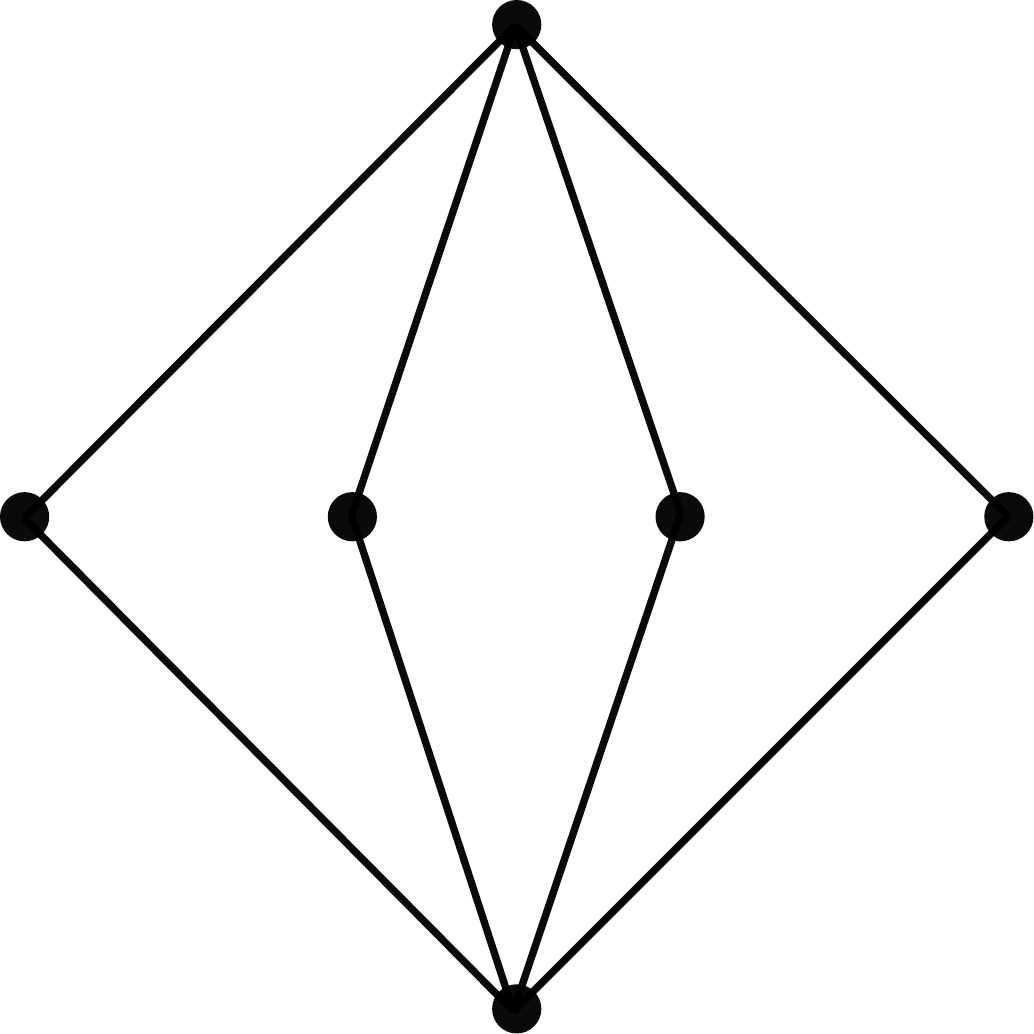}
\label{fig:clique_stability_proof_a}}
\hspace{2cm}
\subfigure[]{\includegraphics[height=3.cm]{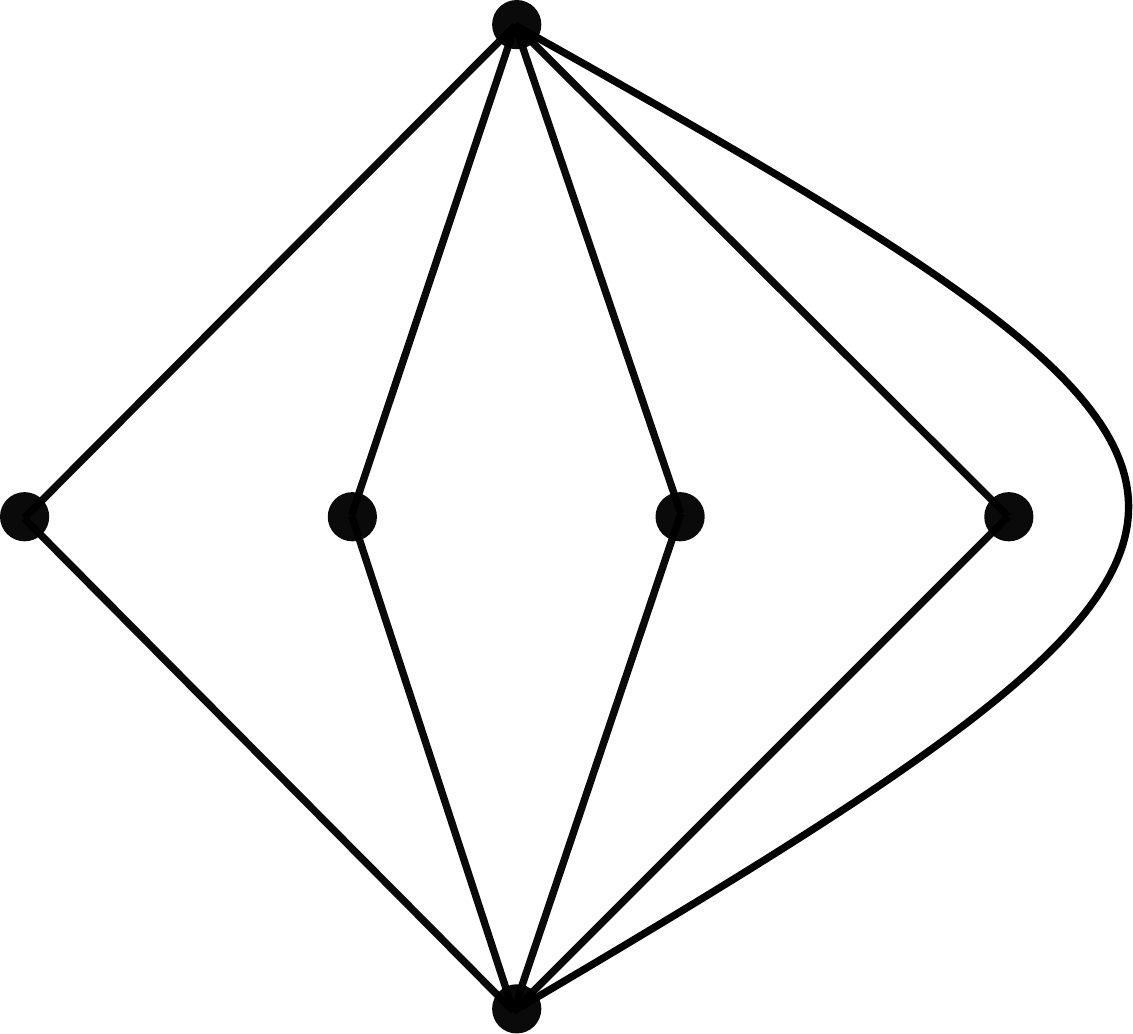}
\label{fig:clique_stability_proof_b}}
\caption{The graphs in (a) and (b) differ by a single edge.}
\label{fig:intro_shapes_single_component}
\end{center}
\end{figure}

Theorem \ref{th1} demonstrates that the addition of a single edge to the graph in Figure \ref{fig:clique_stability_proof_a} results in the persistence diagram changing by a distance of $\infty$. Note that, it is possible to construct a pair of graphs which differ by a single edge such that the corresponding $1$-dimensional persistence diagrams differ by an arbitrary number of points equalling $(0,\infty)$. This can be achieved through adding additional paths from the upper to lower vertices of the graph in Figure \ref{fig:clique_stability_proof_a}.

\begin{theorem}
\label{th2}
Let $G_1 = (V, E_1)$, $G_2 = (V, E_2)$ $\in \mathcal{G}^U_V$ be two unweighted graphs. Let $P_1$ and $P_2$ be persistent diagrams of a given dimension $p$ corresponding to $G_1$ and $G_2$ respectively and computed using the power complex filtration. The Bottleneck distance $B$ between $P_1$ and $P_2$ is not bounded above by a constant times the distance $D^U$ between the graphs $G_1$ and $G_2$; that is, $B(P_1, P_2)$ $\nleq$ $\alpha D^U(G_1, G_2)$.
\end{theorem}
 
\begin{proof}
We present a proof by counter example. Let $G_1$ and $G_2$ be the two graphs in Figures \ref{fig:power_stability_proof_a} and \ref{fig:power_stability_proof_b} respectively. These two graphs differ by a single edge and consequently $D^U(G_1, G_2)=1$. Let us denote the $0$-dimensional persistence diagrams corresponding to these graphs as $P_1$ and $P_2$ respectively which equal the sets $\lbrace (0,\infty), (0,\infty) \rbrace$ and $\lbrace (0,\infty) \rbrace$ respectively. The bottleneck distance between these persistence diagrams is $\infty$; that is, $B(P_1, P_2) = \infty$. Hence $B(P_1, P_2)$ $\nleq$ $\alpha D^U(G_1, G_2)$
\end{proof}

\begin{figure}
\begin{center}
\subfigure[]{\includegraphics[height=2.cm]{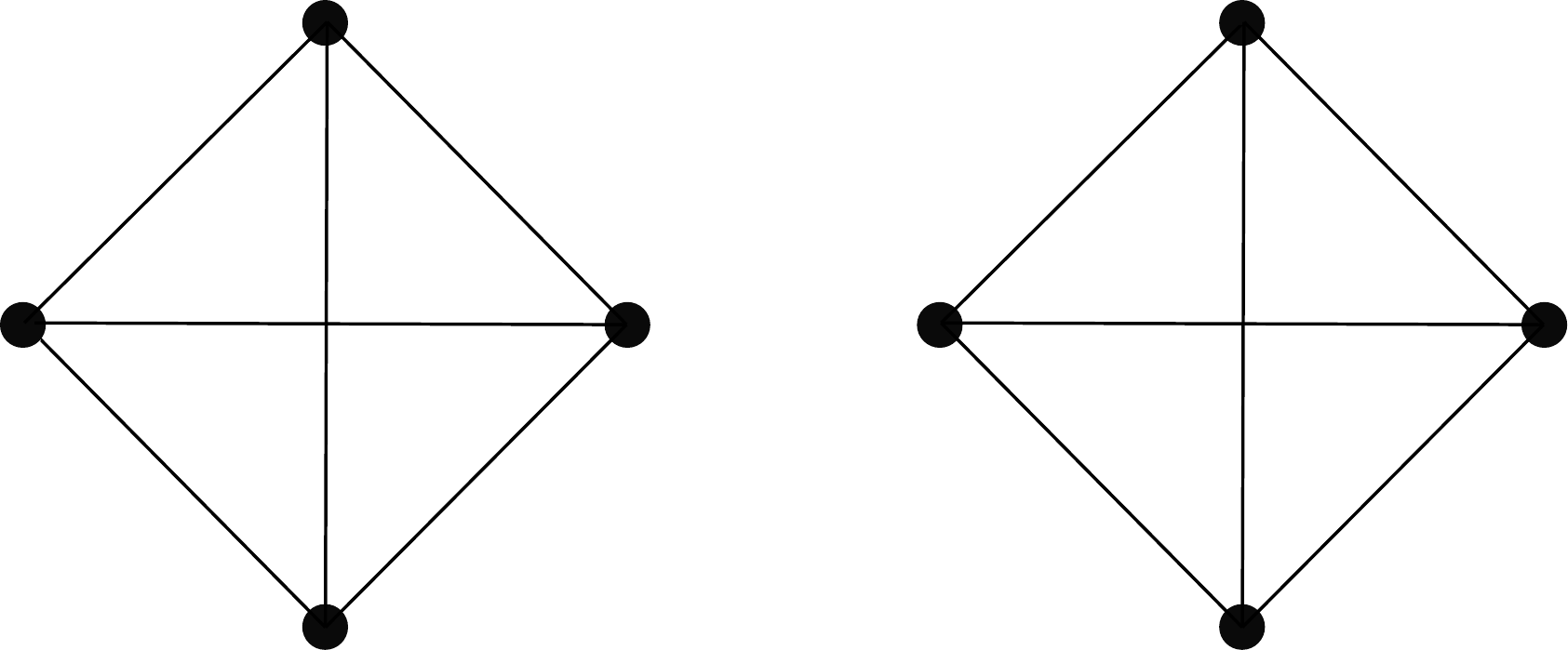}
\label{fig:power_stability_proof_a}}
\hspace{1cm}
\subfigure[]{\includegraphics[height=2.cm]{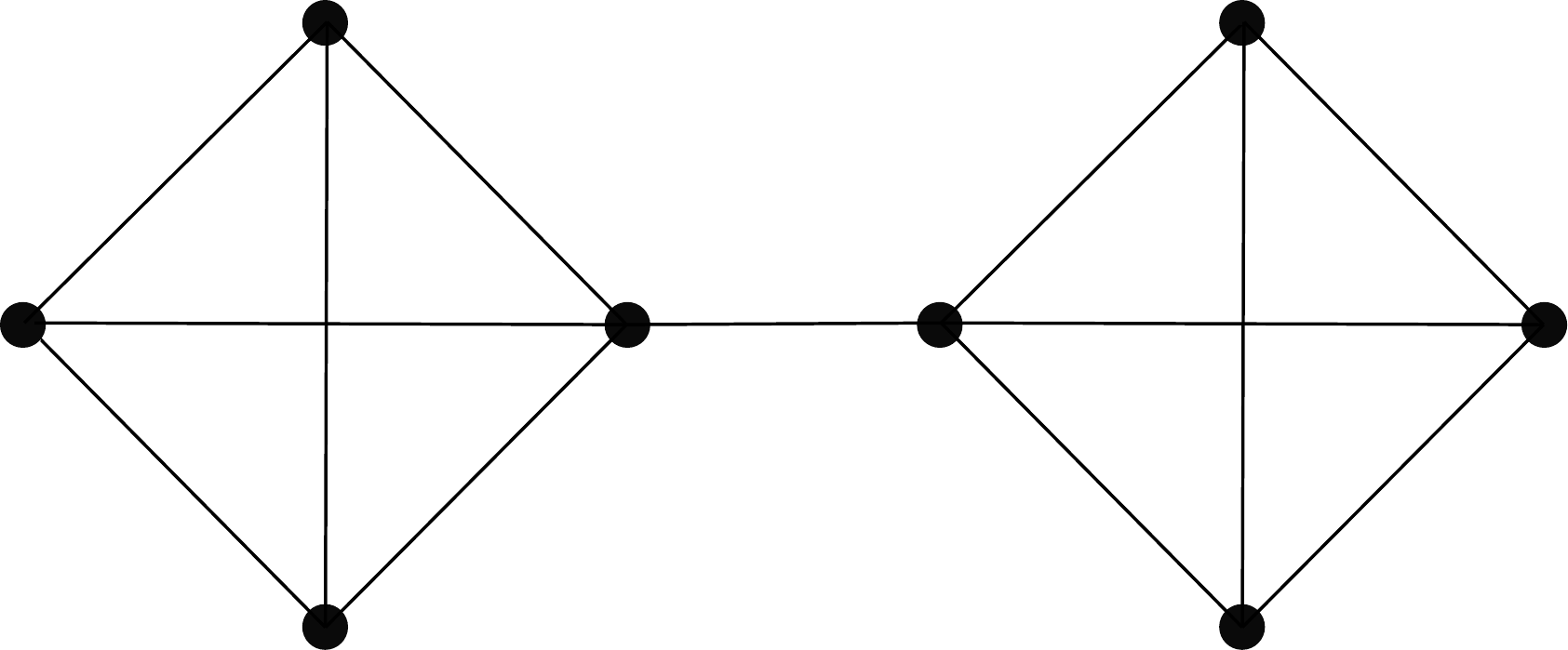}
\label{fig:power_stability_proof_b}}
\caption{The graphs in (a) and (b) differ by a single edge.}
\label{fig:intro_shapes_single_component}
\end{center}
\end{figure}

\subsection{Cliqueness Complex Filtration}
\label{sec:stability_cliqueness}
In this section we prove that the proposed topological graph analysis method which computes a \textit{cliqueness complex filtration} is stable with respect changes in graph connectivity whereby a small change in graph connectivity does not result in a large change in the corresponding persistence diagrams.

Let $\mathcal{G}^W_{V,E}$ be the space of weighted graphs which have an equal set of vertices $V$ and an equal set of edges $E$. We define a distance $D^W: \mathcal{G}^W_{V,E} \times \mathcal{G}^W_{V,E} \rightarrow [ 0,1 ]$ in Equation \ref{eq:weighted_graph_distance}.

\begin{equation}
\label{eq:weighted_graph_distance}
D^W((V, E, C_1), (V, E, C_2)) = \max_{e \in E} | C_1(e) - C_2(e) |
\end{equation}

\begin{lemma}
\label{lem1}
Let $G_1 = (V, E_1)$, $G_2 = (V, E_2)$ $\in \mathcal{G}^U_V$ be two unweighted graphs and let $G'_1 = (V, E, C_1)$, $G'_2 = (V, E, C_2)$ $\in \mathcal{G}^W_{V,E}$ respectively be the corresponding weighted graphs computed using Equations \ref{eq:weighted_complete_graph} and \ref{eq:edge_cliqueness}. The $D^W$ distance between the graphs $G'_1$ and $G'_2$ is bounded above by the distance $D^U$ between the graphs $G_1$ and $G_2$; that is, $D^W(G'_1, G'_2)$ $\leq$ $D^U(G_1, G_2)$.
\end{lemma}
 
\begin{proof}
The distance $D^U(G_1, G_2)$ takes the value $0$ or $1$. We consider each of these cases in turn. $D^U(G_1, G_2)$ takes the value $0$ if and only if $E_1 = E_2$. In this case $D^W(G'_1, G'_2) = 0$ and the inequality is satisfied. $D^U(G_1, G_2)$ takes the value $1$ if and only if $E_1 \neq E_2$. In this case $D^W(G'_1, G'_2)$ takes a value in the range $[0,1]$ and the inequality is satisfied.
\end{proof}

Let $\mathbb{R}^K$ denote the space of real valued monotonic functions on a simplicial complex $K$. We define the distance $D^K:\mathbb{R}^K \times \mathbb{R}^K \rightarrow \mathbb{R}$ in Equation \ref{eq:simplicial_distance} between two elements in this space.

\begin{equation}
\label{eq:simplicial_distance}
D^K(f,g) = \max_{\sigma \in K} | f(\sigma) - g(\sigma) |
\end{equation}

\begin{lemma}
\label{lem2}
Let $G_1 = (V, E, C_1)$, $G_2 = (V, E, C_2)$ $\in \mathcal{G}^W_{V,E}$ be two weighted graphs. Let $K$ be the clique complex corresponding to the graph $(V,E)$. Let $f_1, f_2 \in \mathbb{R}^K$ be the real valued monotonic functions defined on $K$ by Equation \ref{eq:clique_filtration_f} and corresponding to $G_1$ and $G_2$ respectively. The distance $D^K$ between $f_1$ and $f_2$ is bounded above by the distance $D^W$ between $G_1$ and $G_2$; that is, $D^K(f,g)$ $\leq$ $D^W(G_1, G_2)$.
\end{lemma}
 
\begin{proof}
If the dimension of a simplex is $1$ the bijection $b$ between simplices and graph edges means that the difference between $f_1$ and $f_2$ on this simplex equals the difference between $C_1$ and $C_2$ on the corresponding edge. Therefore, in this case the inequality is satisfied.

If the dimension of a simplex is $0$ the difference between $f_1$ and $f_2$ on this simplex equals the difference between the maximum of $f_1$ and $f_2$ on a subset of simplices of dimension $1$. The inequality is satisfied for simplices of dimension $1$ and therefore is also satisfied in this case.

If the dimension of a simplex is $\geq 2$ the difference between $f_1$ and $f_2$ on this simplex equals the difference between the minimum of $f_1$ and $f_2$ on a subset of simplices of dimension $1$. The inequality is satisfied for simplices of dimension $1$ and therefore is also satisfied in this case.
\end{proof}

\begin{lemma}
\label{lem3}
Let $K$ be a simplicial complex and let $f_1, f_2 \in \mathbb{R}^K$ be two real valued monotonic functions on $K$. Let $P_1$ and $P_2$ be the persistent diagrams of a given dimension $p$ corresponding to the filtrations of $K$ induced by the functions $f_1$ and $f_2$ respectively. The Bottleneck distance $B$ between $P_1$ and $P_2$ is bounded above by the distance $D^K$ between the function $f_1$ and $f_2$; that is $B(P_1, P_2)$ $\leq$ $D^K(f,g)$.
\end{lemma}
 
\begin{proof}
This is a classical stability theorem and a proof can be found in \cite{edelsbrunner2010computational}.
\end{proof}

\begin{theorem}
Let $G_1 = (V, E_1)$, $G_2 = (V, E_2)$ $\in \mathcal{G}^U_V$ be two unweighted graphs. Let $P_1$ and $P_2$ be the persistent diagrams of a given dimension $p$ computed from $G_1$ and $G_2$ using the method proposed in this article. The Bottleneck distance $B$ between $P_1$ and $P_2$ is bounded above by the distance $D^U$ distance between the graphs $G_1$ and $G_2$; that is, $B(P_1, P_2)$ $\leq$ $D^U(G_1, G_2)$.
\end{theorem}
 
\begin{proof}
Combining Lemmas \ref{lem1}, \ref{lem2} and \ref{lem3} gives $B(P_1, P_2)$ $\leq$  $D^K(f,g)$ $\leq$ $D^W(G'_1, G'_2)$ $\leq$ $D^U(G_1, G_2)$ from which implies $B(P_1, P_2)$ $\leq$ $D^U(G_1, G_2)$.
\end{proof}

\section{Experimental Analysis}
\label{sec:applications}
In this section we present an experimental evaluation of the proposed and existing methods for topological graph analysis. In subsections \ref{sec:applications:synthetic} and \ref{sec:applications:real} we present an experimental evaluation of these methods with respect to random and real graphs respectively.

\subsection{Analysis of Random Graphs}
\label{sec:applications:synthetic}
A community in a graph is a subset of vertices which are densely connected. The \textit{stochastic block model} is a generative model for undirected and unweighted graphs where the graphs generated have a high probability of containing communities \cite{karrer2011stochastic}. The first random graphs we considered were two graphs containing $1000$ vertices and sampled from a stochastic block model. This sampling was performed using the \textit{graph-tool} software implementation \cite{peixoto2019bayesian}. The graphs in question contain one and four communities, and are displayed in Figures \ref{fig:erdos} and \ref{fig:blockmodel} respectively.

The $0$- and $1$-dimensional persistent diagrams corresponding to these graphs and computing using the proposed method are displayed in Figures \ref{fig:erdos_cliqueness_PD} and \ref{fig:blockmodel_cliqueness_PD} respectively. For each graph the corresponding $0$-dimensional persistent diagram contains a point of infinite persistence plus a set of points of finite persistence. The $0$-dimensional persistent diagram corresponding to the graph in Figure \ref{fig:erdos} contains no points of significant persistence other than the point with infinite persistence. On the other hand, the $0$-dimensional persistent diagram corresponding to the graph in Figure \ref{fig:blockmodel} contains three points of significant persistence other than the point with infinite persistence. This demonstrates that the proposed method accurately models the existence of one and four communities in the graphs and in turn accurately discriminates between the graphs.

The $0$- and $1$-dimensional persistent diagrams corresponding to the graphs in Figures \ref{fig:erdos} and \ref{fig:blockmodel} and computing using the method which computes a \textit{clique complex filtration} are displayed in Figures \ref{fig:erdos_clique_PD} and \ref{fig:blockmodel_clique_PD} respectively. For each graph the corresponding $0$-dimensional persistent diagram contains a point of finite persistence for each individual graph vertex plus one point of infinite persistence. This demonstrates that this method does not accurately discriminate between the graphs.

The $0$- and $1$-dimensional persistent diagrams corresponding to the graphs in Figures \ref{fig:erdos} and \ref{fig:blockmodel} and computing using the method which computes a \textit{power complex filtration} are displayed in Figures \ref{fig:erdos_rips_PD} and \ref{fig:blockmodel_rips_PD} respectively. Similar to the previous discussion, for each graph the corresponding $0$-dimensional persistent diagram contains a point of finite persistence for each individual graph vertex plus one point of finite persistence. This demonstrates that this method also does not accurately discriminate between the graphs.

\begin{figure}
\begin{center}
\subfigure[]{\includegraphics[height=3cm]{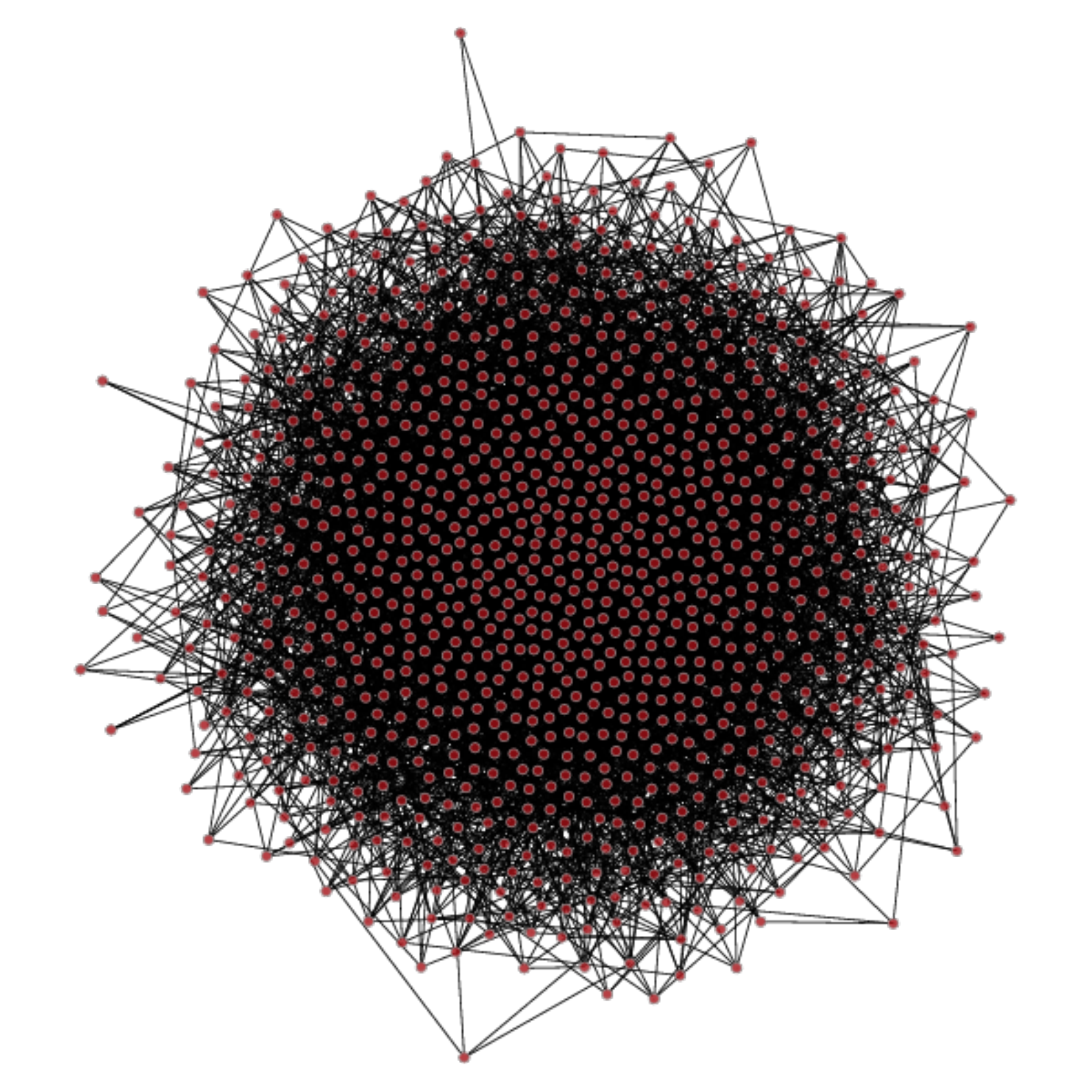}
\label{fig:erdos}}
\hspace{.1cm}
\subfigure[]{\includegraphics[height=2.5cm]{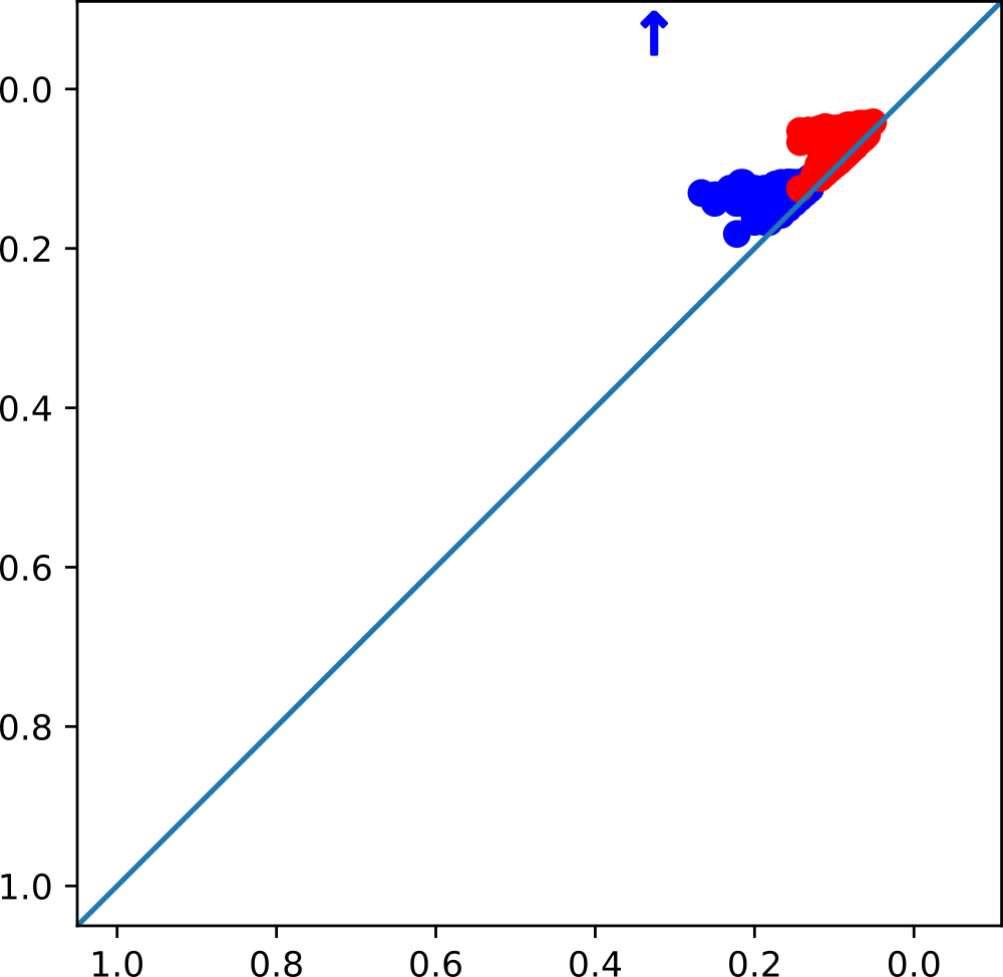}
\label{fig:erdos_cliqueness_PD}}
\hspace{.1cm}
\subfigure[]{\includegraphics[height=2.5cm]{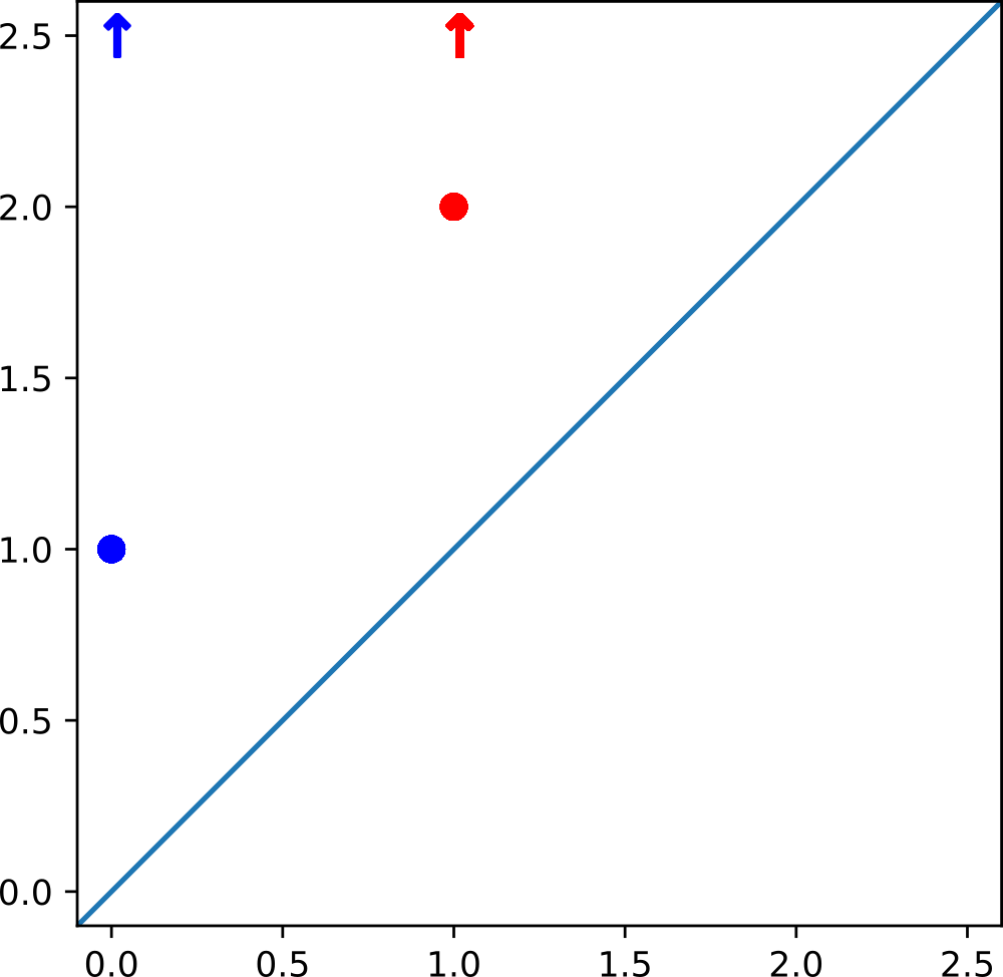}
\label{fig:erdos_clique_PD}}
\hspace{.1cm}
\subfigure[]{\includegraphics[height=2.5cm]{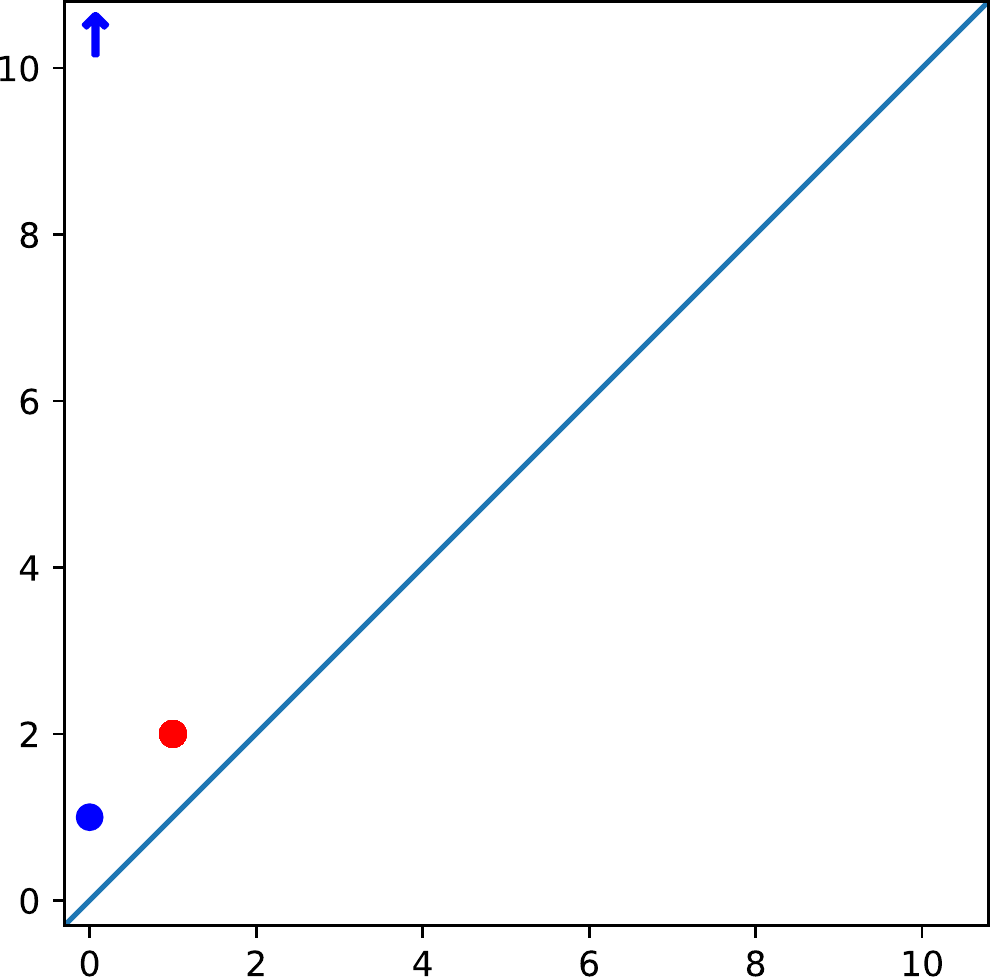}
\label{fig:erdos_rips_PD}}
\\
\subfigure[]{\includegraphics[height=3cm]{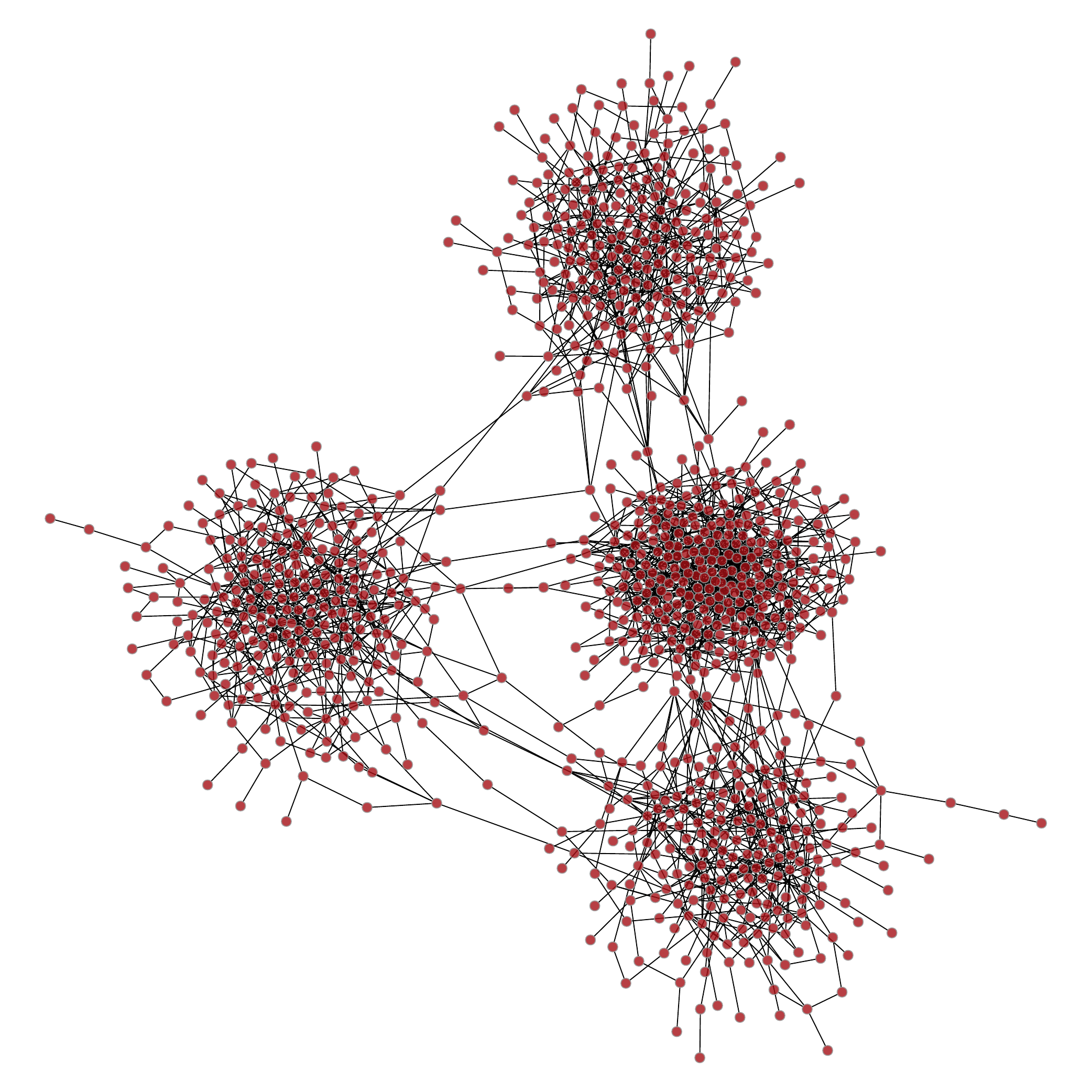}
\label{fig:blockmodel}}
\hspace{.1cm}
\subfigure[]{\includegraphics[height=2.5cm]{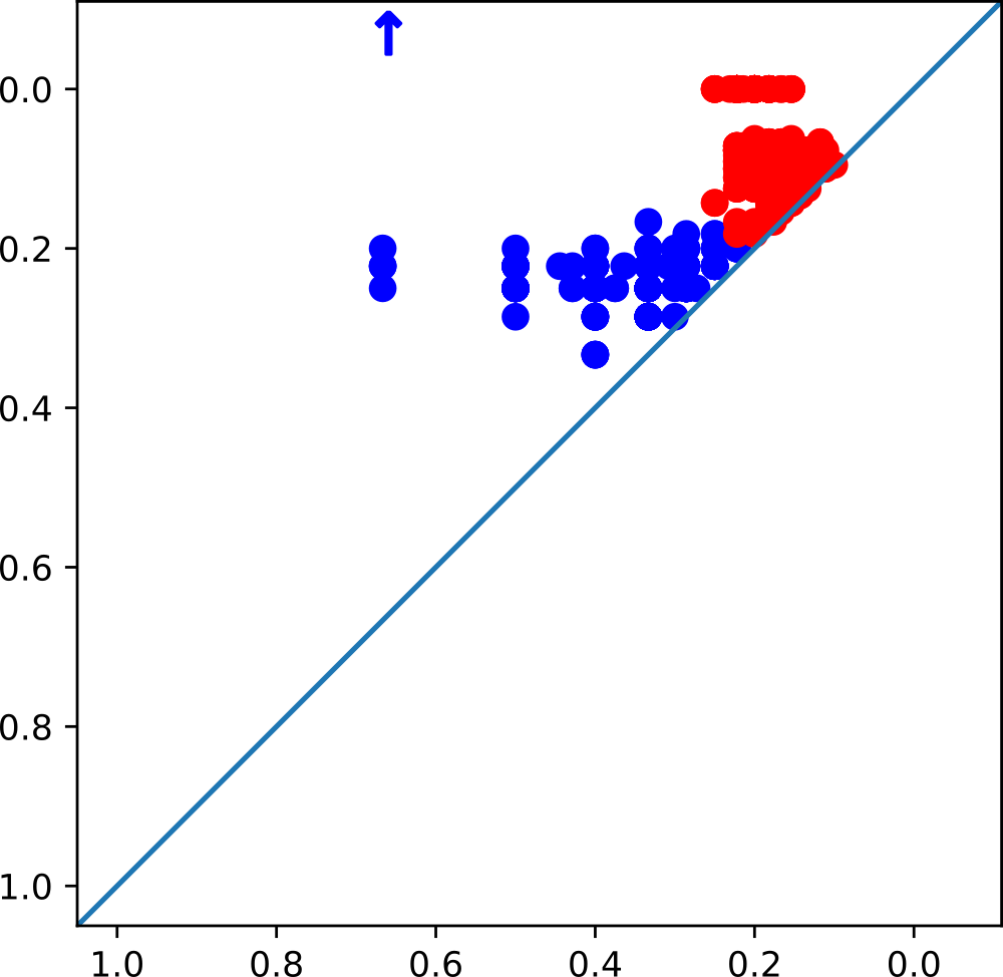}
\label{fig:blockmodel_cliqueness_PD}}
\hspace{.1cm}
\subfigure[]{\includegraphics[height=2.5cm]{images/blockmodel_clique_PD}
\label{fig:blockmodel_clique_PD}}
\hspace{.1cm}
\subfigure[]{\includegraphics[height=2.5cm]{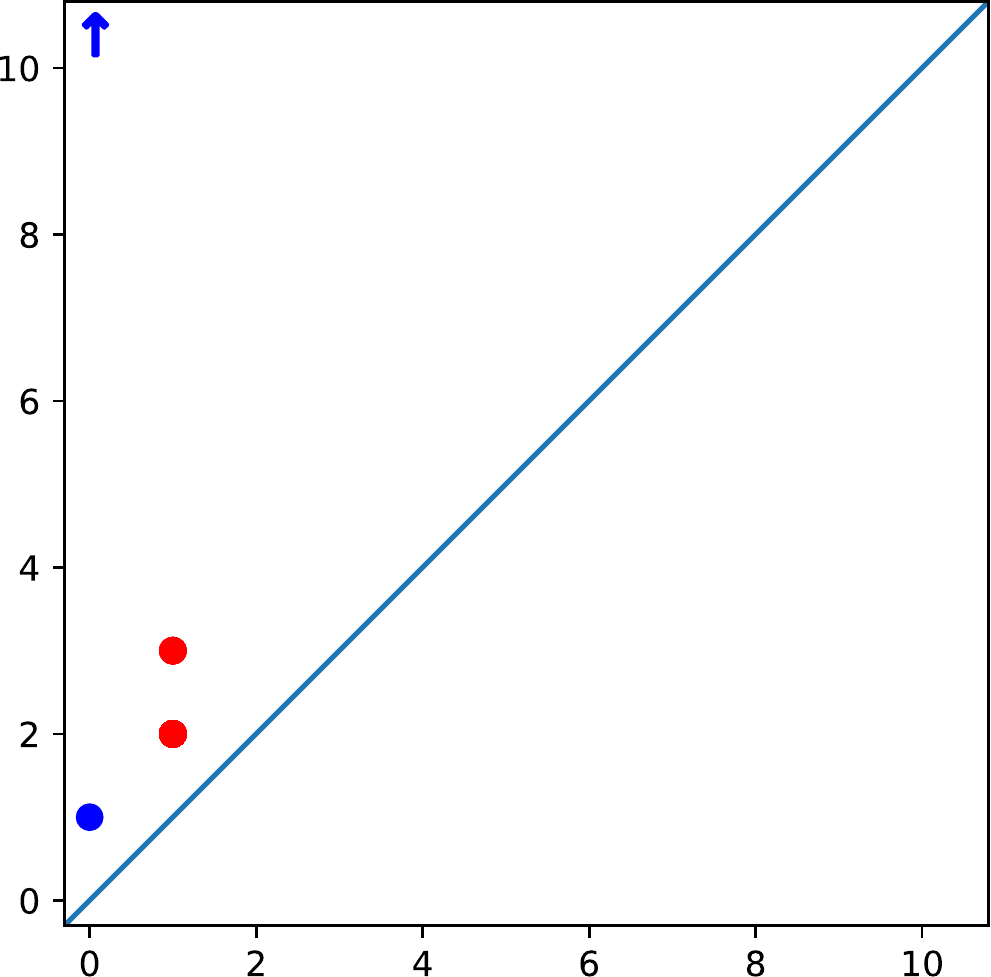}
\label{fig:blockmodel_rips_PD}}
\caption{The graphs in (a) and (e) were sampled from a stochastic block model and contain one and four communities respectively. The $0$- and $1$-dimensional persistent diagrams computed using the proposed method are displayed in (b) and (f) respectively. The $0$- and $1$-dimensional persistent diagrams computed using a \textit{clique complex filtration} are displayed in (c) and (g) respectively. The $0$- and $1$-dimensional persistent diagrams computed using a \textit{power complex filtration} are displayed in (d) and (h) respectively.}
\label{fig:random_community_graphs}
\end{center}
\end{figure}

A random geometric graph is a graph where the vertices in question are random points in a metric space and two vertices are connected by an edge if and only if the distance between the vertices in question is less than a specified value. The second random graphs we considered were two random geometric graphs. The first graph was constructed by randomly sampling $200$ vertices on the unit circle embedded in the $\mathbb{R}^2$ and connecting two vertices by an edge if and only if the distance between the corresponding vertices was less than or equal to $0.25$. The graph in question is displayed in Figure \ref{fig:geometric_graph_circle_no_cross}. The mean vertex degree of this graph is $16.6$ with a standard deviation of  $0.3$. This demonstrates that the graph consists of a single densely connected cycle. The second random geometric graph equals the previous graph plus an additional edge between two randomly selected vertices. The graph in question is displayed in Figure \ref{fig:geometric_graph_circle_cross}. This additional edge creates an additional cycle in the graph which is less densely connected than the previous cycle.

The $0$- and $1$-dimensional persistent diagrams corresponding to these graphs and computing using the proposed method are displayed in Figures \ref{fig:geometric_graph_circle_no_cross_cliqueness} and \ref{fig:geometric_graph_circle_cross_cliqueness} respectively. For each graph the corresponding $1$-dimensional persistent diagram contains one and two points of finite persistence respectively. The additional point in the latter persistent diagram has relatively small persistence and is a consequence of the less densely connected cycle. This accurately models the fact that the corresponding cycle is less significant and demonstrates the stability of the method.

The $0$- and $1$-dimensional persistent diagrams corresponding to the graphs in Figures \ref{fig:geometric_graph_circle_no_cross} and \ref{fig:geometric_graph_circle_cross} and computing using the method which computes a \textit{clique complex filtration} are displayed in Figures \ref{fig:geometric_graph_circle_no_cross_clique} and \ref{fig:geometric_graph_circle_cross_clique} respectively. For each graph the corresponding $1$-dimensional persistent diagram contains one and two points of infinite persistence respectively. The additional point in the latter persistent diagram is a consequence of the less densely connected cycle. The infinite persistence of this point does not accurately model the fact that the corresponding cycle is less significant and demonstrates the instability of the method.

The $0$- and $1$-dimensional persistent diagrams corresponding to the graphs in Figures \ref{fig:geometric_graph_circle_no_cross} and \ref{fig:geometric_graph_circle_cross} and computing using the method which computes a \textit{power complex filtration} are displayed in Figures \ref{fig:geometric_graph_circle_no_cross_power} and \ref{fig:geometric_graph_circle_cross_power} respectively. For each graph the corresponding $1$-dimensional persistent diagram contains one and two points of finite persistence respectively. The two points in the latter diagram have equal coordinates and in turn persistence; this does not accurately model the fact that one cycle is less significant and demonstrates the instability of the method.

\begin{figure}
\begin{center}
\subfigure[]{\includegraphics[height=2.5cm]{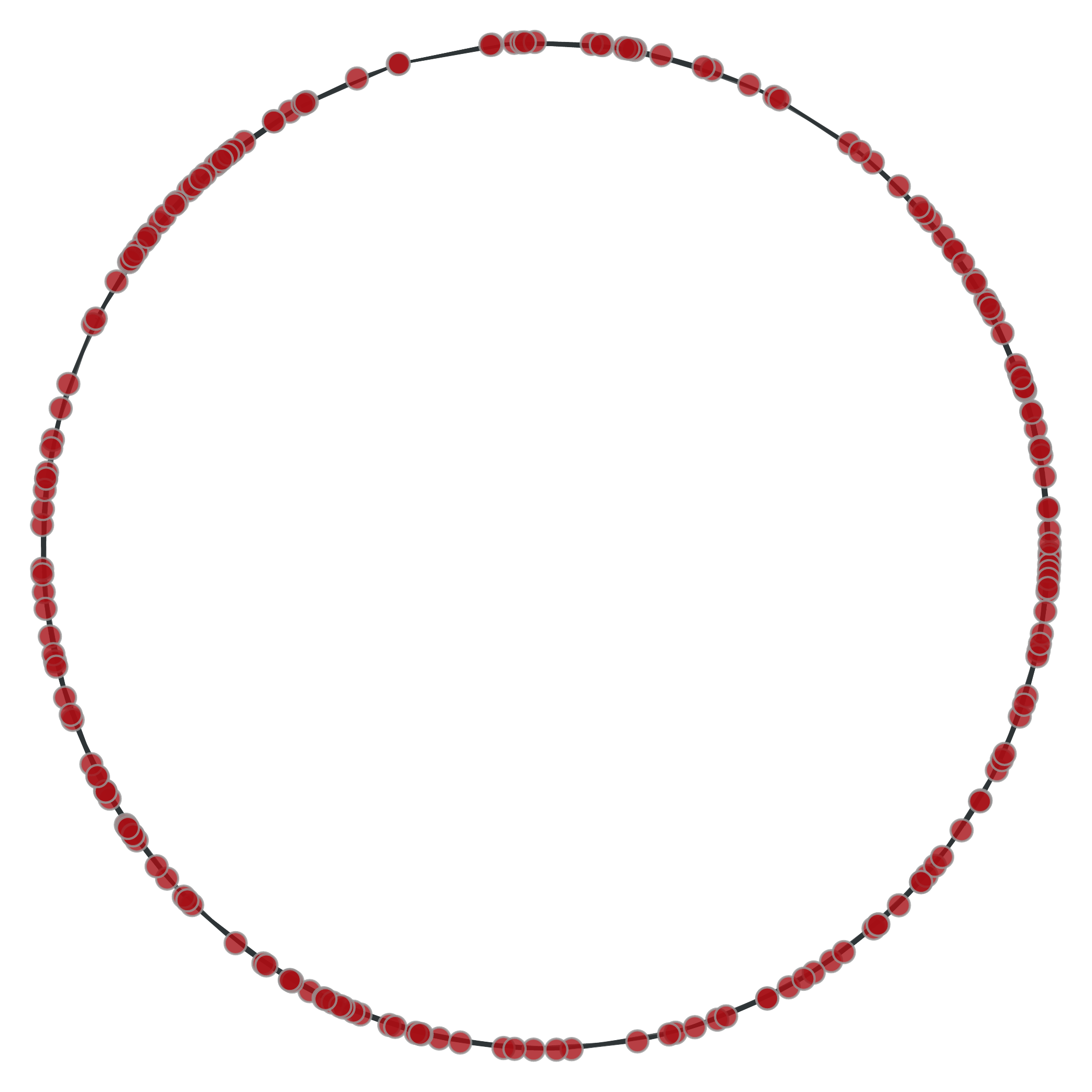}
\label{fig:geometric_graph_circle_no_cross}}
\hspace{.1cm}
\subfigure[]{\includegraphics[height=2.5cm]{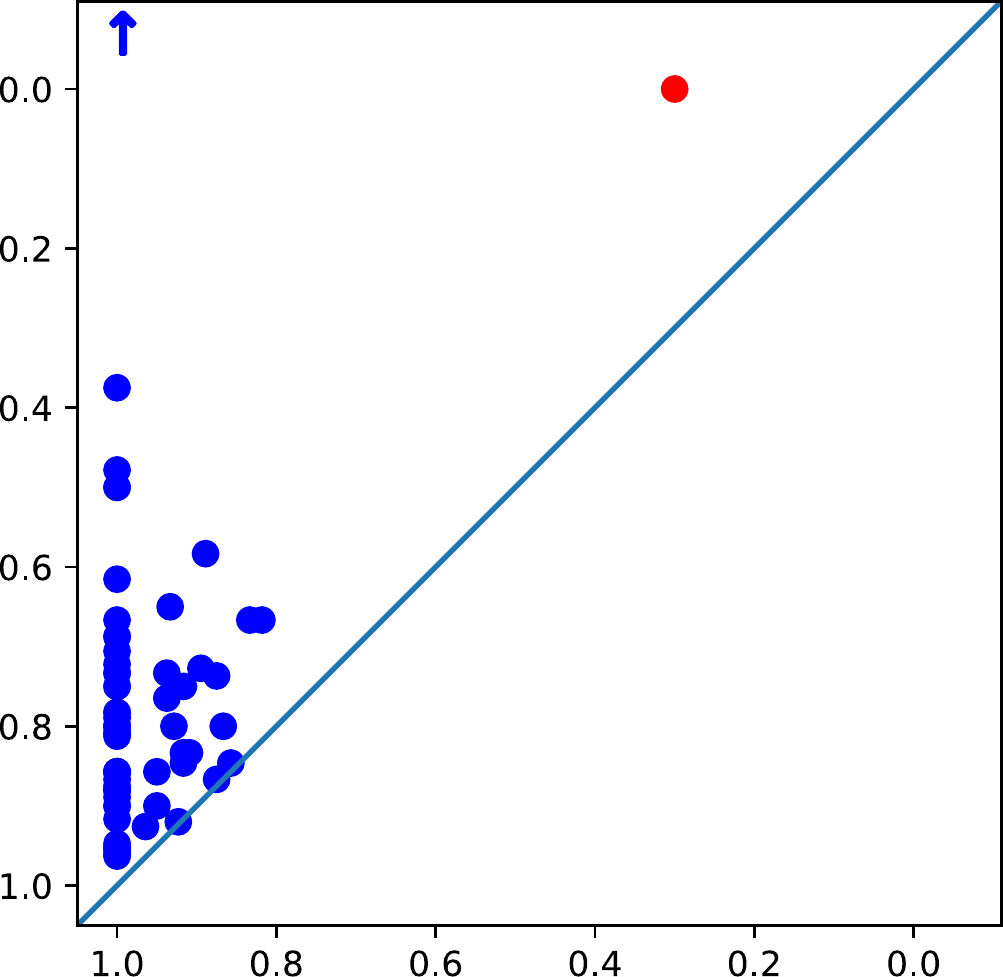}
\label{fig:geometric_graph_circle_no_cross_cliqueness}}
\hspace{.1cm}
\subfigure[]{\includegraphics[height=2.5cm]{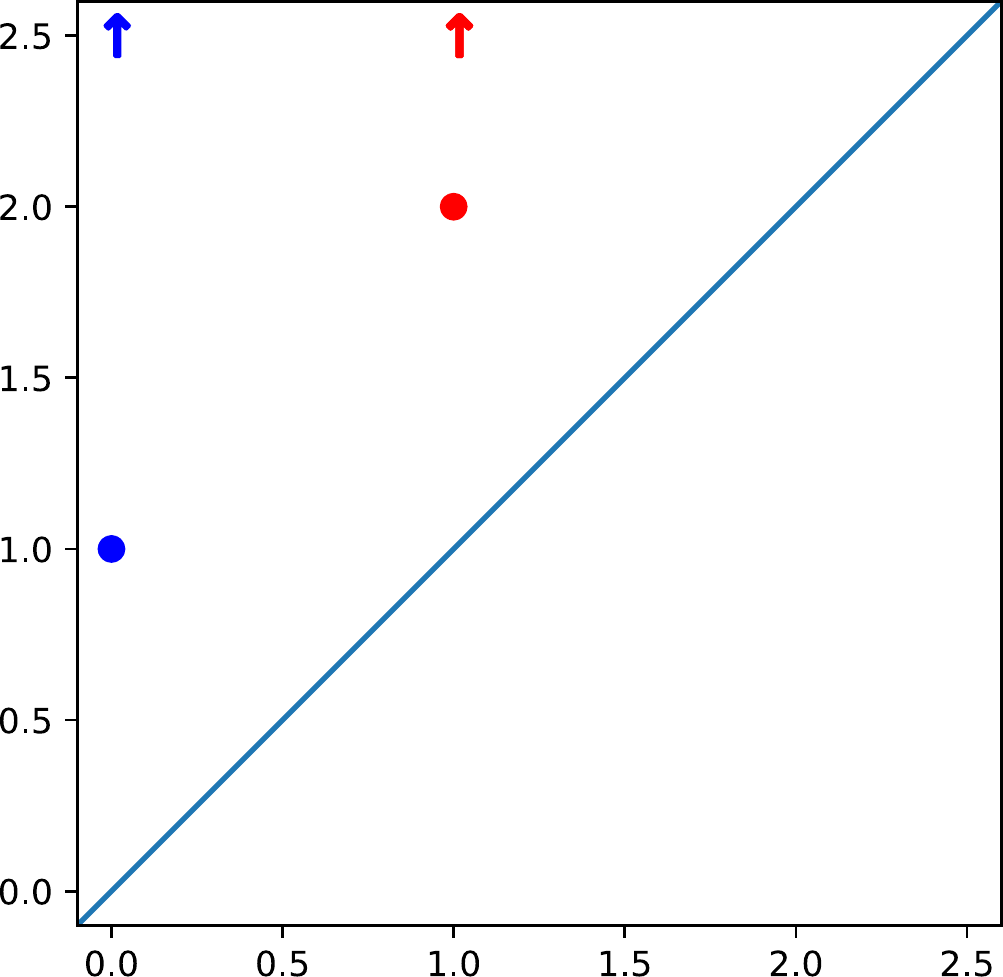}
\label{fig:geometric_graph_circle_no_cross_clique}}
\hspace{.1cm}
\subfigure[]{\includegraphics[height=2.5cm]{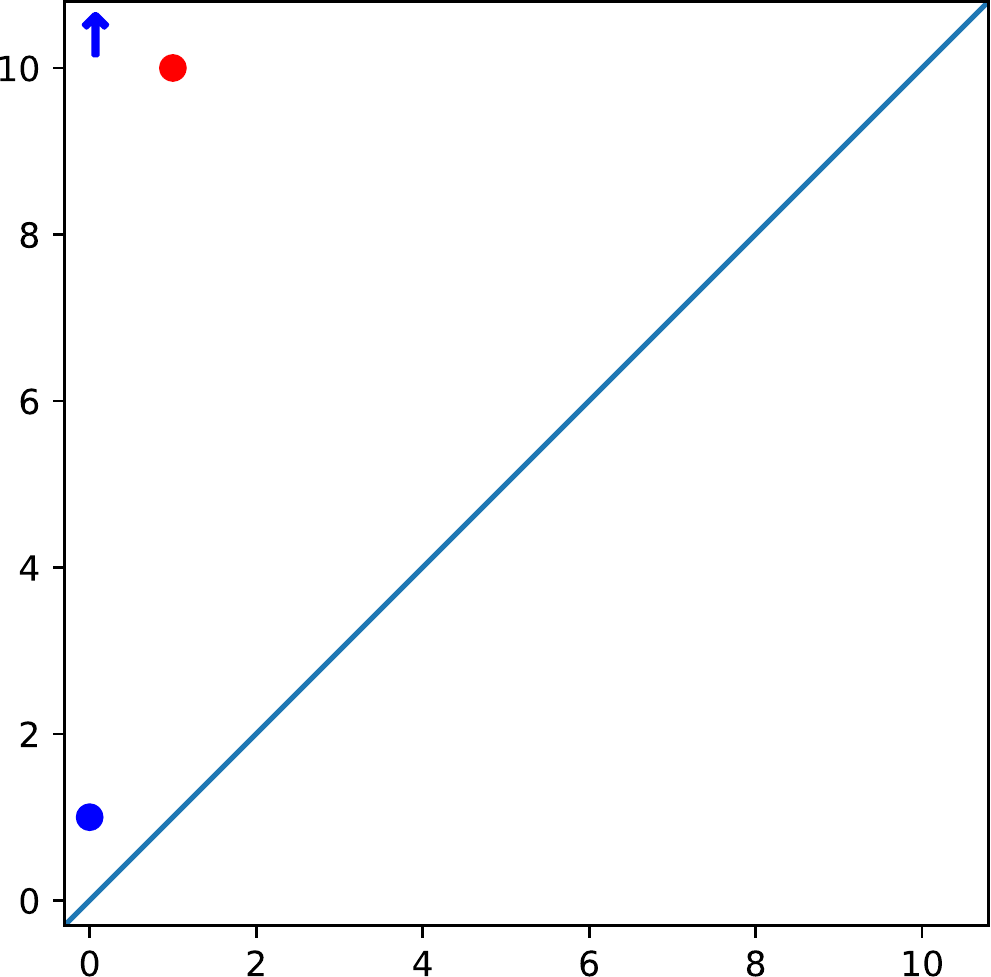}
\label{fig:geometric_graph_circle_no_cross_power}}
\\
\subfigure[]{\includegraphics[height=2.5cm]{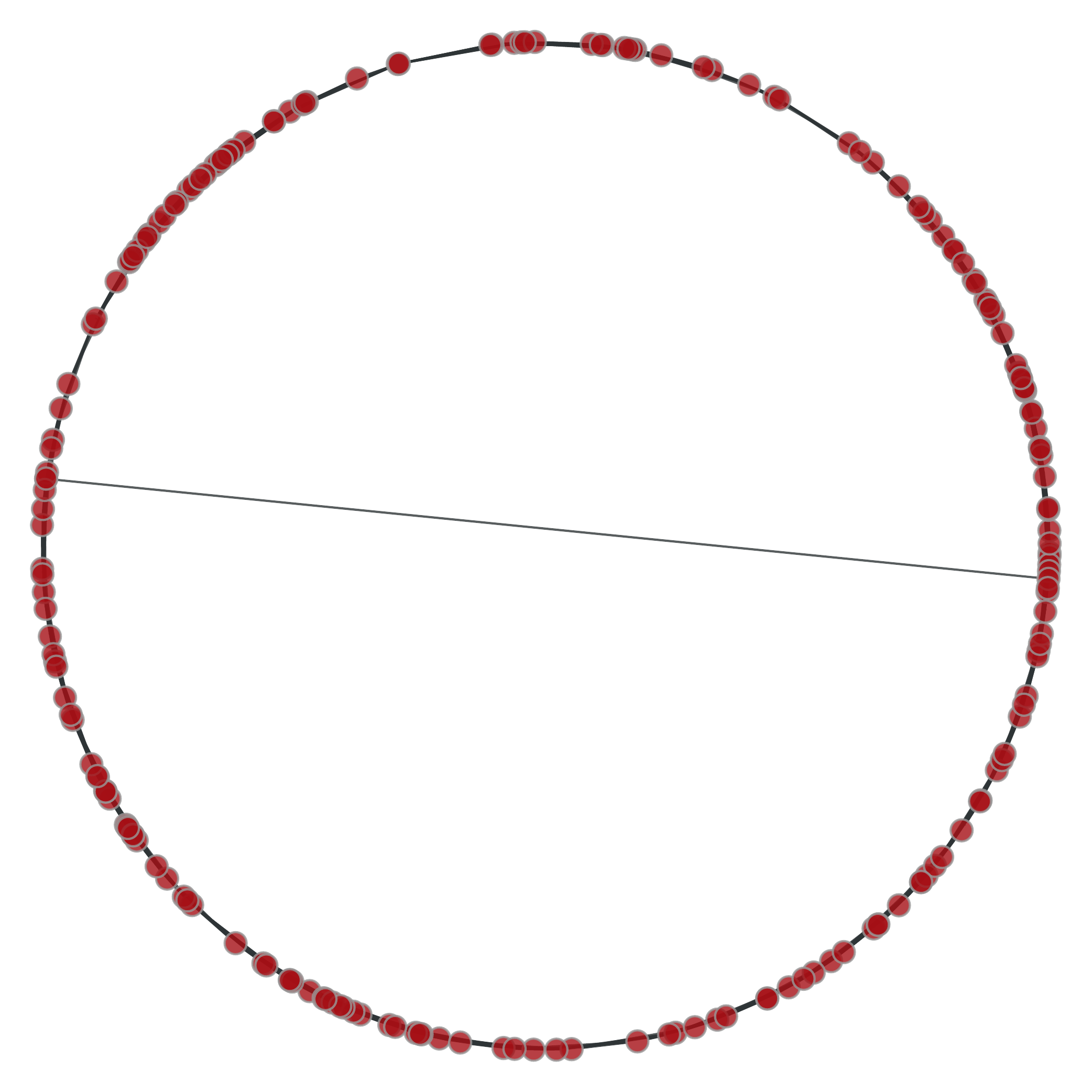}
\label{fig:geometric_graph_circle_cross}}
\hspace{.1cm}
\subfigure[]{\includegraphics[height=2.5cm]{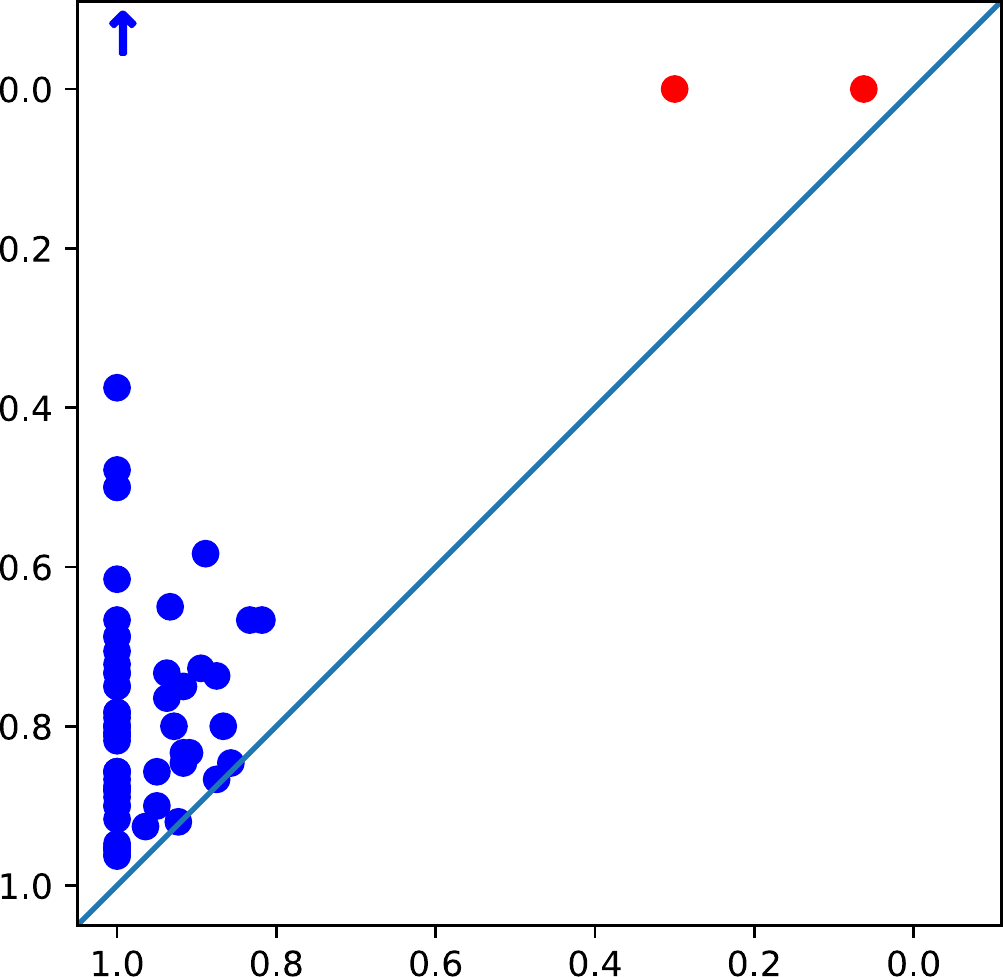}
\label{fig:geometric_graph_circle_cross_cliqueness}}
\hspace{.1cm}
\subfigure[]{\includegraphics[height=2.5cm]{images/geometric_graph_circle_cross_clique}
\label{fig:geometric_graph_circle_cross_clique}}
\hspace{.1cm}
\subfigure[]{\includegraphics[height=2.5cm]{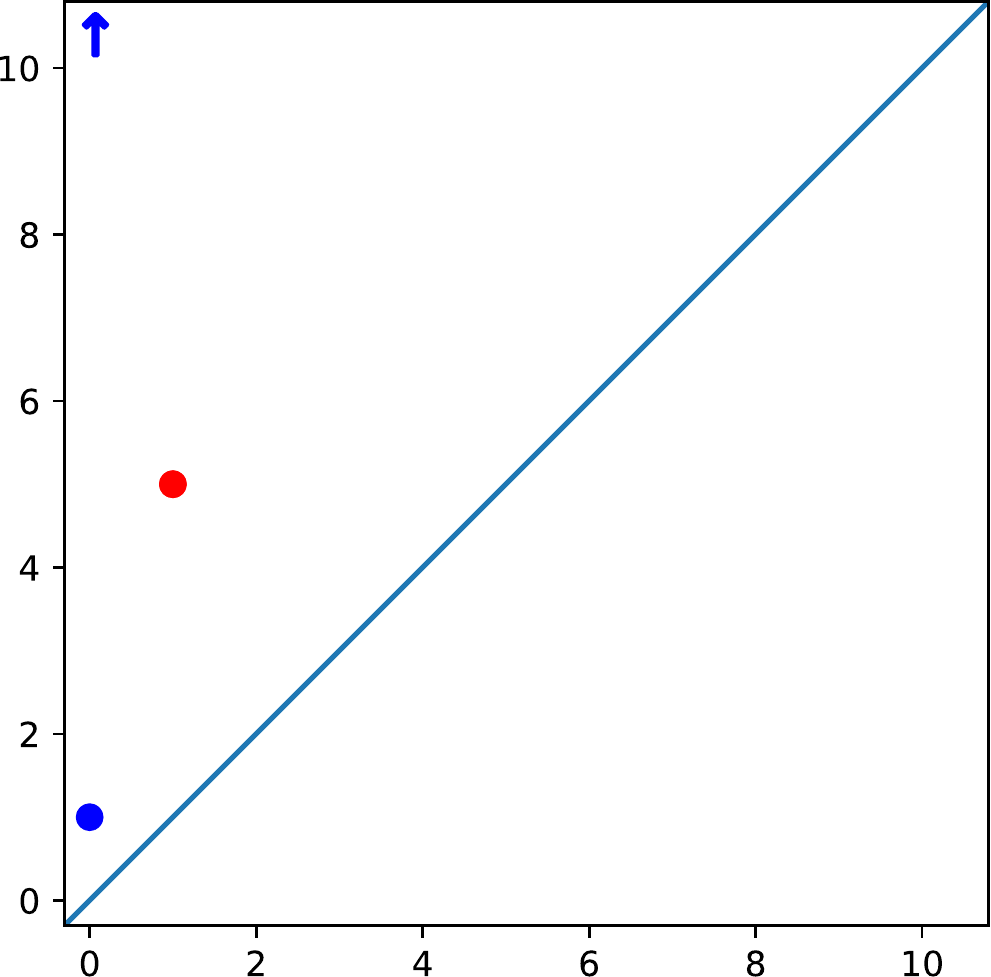}
\label{fig:geometric_graph_circle_cross_power}}
\caption{The graph in (a) is a random geometric graph in $\mathbb{R}^2$ where the vertices are randomly selected points on the unit circle. The graph in (e) is equal to that in (a) but an additional random edge. The $0$- and $1$-dimensional persistent diagrams computed using the proposed method are displayed in (b) and (f) respectively. The $0$- and $1$-dimensional persistent diagrams computed using a \textit{clique complex filtration} are displayed in (c) and (g) respectively. The $0$- and $1$-dimensional persistent diagrams computed using a \textit{power complex filtration} are displayed in (d) and (h) respectively.}
\label{fig:geometric_random_graphs}
\end{center}
\end{figure}

\subsection{Analysis of Real Graphs}
\label{sec:applications:real}
In our analysis we considered three real graphs. The first graph is the Zachary karate club social network \cite{zachary1977information}. In this graph a vertex corresponds to a member of a university karate club and an edge corresponds to a tie between two members. The graph contains 34 vertices and 78 edges. The second graph is a social network of bottlenose dolphins \cite{lusseau2003bottlenose}. In this graph a vertex corresponds to a bottlenose dolphin living of Doubtful Sound, a fjord in New Zealand, and observed between 1994 and 2001. An edge corresponds to a frequent association between two bottlenose dolphins. The graph contains 62 vertices and 159 edges. The third graph is a protein interaction network contained in yeast \cite{konect:coulomb2005}. In this graph a vertex corresponds to a protein and an edge corresponds to a metabolic interaction between two proteins. In our analysis we considered the largest connected component in the graph which contains 1,458 vertices and 1,993 edges. The three graphs described above are displayed in Figures \ref{fig:karate_network}, \ref{fig:dolphins_network} and \ref{fig:protein_network} respectively.

The $0$- and $1$-dimensional persistent diagrams corresponding to these graphs and computing using the proposed method are displayed in Figures \ref{fig:karate_cliqueness}, \ref{fig:dolphins_cliqueness} and \ref{fig:protein_cliqueness} respectively. For each graph the corresponding $0$-dimensional persistent diagram contains a number of points with significant persistence indicating the existence of multiple community structures. It is well known that the first two graphs contain such structures and in fact these graphs are commonly used to evaluate community detection algorithms. For the last graph the corresponding $1$-dimensional persistent diagram contains a number of points with significant persistence indicating the existence of densely connected cycles in the graph.

The $0$- and $1$-dimensional persistent diagrams corresponding to the graphs in Figures \ref{fig:karate_network}, \ref{fig:dolphins_network} and \ref{fig:protein_network} and computing using the method which computes a \textit{clique complex filtration} are displayed in Figures \ref{fig:karate_clique}, \ref{fig:dolphins_clique} and \ref{fig:protein_clique} respectively. For each graph the corresponding $0$-dimensional persistent diagram contains one point of significant persistence which does not indicate the existence of multiple communities in the graph. As demonstrated by our analysis of the graphs in Figures \ref{fig:random_community_graphs} this method does not accurately model communities given edges between communities. Consequently this does not mean that multiple communities do not exist in the graph. For each graph the corresponding $1$-dimensional persistent diagrams contain multiple points of significant persistence. As demonstrated by our analysis of the graphs in Figure \ref{fig:geometric_random_graphs}, this method does not discriminate between sparsely and densely connected cycles in the graph. Consequently this  means we cannot determine if the cycles in question are densely connected and in turn significant.

The $0$- and $1$-dimensional persistent diagrams corresponding to the graphs in Figures \ref{fig:karate_network}, \ref{fig:dolphins_network} and \ref{fig:protein_network} and computing using the method which computes a \textit{clique complex filtration} are displayed in Figures \ref{fig:karate_rips}, \ref{fig:dolphins_rips} and \ref{fig:protein_rips} respectively. For each graph the corresponding $0$-dimensional persistent diagram contains one point of significant persistence which does not indicate the existence of multiple communities in the graph. For the same reasons as above this does not mean that multiple communities do not exist in the graph. For the graph in Figure \ref{fig:protein_network} the corresponding $1$-dimensional persistent diagram contains multiple points of significant persistence. For the same reasons as above we cannot determine if the cycles in question are densely connected and in turn significant.

\begin{figure}
\begin{center}
\subfigure[]{\includegraphics[height=2.3cm]{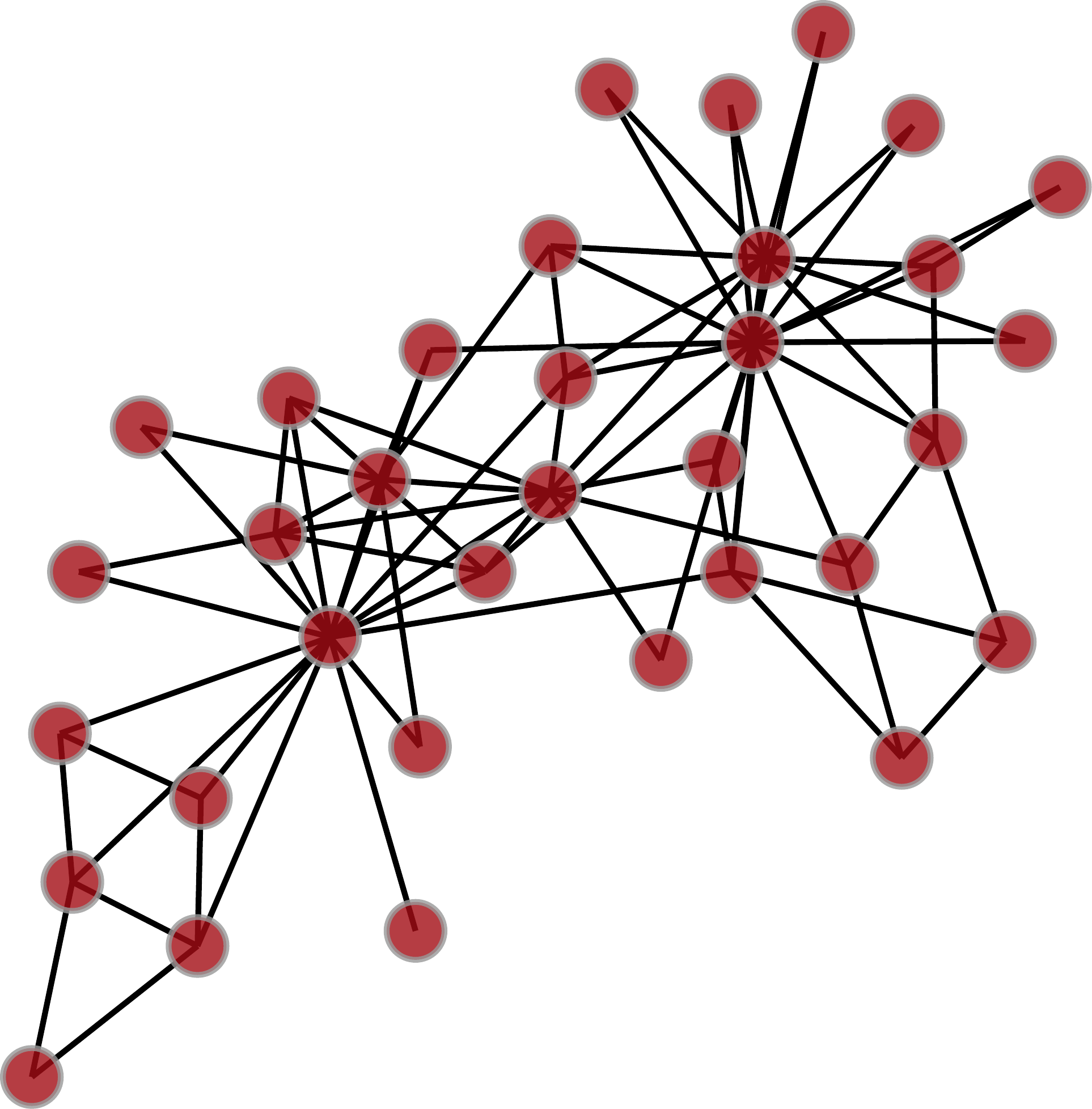}
\label{fig:karate_network}}
\hspace{.1cm}
\subfigure[]{\includegraphics[height=2.5cm]{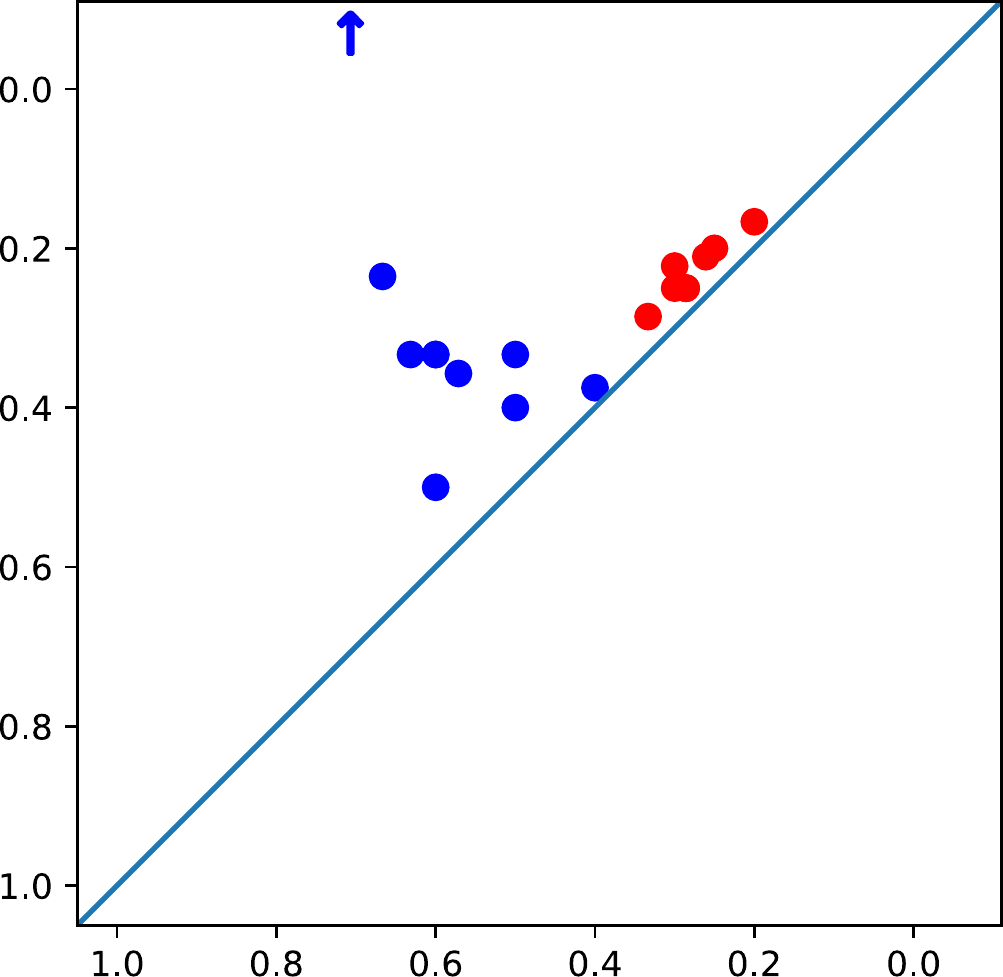}
\label{fig:karate_cliqueness}}
\hspace{.1cm}
\subfigure[]{\includegraphics[height=2.5cm]{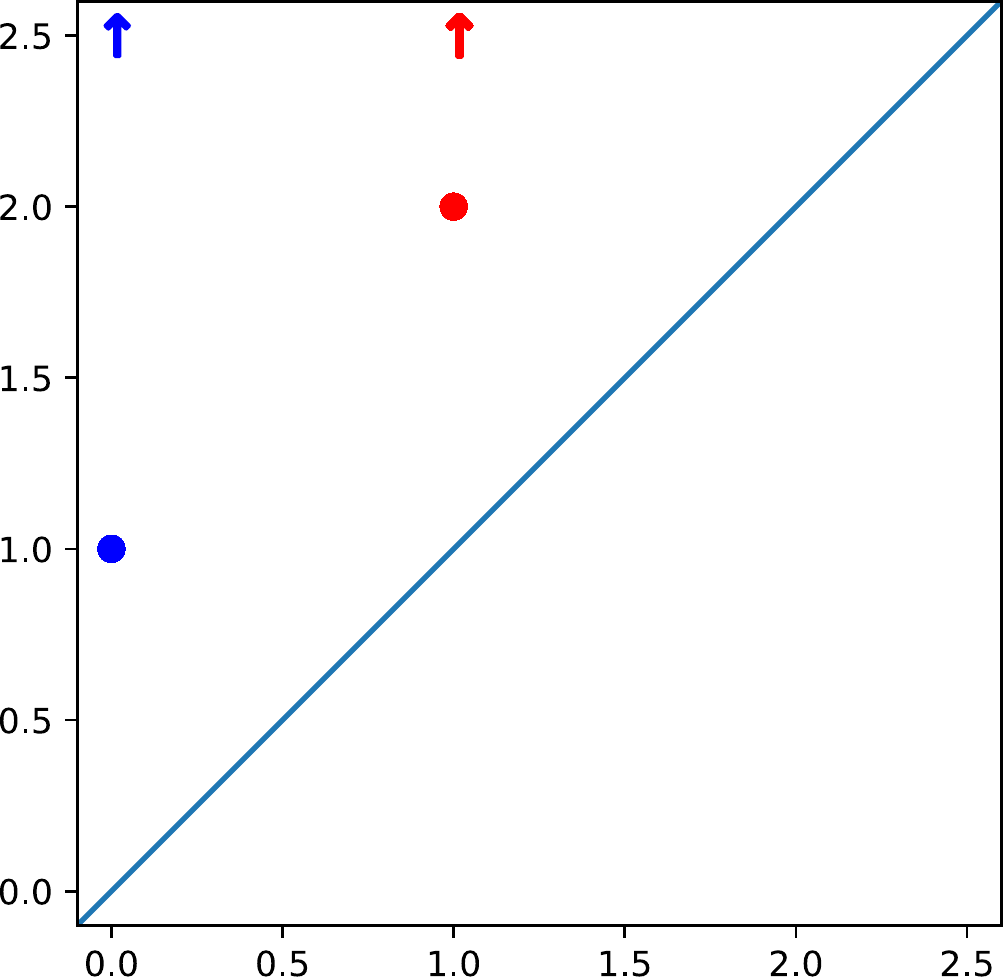}
\label{fig:karate_clique}}
\hspace{.1cm}
\subfigure[]{\includegraphics[height=2.5cm]{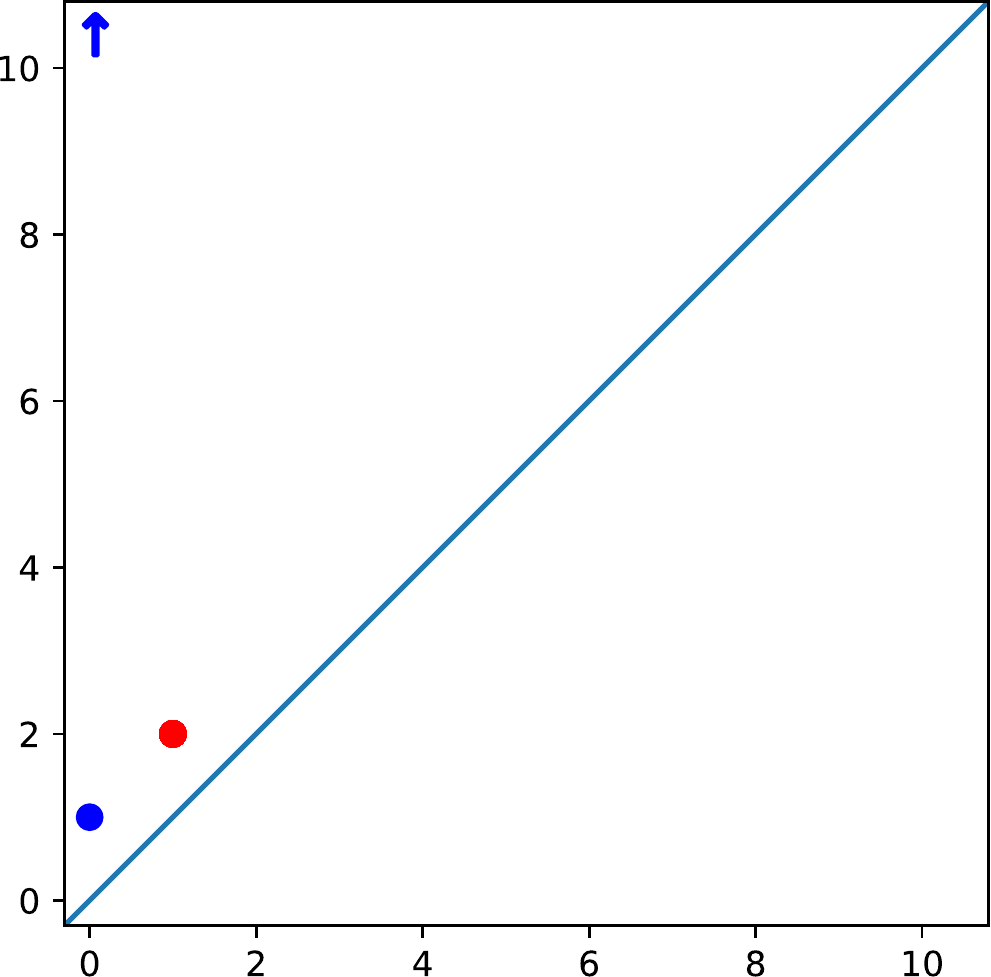}
\label{fig:karate_rips}}
\\
\subfigure[]{\includegraphics[height=2.5cm]{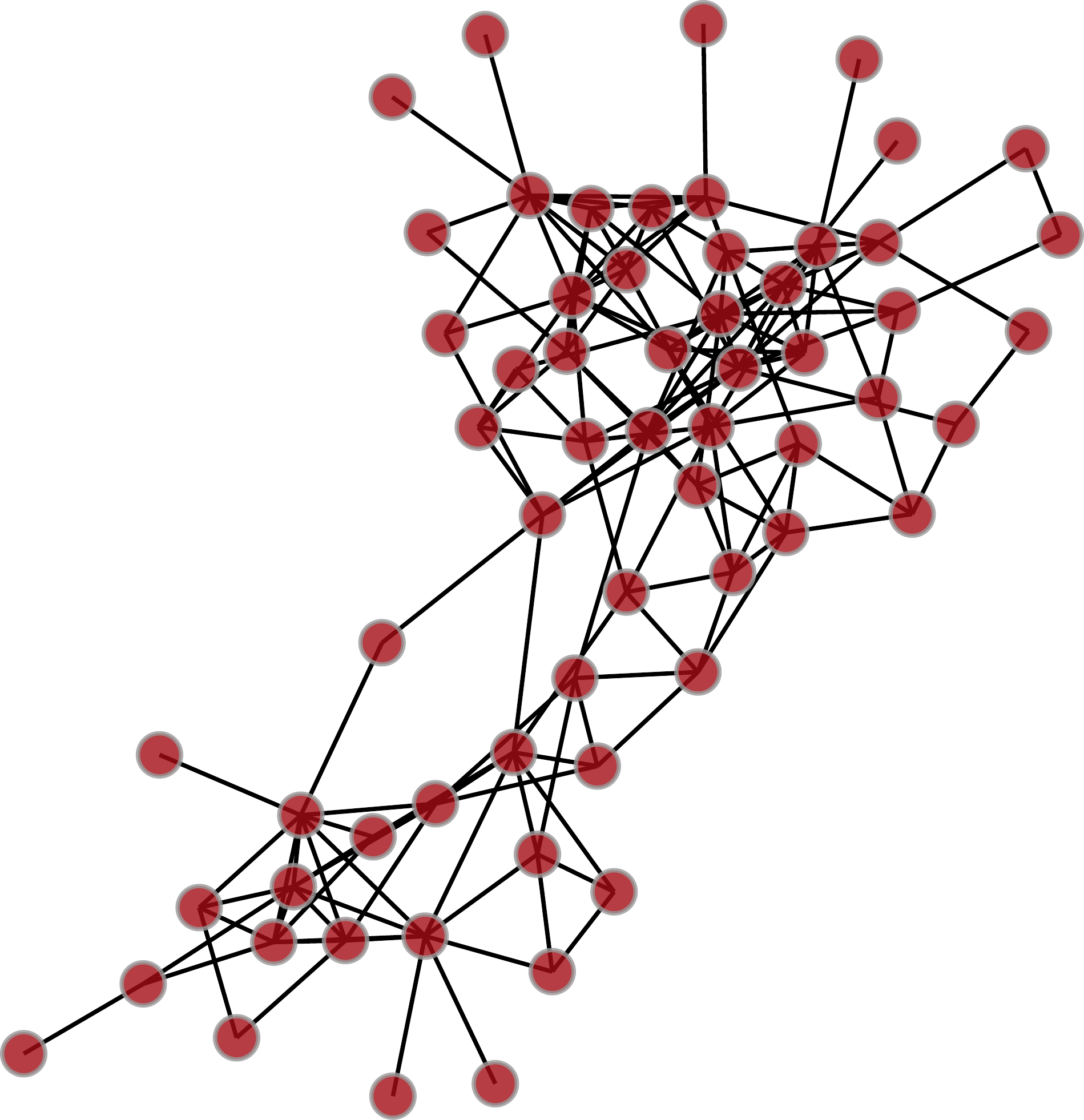}
\label{fig:dolphins_network}}
\hspace{.1cm}
\subfigure[]{\includegraphics[height=2.5cm]{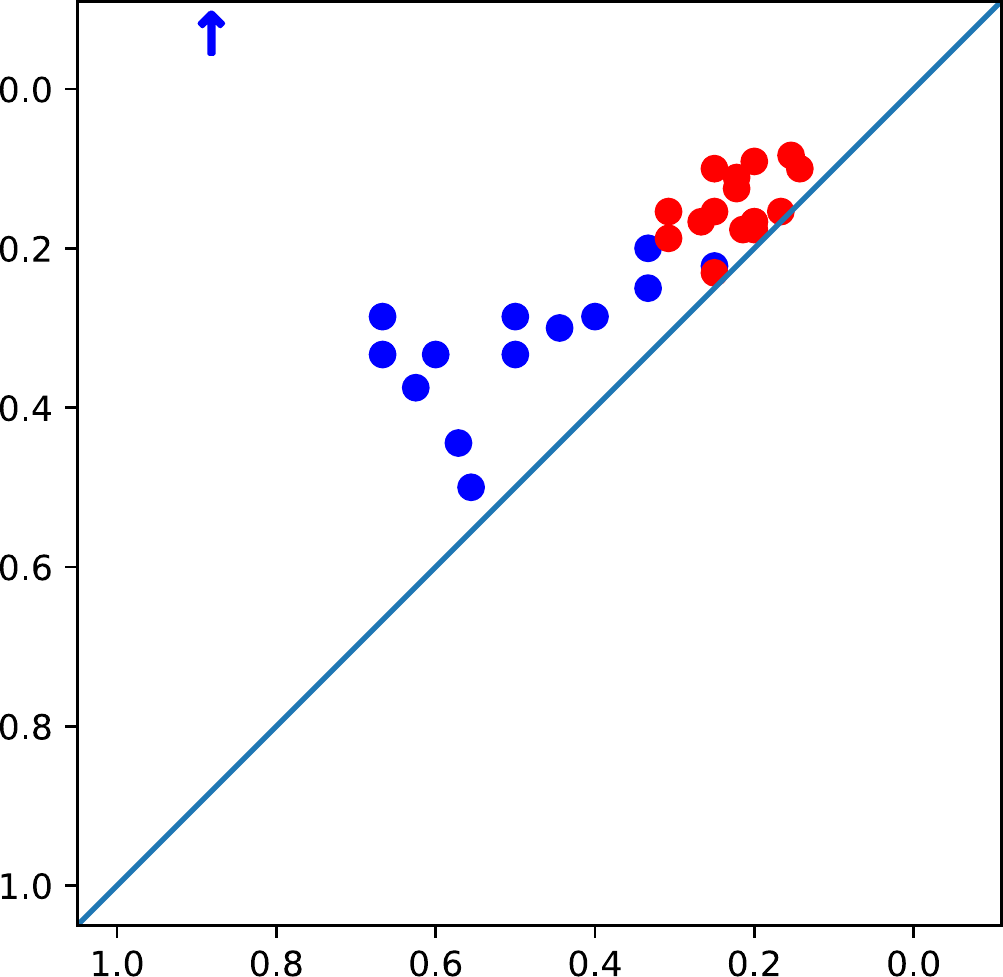}
\label{fig:dolphins_cliqueness}}
\hspace{.1cm}
\subfigure[]{\includegraphics[height=2.5cm]{images/dolphins_clique}
\label{fig:dolphins_clique}}
\hspace{.1cm}
\subfigure[]{\includegraphics[height=2.5cm]{images/dolphins_rips}
\label{fig:dolphins_rips}}
\\
\subfigure[]{\includegraphics[height=2.3cm]{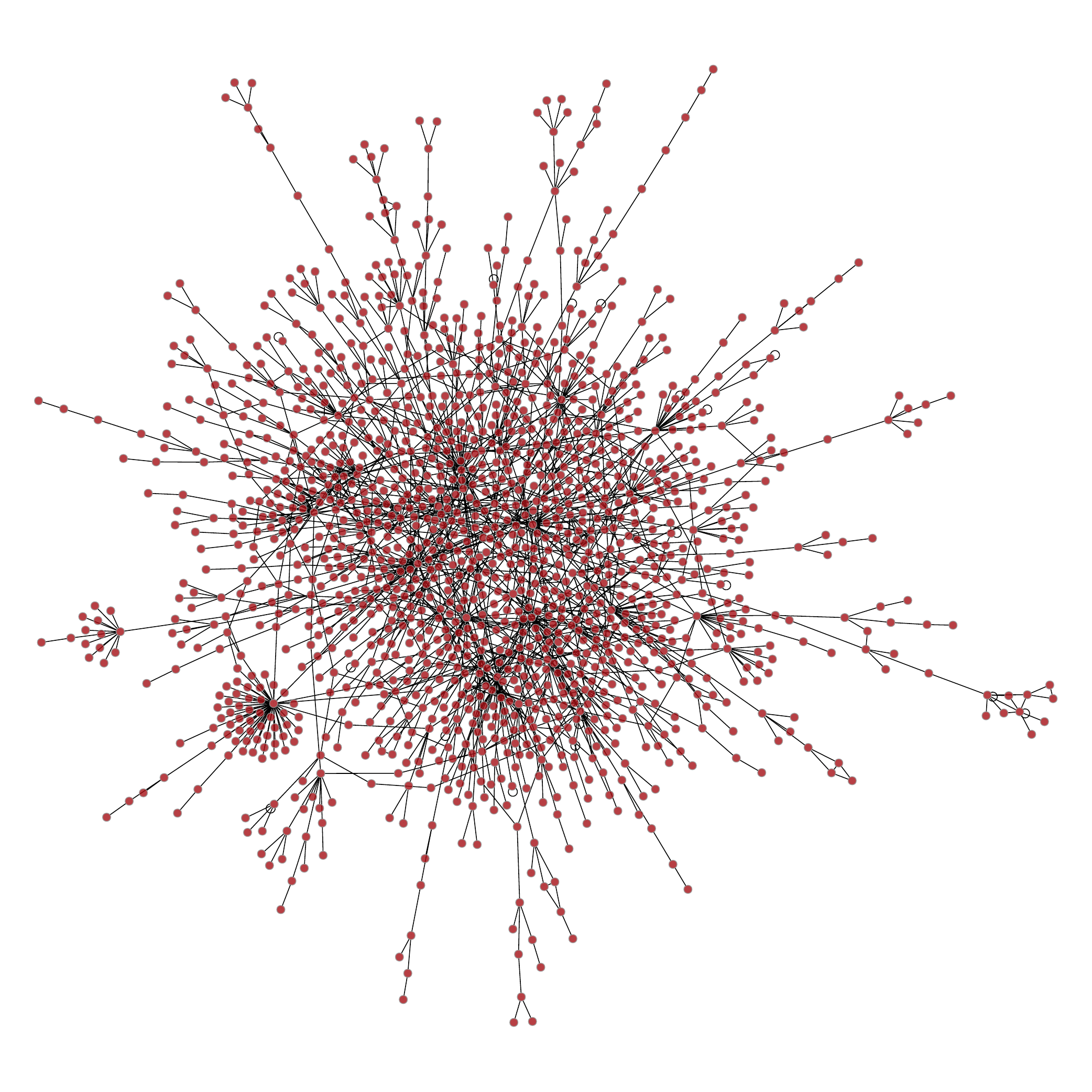}
\label{fig:protein_network}}
\hspace{.1cm}
\subfigure[]{\includegraphics[height=2.5cm]{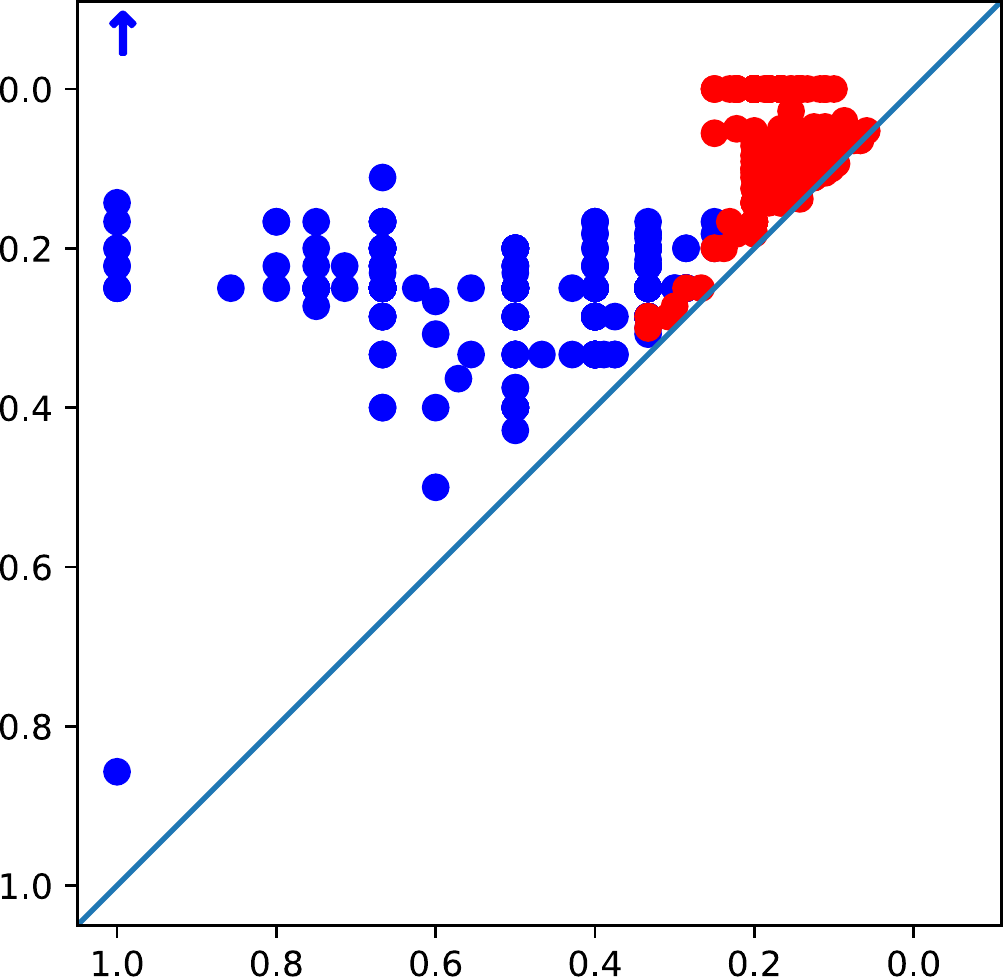}
\label{fig:protein_cliqueness}}
\hspace{.1cm}
\subfigure[]{\includegraphics[height=2.5cm]{images/dolphins_clique}
\label{fig:protein_clique}}
\hspace{.1cm}
\subfigure[]{\includegraphics[height=2.5cm]{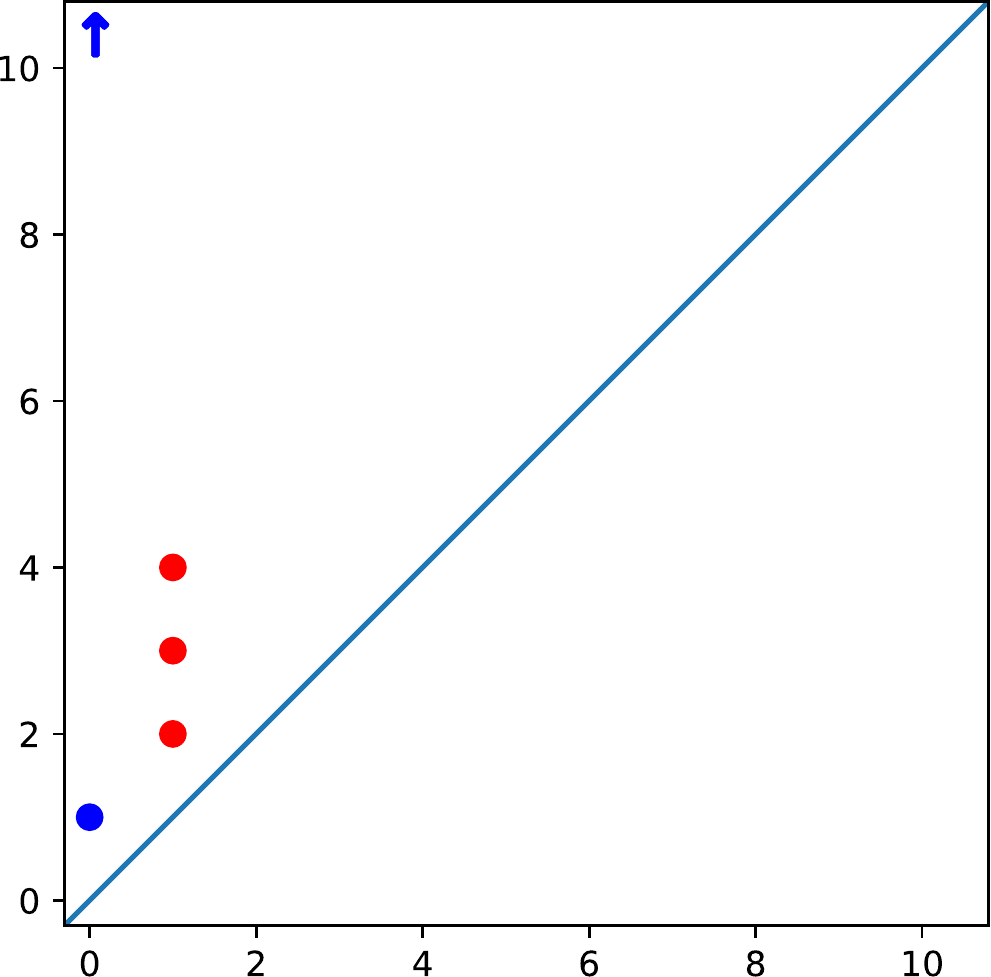}
\label{fig:protein_rips}}
\caption{The graphs in (a), (b) and (c) are the Zachary karate club social network, a social network of bottlenose dolphins and a protein interaction network respectively. The $0$- and $1$-dimensional persistent diagrams computed using the proposed method are displayed in (b), (f) and (j) respectively. The $0$- and $1$-dimensional persistent diagrams computed using a \textit{clique complex filtration} are displayed in (c), (g) and (k) respectively. The $0$- and $1$-dimensional persistent diagrams computed using a \textit{power complex filtration} are displayed in (d), (h) and (l) respectively.}
\label{fig:dolphins}
\end{center}
\end{figure}

\section{Conclusions}
\label{sec:conclusions}
This article proposes a novel method for topological graph analysis which is based on \textit{persistent homology}. This method possesses the properties of being stable and performing accurate discrimination and therefore can make accurate inferences regarding the topological features of a given graph. On the other hand, we find that existing topological graph analysis methods considered in this work do not possess these properties making it difficult for them to make such inferences.

The experimental evaluation presented in this article only considers a handful of random and real graphs. As described in the introduction to this article, topological graph analysis has many application domains including social network analysis \cite{carstens2013persistent} and computational neuroscience \cite{giusti2016two, chung2019exact}. Given the benefits of the proposed method relative to existing methods, there exists significant scope and potential for future research in the application of the proposed method to these and other domains.

\end{document}